%% file: main.tex
\documentclass[a4paper,11pt]{amsart}

\usepackage{amsthm}
\usepackage[arrow,curve,matrix]{xy}
\usepackage{amsmath}  
\usepackage{amssymb}
\usepackage{graphicx}
\usepackage{bbm}
\usepackage{mathrsfs}
\usepackage{amscd}
\usepackage[utf8x]{inputenc}

\usepackage{a4wide}

\allowdisplaybreaks

\DeclareMathOperator{\Stab}{Stab}
\DeclareMathOperator{\Spec}{Spec}
\DeclareMathOperator{\pr}{pr}
\DeclareMathOperator{\id}{id}
\DeclareMathOperator{\sgn}{sgn}
\DeclareMathOperator{\Hom}{Hom}

\DeclareMathOperator{\sHom}{\mathcal{H}\textit{om}}
\DeclareMathOperator{\Ext}{Ext}

\DeclareMathOperator{\Coh}{Coh}
\DeclareMathOperator{\QCoh}{QCoh}
\DeclareMathOperator{\sExt}{\mathcal{E}\textit{xt}}
\DeclareMathOperator{\D}{D}
\DeclareMathOperator{\Ho}{H}

\DeclareMathOperator{\Inf}{Inf}

\DeclareMathOperator{\res}{res}

\DeclareMathOperator{\glob}{\Gamma}

\DeclareMathOperator{\codim}{codim}

\DeclareMathOperator{\Pic}{Pic}

\DeclareMathOperator{\im}{im}
\DeclareMathOperator{\Map}{Map}
\DeclareMathOperator{\K}{K}

\DeclareMathOperator{\GVS}{GVS}
\DeclareMathOperator{\SVS}{SVS}
\DeclareMathOperator{\Ve}{Vec}

\newcommand{\sym}{\mathfrak S}
\newcommand{\C}{\mathbbm C}
\newcommand{\N}{\mathbbm N}

\newcommand{\Z}{\mathbbm Z}

\newcommand{\F}{\mathcal F}
\newcommand{\E}{\mathcal E}

\newcommand{\Fb}{\mathcal F^\bullet}
\newcommand{\Eb}{\mathcal E^\bullet}
\newcommand{\Gb}{\mathcal G^\bullet}

\newcommand{\dv}{\mathbbm v}

\newcommand{\reg}{\mathcal O}
\newcommand{\I}{\mathcal I}
\newcommand{\n}{[n]}
\newcommand{\eps}{\varepsilon}
\renewcommand{\theta}{\vartheta}
\renewcommand{\rho}{\varrho}
\renewcommand{\phi}{\varphi}
\renewcommand{\_}{\underline{\,\,\,\,}}

\newtheorem{theorem}{Theorem}[section]
\newtheorem{prop}[theorem]{Proposition}
\newtheorem{lemma}[theorem]{Lemma}
\newtheorem{cor}[theorem]{Corollary}

\theoremstyle{definition}
\newtheorem{defin}[theorem]{Definition}

\newtheorem{remark}[theorem]{Remark}

\begin{document}
\title[Tensor products of tautological bundles under the BKRH equivalence]{Tensor products of tautological bundles under the Bridgeland--King--Reid--Haiman equivalence}
\author{Andreas Krug}
\address{Universit\" at Bonn}
\email{akrug@math.uni-bonn.de}
\begin{abstract}
We give a description of the image of tensor products of tautological bundles on Hilbert schemes of points on surfaces under the Bridgeland--King--Reid--Haiman equivalence. Using this, some new formulas for cohomological invariants of these bundles are obtained. In particular, we give formulas for the Euler characteristic of arbitrary tensor products on the Hilbert scheme of two points and of triple tensor products in general.  
\end{abstract}
\maketitle


\input{intro}



\input{notation}
\input{bkr}

\input{image}

\input{orbits}

\input{globalsec}
\input{longtwo}

\input{longopen}

\input{invatwo}

\input{bichar}

\input{threetens}

\input{deterline}

\begin{appendix}
\input{binomial}

\input{graded}

\input{Danila}
\input{normal}
\input{restriction}
\end{appendix}

\bibliographystyle{alpha}
\addcontentsline{toc}{chapter}{References}
\bibliography{references}
\end{document}

%% file: intro.tex
%
\section{Introduction}
For every smooth quasi-projective surface $X$ over $\C$ there is a series of associated higher dimensional smooth varieties namely the \textit{Hilbert schemes of $n$ points on $X$} for $n\in\N$. They are the fine moduli spaces $X^{[n]}$ of zero dimensional subschemes of length $n$ of $X$.
Thus, there is a universal family $\Xi$ together  with its projections
\[X\xleftarrow{\pr_X}\Xi\xrightarrow{\pr_{X^{[n]}}} X^{[n]}\,.\]
Using this, one can associate to every coherent sheaf $F$ on $X$ the so called \textit{tautological sheaf} $F^{[n]}$ on each $X^{[n]}$ given by  
\[F^{[n]}:=\pr_{X^{[n]}*}\pr_X^* F\,.\]
It is well known (see \cite{Fog}) that the Hilbert scheme $X^{[n]}$ of $n$ points on $X$ is a resolution of the singularities of $S^nX=X^n/\sym_n$ via the \textit{Hilbert--Chow morphism}
\[\mu\colon X^{[n]}\to S^nX\quad,\quad \xi\mapsto\sum_{x\in\xi}\ell(\xi,x)\cdot x\,.\]
For every line bundle $L$ on $X$ the line bundle $L^{\boxtimes n}\in \Pic(X^n)$ descends to the line bundle $(L^{\boxtimes n})^{\sym_n}$ on $S^nX$. Thus for every $L\in\Pic(X)$ there is the \textit{determinant line bundle} on $X^{[n]}$ given by
\[\mathcal D_L:=\mu^*((L^{\boxtimes n})^{\sym_n})\,.\]
One goal in studying Hilbert schemes of points is to find formulas expressing the invariants of $X^{[n]}$ in terms of the invariants of the surface $X$. This includes the invariants of the induced sheaves defined above. There are already some results in this area.
For example, in \cite{Leh} there is a formula for the Chern classes of $F^{[n]}$ in terms of those of $F$ in the case that $F$ is a line bundle. In \cite{BNW} the existence of
universal formulas, i.e. formulas independent of the surface $X$, expressing the characteristic classes of any tautological sheaf in terms of the characteristic classes of $F$ is shown and those formulas are computed  in some cases.
Furthermore Danila (\cite{Dan}, \cite{Danglob}, \cite{Dandual}) and Scala (\cite{Sca1}, \cite{Sca2}) proved formulas for the cohomology of tautological sheaves, determinant line bundles, and some natural constructions (tensor, wedge, and symmetric products) of these. In particular, in \cite{Sca1} there is the formula
\begin{align} \label{Scaformula}\Ho^*(X^{[n]},F^{[n]}\otimes \mathcal D_L)\cong \Ho^*(F\otimes L)\otimes S^{n-1}\Ho^*(L)\end{align}
for the cohomology of a tautological sheaf twisted by a determinant line bundle.
We will use and further develop Scala's approach of \cite{Sca1} and \cite{Sca2} which in turn uses the \textit{Bridgeland--King--Reid--Haiman equivalence} (see \cite{BKR} and \cite{Hai}). It is an equivalence
\[\Phi:=\Phi_{\reg_{I^nX}}^{X^{[n]}\to X^n}\colon \D^b(X^{[n]})\xrightarrow\simeq \D^b_{\sym_n}(X^n)\]
between the bounded derived category of $X^{[n]}$ and the bounded derived category of $\sym_n$-equivariant sheaves on $X^n$
(for details about equivariant derived categories see e.g. \cite{BL} or \cite[Section 3]{Kru}). 
The equivalence is given by the Fourier--Mukai transform with kernel the structural sheaf of the \textit{isospectral Hilbert scheme} $I^nX:=(X^{[n]}\times_{S^nX} X^n)_{\text{red}}$.
It induces natural isomorphisms 
\[\Ext^*_{X^{[n]}}(\Eb,\Fb)\cong \sym_n\Ext^*_{X^n}(\Phi(\Eb),\Phi(\Fb))=\Ext^*_{X^n}(\Phi(\Eb),\Phi(\Fb))^{\sym_n}\,.\]
Furthermore, there is a natural isomorphism $R\mu_*\F^\bullet\simeq \Phi(\F^\bullet)^{\sym_n}$ (see \cite{Sca1}) which yields
\begin{align}\label{mucoh}\Ho^*(X^{[n]},\F^\bullet)\cong \Ho^*(S^nX,\Phi(\F^\bullet)^{\sym_n})\cong \Ho^*(X^n,\Phi(\F^\bullet))^{\sym_n}\,.\end{align}
So instead of computing the cohomology and extension groups of constructions of tautological sheaves on $X^{[n]}$ directly, the approach is to compute them for the image of these sheaves under the Bridgeland--King--Reid--Haiman equivalence.  
In order to do this we need a good description of $\Phi(F^{[n]})\in \D^b_{\sym_n}(X^n)$ for $F^{[n]}$ a tautological sheaf. This was provided by Scala in \cite{Sca1} and \cite{Sca2}. He showed that $\Phi(F^{[n]})$ is always concentrated in degree zero. This means that we can replace $\Phi$ by its non-derived version $p_*q^*$ where $p$ and $q$ are the projections from $I^nX$ to $X^n$ and $X^{[n]}$ respectively, i.e. we have $\Phi(F^{[n]})\simeq p_*q^*(F^{[n]})$. Moreover, he gave for $p_*q^*(F^{[n]})$ a right resolution $C^\bullet_F$. This is a $\sym_n$-equivariant complex associated to $F$ concentrated in non-negative degrees whose terms are \textit{good} sheaves.
For us a good $\sym_n$-equivariant sheaf on $X^n$ is a sheaf which is constructed out of sheaves on the surface $X$ in a not too complicated way.
In particular, it should be possible to give a formula for its ($\sym_n$-invariant) cohomology in terms of the cohomology of sheaves on $X$.
For example the degree zero term of the complex $C^\bullet_F$ is
$C^0_F=\bigoplus_{i=1}^n \pr_i^* F$. Note that if $F$ is locally free $C^0_F$ is, too. Its cohomology is by the K\"unneth formula given by
\[\Ho^*(X^n,C^0_F)=\left( \Ho^*(F)\otimes \Ho^*(\reg_X)^{\otimes n-1} \right)^{\oplus n}\,.\]
The $\sym_n$-invariants of the cohomology can be computed as
\[\Ho^*(X^n,C^0_F)^{\sym_n}=\Ho^*(F)\otimes S^{n-1}\Ho^*(\reg_X)\]
(For the proof of Scala's formula (\ref{Scaformula}) in the case $L=\reg_X$ it only remains to show that the invariants of $C^p_F$ for $p\ge 1$ vanish).
Let now $E_1,\dots,E_k$ be locally free sheaves on $X$. The associated tautological sheaves $E_i^{[n]}$ on $X^{[n]}$ are again locally free and hence called \textit{tautological bundles}. In \cite{Sca1} it is shown that again
\[\Phi(E_1^{[n]}\otimes \dots\otimes E_k^{[n]})\simeq p_*q^*(E_1^{[n]}\otimes \dots\otimes E_k^{[n]})\,.\]
Furthermore, a description of $p_*q^*(\otimes_i E_i^{[n]})$ as the $E^{0,0}_\infty$ term of a certain spectral sequence is given. We will give a more concrete description of $p_*q^*(E_1^{[n]}\otimes\dots\otimes E_k^{[n]})$ as a subsheaf of $K_0:=C^0_{E_1}\otimes \dots\otimes C^0_{E_k}$ as follows: We construct successively $\sym_n$-equivariant morphisms
$\phi_\ell\colon K_{\ell-1}\to T_\ell$ for $\ell=1,\dots,k$, where $K_\ell:=\ker{\phi_\ell}$ and the $T_\ell$ are good sheaves given by
\[
 T_\ell=\bigoplus_{(M;i,j;a)}\left( S^{\ell-1}\Omega_X\otimes \bigotimes_{t\in M} E_t\right)_{ij}\otimes \bigotimes_{t\in[k]\setminus M}\pr_{a(t)}^*E_t\,.
\]
The sum is taken over all tuples $(M;i,j;a)$ with $M\subset [k]:=\{1,\dots,k\}$, $|M|=\ell$, $i,j\in [n]$, $i\neq j$, and $a\colon [k]\setminus M\to [n]$.
The  functor $(\_)_{ij}$ is the composition $\iota_{ij*}p_{ij}^*$, where $\iota_{ij}\colon \Delta_{ij}\to X^n$ is the inclusion of the pairwise diagonal and
$p_{ij}\colon \Delta_{ij}\to X$ is the restriction of the projection $\pr_i\colon X^n\to X$ respectively of $\pr_j\colon X^n\to X$.
Then we show that $K_k=p_*q^*(\otimes_{i=1}^kE_i^{[n]})$. If the exact sequences
\[0\to K_{\ell}\to K_{\ell-1}\to T_\ell\]
for $\ell=1,\dots,k$ were also exact with a zero on the right, this result would yield directly a description of the cohomology of $E_1^{[n]}\otimes \dots\otimes E_k^{[n]}$
via long exact sequences and an explicit formula for its Euler characteristic. Since this is not the case, we have to enlarge the sequences to exact sequences with a zero on the right or at least do the same with the sequences
\[0\to K_{\ell}^{\sym_n}\to K_{\ell-1}^{\sym_n}\to T_\ell^{\sym_n}\]
of invariants on $S^nX$.
The latter also yields the cohomological invariants since by (\ref{mucoh}) we have $\Ho^*(X^{[n]},\otimes_i E_i^{[n]})\cong \Ho^*(S^nX, K_k^{\sym_n})$.
We are able to get the following results:
\begin{itemize}
\item A formula for the cohomology of tensor products of tautological sheaves in the maximal cohomological degree $2n$. 
 \item For $E_1,\dots,E_k$ locally free sheaves on a projective surface $X$ and $k\le n$ the  formula
\[\Ho^0(X^{[n]},E_1^{[n]}\otimes \dots\otimes E_k^{[n]})\cong \Ho^0(E_1)\otimes\dots \otimes \Ho^0(E_k)\,.\]
\item For $E_1,\dots,E_k$ locally free sheaves on $X$ and arbitrary $k$ long exact sequences
\[0\to K_\ell\to K_{\ell-1}\to T_\ell\to T_\ell^1\to\dots\to T_\ell^{k-\ell}\to 0\]
on $X^2$ with good sheaves $T_\ell^i$. This yields a description via long exact sequences of the cohomology and an explicit formula for the Euler characteristic of $E_1^{[2]}\otimes\dots\otimes E_k^{[2]}$ and more general the Euler bicharacteristics between two different tensor products of tautological sheaves on $X^{[2]}$.
\item Similar long exact sequences for arbitrary $n$ over the open subset $X^n_{**}$ consisting of points $(x_1,\dots,x_n)$ where at most two $x_i$ coincide.
\item For $E_1,E_2,E_3$ locally free sheaves on $X$ long exact sequences on $S^nX$ whose kernels converge to $K_3^{\sym_n}$.  
This yields a description via long exact sequences of the cohomology and an explicit formula for the Euler characteristic of $E_1^{[n]}\otimes E_2^{[n]}\otimes E_3^{[n]}$.
\end{itemize}
Using that for every $\F^\bullet\in \D^b(X^{[n]})$ and $L$ a line bundle on $X$ we have
\[\Phi(\F^\bullet\otimes \mathcal D_L)\simeq \Phi(\F^\bullet)\otimes L^{\boxtimes n}\,,\]
we can generalise the results from products of tautological bundles to products of tautological bundles twisted by a determinant line bundle by simply tensoring the exact sequences on $X^n$ by $L^{\boxtimes n}$ respectively tensoring the exact sequences on $S^nX$ by $(L^{\boxtimes n})^{\sym_n}$.
Also, we can generalise the results on the Euler characteristics form tautological bundles to arbitrary tautological objects. We will not state the final results in this introduction since they are presented in a compact form in subsection \ref{gen}.  
\vspace{0.5cm}

\textbf{Acknowledgements}: Most of the content of this article is also part of the authors PhD thesis. The author wants to thank his adviser Marc Nieper-Wi\ss kirchen for his support. The article was finished during the authors stay at the SFB Transregio 45 in Bonn.

%% file: notation.tex
%
\section{Notations and conventions}\label{not}
\begin{enumerate}
\item On a scheme $X$ the sheaf-Hom functor is denoted by $\sHom_{\reg_X}$, $\sHom_X$ or just $\sHom$.
We write $(\_)^\vee=\sHom(\_,\reg_X)$ for the operation of taking the dual of a sheaf and $(\_)^\dv=R\sHom(\_,\reg_X)$ for the derived dual.
\item\label{v} For a finite vector space $V$ we will write $v:=\dim V$.
\item An empty tensor, wedge or symmetric product of sheaves on $X$ is the 
sheaf $\reg_X$.
\item Given a product of schemes $X=\prod_{i\in I} X_i$ we will denote by $\pr_i$, $\pr_{X_i}$, $p_i$, or $p_{X_i}$ the projection $X\to X_i$. 
\item In formulas with enumerations putting the sign $\,\hat{}\,$ over an element means that this element is omitted. For example $\{1,\dots,\hat 3,\dots,5\}$ denotes the  set $\{1,2,4,5\}$. 
\item For a local section $s$ of a sheaf $\F$ we will often write $s\in \F$.
\item For a direct sum $V=\oplus_{i\in I} V_i$ of vector spaces or sheaves we will write interchangeably $V_i$ and  $V(i)$ for the summands. For an element respectively local section $s\in V$ we will write $s(i)$ or $s_i$ for its component in $V(i)$. 
We denote the components of a morphism $\psi\colon Z\to V$ by $\psi(i)\colon Z\to V(i)$. 
Let $W=\oplus_{j\in J} W_j$ be an other direct sum and $\phi\colon V\to W$ a morphism. We will denote the components of $\phi$ by \[\phi(i\to j)\phi(i,j)\colon V(i)\to W(j)\,.\] 
\item Let $\iota\colon Z\to X$ be a closed embedding of schemes and let $F\in \QCoh(X)$ be a quasi-coherent sheaf on $X$. The symbol $F_{\mid Z}$ will sometimes denote the sheaf $\iota^*F\in \QCoh(Z)$ and at other times the sheaf $\iota_*\iota^*F\in\QCoh(X)$. The restriction morphism \[F\to F_{\mid Z}=\iota_*\iota^*F\] is the unit of the adjunction $(\iota^*, \iota_*)$. The image of a section $s\in F$ under this morphism is denoted by $s_{\mid Z}$. 
\item For any finite set $M$ the symmetric group $\sym_M$ is the group of bijections of $M$. 
Note that we have $\sym_{\emptyset}\cong 1$. 
For two positive integers $n<m$ we use the notation
\[[n]=\{1,2,\dots,n\}\quad,\quad  [n,m]=\{n,n+1,\dots,m\}\,.\]
If $n>m$ we set $[n,m]:=\emptyset$. We interpret the symmetric group $\sym_n$  as the group acting on $[n]$, i.e.
$\sym_n=\sym_{[n]}$.
For any subset $I\subset \n$ we denote by $\bar I=\n\setminus I$ its complement in $\n$.
For better readability we sometimes write
$\overline{\sym_I}$ instead of $\sym_{\bar I}$.
\item We will write multi-indices mostly in the form of maps, i.e. for two positive integers $n,k\in \N$ we denote multi-indices with $k$ values between $1$ and $n$ rather as 
elements of $\Map([k],[n])$ than as elements of $[n]^k$. But sometimes we will switch between the notations and write a multi-index $a\colon [k]\to [n]$ in the form $a=(a(1),\dots,a(k))$ or $a=(a_1,\dots,a_k)$. For two maps $a\colon M\to K$ and $b\colon N\to K$ with disjoint domains we write $a\uplus b\colon M\coprod N\to K$ for the induced map on the union. If $N=\{i\}$ consists of only one element we will also write $a\uplus b=(a,i\mapsto b(i))$. For $x\in K$ we write 
$\underline x\colon M\to K$ for the map which is constantly $x$.  
For a multi-index $a\colon M\to \{i<j\}$ with a totally ordered codomain consisting of two elements we introduce the sign
\[\eps_a:=(-1)^{\#a^{-1}(\{j\})}\,.\]
For the preimage sets of one element $i$ in the codomain of $a$ we will often write  for short $a^{-1}(i)$ instead of $a^{-1}(\{i\})$.
\item Putting the symbol PF over an isomorphism sign means that the isomorphism is given by projection formula.
\end{enumerate}

%% file: bkr.tex
\section{Image of tautological sheaves under the Bridgeland--King--Reid--Haiman equivalence}\label{bkr}
\subsection{The Bridgeland--King--Reid--Haiman equivalence}
From now on let $X$ be a smooth quasi-projective surface. For $n\in \N$ we denote by $X^{[n]}$ the \textit{Hilbert scheme of $n$ points on the surface $X$} which is the fine moduli space of zero-dimensional subschemes $\xi$ of $X$ of length
$\ell(\xi):=h^0(\xi,\reg_\xi)=n$. 
The \textit{isospectral Hilbert scheme} is given by $I^nX:=(X^{[n]}\times_{S^nX}X^n)_{\text{red}}$. Here, the fibre product is defined via the Hilbert-Chow morphism $\mu\colon X^{[n]}\to S^nX$ and the quotient morphism $\pi\colon X^n\to S^nX=X^n/\sym_n$. 
We will use the Bridgeland--King--Reid--Haiman equivalence in order to compute the cohomology of certain sheaves on $X^{[n]}$.  
\begin{theorem}[\cite{BKR}, \cite{Hai}]\label{equi}
The Fourier-Mukai Transform
 \[\Phi:=\Phi^{X^{[n]}\to X^n}_{\reg_{I^nX}}\colon \D^b(X^{[n]})\to \D^b_{\mathfrak S_n}(X^n)\]
is an equivalence of triangulated categories. 
\end{theorem}
The Bridgeland--King--Reid equivalence can be used to compute the extension groups of objects in the derived category of the Hilbert scheme.
\begin{cor}\label{bkrext}
 Let $\Fb,\Gb\in\D^b(X^{[n]})$. Then 
\[\Ext^i_{X^{[n]}}(\Fb,\Gb)\cong \sym_n\Ext^i_{X^n}(\Phi(\Fb),\Phi(\Gb))\quad\text{for all $i\in\Z$}\,.\] 
\end{cor}
\begin{proof}
Using the last corollary we indeed have
\[ 
\begin{aligned} \Ext^i_{X^{[n]}}(\Fb,\Gb)\cong \Hom_{\D^b(X^{[n]})}(\Fb,\Gb[i])
 &\cong \Hom_{\D^b_{\sym_n}(X^n)}(\Phi(\Fb),\Phi(\Gb[i]))\\
&\cong \Hom_{\D^b_{\sym_n}(X^n)}(\Phi(\Fb),\Phi(\Gb)[i])\\
&\cong \sym_n\Ext^i_{X^n}(\Phi(\Fb),\Phi(\Gb))\,.
\end{aligned}
\]
\end{proof}
We will abbreviate the functor $[\_]^{\sym_n}\circ\pi_*\colon \D^b_{\sym_n}(X^n)\to\D^b(S^nX)$ by $[\_]^{\sym_n}$. Note that $\pi_*$ indeed does not need to be derived since $\pi$ is finite.
\begin{prop}[\cite{Sca1}]\label{mubkr}
For every $\F^\bullet\in \D^b(X^{[n]})$ there is a natural isomorphism 
\[R\mu_*\F^\bullet\simeq [\Phi(\F^\bullet)]^{\sym_n}\,.\]
Furthermore there are natural isomorphisms
\[\Ho^*(X^{[n]},\F^\bullet)\cong \Ho^*(S^nX,\Phi(\F^\bullet)^{\sym_n})\cong \Ho^*(X^n,\Phi(F^\bullet))^{\sym_n}\,.\]
\end{prop}
\begin{proof}
 The morphisms $q$ is the $\sym_n$-quotient morphism (see \cite[section 1.5]{Sca1}). Thus, $(q_*\reg_{I^nX})^{\sym_n}=\reg_{X^{[n]}}$. Using this, we get 
\begin{align*}
 \Phi(\F^\bullet)^{\sym_n}\simeq[\pi_*\circ Rp_*\circ q^* \F^\bullet]^{\sym_n}\simeq [R\mu_*\circ q_*\circ q^* \F^\bullet]^{\sym_n}&\simeq R\mu_*[q_*\circ q^* \F^\bullet]^{\sym_n}\\&\overset{PF}\simeq R\mu_* [\F^\bullet\otimes^L q_*\reg_{I^nX}]^{\sym_n}\\&\overset{\ref{tensinv}}\simeq R\mu_* (\F^\bullet\otimes^L [q_*\reg_{I^nX}]^{\sym_n})\\&\simeq R\mu_* \F^\bullet\,.
\end{align*}
Now indeed $\Ho^*(X^{[n]},\F^\bullet)\cong \Ho^*(S^nX,R\mu_*\F^\bullet)\cong \Ho^*(S^nX, \Phi(\F^\bullet)^{\sym_n})$.
The last isomorphisms of the second assertion is due to the fact that $[\_]^{\sym_n}$ is an exact functor. 
\end{proof}
\subsection{Tautological sheaves}
\begin{defin}
We define the \textit{tautological functor for sheaves} as
\[(\_)^{[n]}:=\pr_{X^{[n]}*}(\reg_\Xi\otimes \pr_X^*(\_))\colon\Coh(X)\to\Coh(X^{[n]})\,.\]
For a sheaf $F\in\Coh(X)$ we call its image $F^{[n]}$ under this functor the \textit{tautological sheaf associated with $F$}. 
In \cite[Proposition 2.3]{Sca2} it is shown that the functor $(\_)^{[n]}$ is exact. Thus, it induces the \textit{tautological functor for objects}
$(\_)^{[n]}\colon \D^b(X)\to \D^b(X^{[n]})$.  
For an object $F^\bullet\in\D^b(X)$ the \textit{tautological object associated to
$F^\bullet$} is $(F^\bullet)^{[n]}$.
\end{defin}
\begin{remark}\label{tautexact}
The tautological functor for objects is isomorphic to the Fourier-Mukai transform with kernel the structural sheaf of the universal family, i.e.
$(F^\bullet)^{[n]}\simeq \Phi^{X\to X^{[n]}}_{\reg_\Xi}(F^\bullet)$ for every $F^\bullet\in \D^b(X)$.
\end{remark}
\begin{remark}
If $F$ is locally free of rank $k$ the \textit{tautological bundle} $F^{[n]}$ is
locally free of rank $k\cdot n$ with fibers
$F^{[n]}([\xi])=\glob(\xi,F_{\mid \xi})$ since $\pr_{X^{[n]}}\colon \Xi\to X^{[n]}$ is flat and finite of degree $n$.
\end{remark}             
\subsection{The complex $C^\bullet$}\label{ccomplex}
We use the notation from appendix \ref{combi}. 
To any coherent sheaf $F$ on $X$ we associate a $\sym_n$-equivariant
complex $C_F^\bullet$ of sheaves on $X^n$ as follows.
We set 
\[C^0_F=\bigoplus_{i=1}^n
p_i^*F\quad, \quad C^p_F=\bigoplus_{I\subset[n]\,,\,\mid I\mid =p+1} F_I\quad\text{for
$0< p< n$}\quad,\quad C^p_F=0 \quad\text{else}\,.\]
Let $s=(s_I)_{\mid I\mid= p+1}$ be a local section of $C^p_F$. We define the $\sym_n$-linearization of $C^p_F$ by
\[\lambda_{\sigma}(s)_I:=\eps_{\sigma,I}\cdot\sigma_*(s_{\sigma^{-1}(I)})\,,\]
where $\sigma_*$ is the flat base change isomorphism from the following diagram with $p_I\circ \sigma =p_{\sigma^{-1}(I)}$
\[\xymatrix{
\Delta_{\sigma^{-1}(I)} \ar^\sigma[r] \ar_{\iota_{\sigma^{-1}(I)}}[d]  &   \Delta_I\ar^{p_I}[r] \ar_{\iota_I}[d] &   X   \\
X^n    \ar^\sigma[r]         &         X^n  \,.   &            
}
\]
This gives also a $\sym_n$-linearization of $C^0_F$ using the convention $F_{\{i\}}:=p_i^*F$ and $\Delta_{\{i\}}:=X^n$.
Finally, we define the differentials $d^p\colon C^p_F\to C^{p+1}_F$ by the formula
\[d^p(s)_J:=\sum_{i\in J} \eps_{i,J}\cdot s_{J\setminus\{i\}\mid _{\Delta_J}}\,.\]
As one can check using lemma \ref{signlemma}, $C^\bullet_F$ is indeed an $\sym_n$-equivariant complex. 
\subsection{Polygraphs and the image of tautological sheaves under $\Phi$}\label{pol}
\begin{theorem}[\cite{Sca2}]\label{Sca}
\begin{enumerate}
 \item 
For every $F\in\Coh(F)$, the object $\Phi(F^{[n]})$ is cohomologically concentrated in degree zero. Furthermore, the complex $C^\bullet_F$ is a right resolution of $p_*q^*(F^{[n]})$. Hence, in $\D^b_{\sym_n}(X^n)$ there are the isomorphisms   
\[\Phi(F^{[n]})\simeq p_*q^*F^{[n]}\simeq C^\bullet_F\,.\]
\item
For every collection $E_1,\dots,E_k\in \Coh(X)$ of locally free sheaves on $X$, the object $\Phi(E_1^{[n]}\otimes \dots\otimes E_k^{[n]})$ is cohomologically concentrated in degree zero. Furthermore, there is a natural $\sym_n$-equivariant surjection 
\[\alpha\colon p_*q^*E_1^{[n]}\otimes \dots\otimes p_*q^*E_k^{[n]}\to p_*q^*(E_1^{[n]}\otimes \dots \otimes E_k^{[n]})\]
whose kernel is the torsion subsheaf. Hence, in $\D^b_{\sym_n}(X^n)$ there are the isomorphisms   
\[\Phi\left(\bigotimes_{i=1}^kE_i^{[n]}\right)\simeq p_*q^*\left(\bigotimes_{i=1}^kE_i^{[n]}\right)\simeq \bigotimes_{i=1}^kp_*q^*E_i^{[n]}/\text{torsion}\,.\]
\end{enumerate}
\end{theorem} 
We denote for $i\in [k]$ the augmentation map by $\gamma_i\colon p_*q^*(E_i^{[n]})\to C^0_{E_i}$.
\begin{prop}\label{im}
 Let $E_1,\dots,E_k$ be locally free sheaves on $X$. Then there is a natural isomorphism
\[p_*q^*(E_1^{[n]}\otimes \dots\otimes E_k^{[n]})\cong \im(\gamma_1\otimes\dots\otimes \gamma_k)\subset C^0_{E_1}\otimes \dots\otimes C^0_{E_k}\,.\]
\end{prop}
\begin{proof}
The sheaf $\otimes_i C^0_{E_i}$ is locally free and hence torsion-free. 
Thus, the kernel of $\otimes_i \gamma_i$ must contain the torsion subsheaf.
Since $C^1_{E_i}$ is supported on the big diagonal $\mathbbm D$ for every $i\in [k]$, the map $\otimes_i\gamma_i$
is an isomorphism outside of $\mathbbm D$. Thus, the kernel is supported on $\mathbbm D$ and hence torsion. In summary,  $\im(\otimes_i \gamma_i)$ is the quotient of $\otimes_ip_*q^*(E_i^{[n]})$ by the torsion subsheaf, which leads to the identification with $p_*q^*(\otimes_i E_i^{[n]})$ by theorem \ref{Sca}.
\end{proof}  
\begin{remark}\label{permute}
The morphism $\otimes_i \gamma_i \colon \otimes_i p_*q^*(E_i^{[n]})\to \otimes_i C^0_{E_i}$ is $\sym_n$-equivariant since all the $\gamma_i$ are equivariant. Thus, the isomorphism of the previous proposition is $\sym_n$-equivariant when considering 
$\im(\otimes_i\gamma_1)$ with the linearization induced by the linearization on $\otimes_iC^0_{E_i}$.
In the case $E_1=\dots=E_k=E$ the isomorphism is also $\sym_k$-equivariant when considering $\otimes_i E_i^{[n]}$ as well as 
$\otimes_i C^0_{E_i}$ with the action given by permuting the tensor factors. 
\end{remark}

%% file: image.tex
\section{Description of $p_*q^*(E_1^{[n]}\otimes \dots\otimes E_k^{[n]})$}
In this whole section let $E_1, \dots, E_k$ be locally free sheaves on a quasi-projective surface $X$ and $n\in \N$.
We will use proposition \ref{im} in order to study the image of $E_1^{[n]}\otimes\dots \otimes E_k^{[n]}$ under the Bridgeland--King--Reid--Haiman equivalence.
We have to keep track of the $\sym_n$-linearization, since we later want compute the $\sym_n$-invariant cohomology (see proposition \ref{mubkr}). 
In the case that $E_1=\dots=E_k=E$ we also have to keep track of the $\sym_k$-action in order to get later also results for the symmetric products of tautological bundles (see remark \ref{permute}).  
\subsection{Construction of the $T_\ell$ and $\phi_\ell$}\label{constr}
\begin{defin}
Let $k,n\in\N$. For $1\le \ell\le k$ we define $I_\ell$ as the set of tuples of the form $(M;i,j;a)$ consisting of a subset $M\subset [k]$ with $|M|=\ell$, two numbers $i,j\in [n]$ with $i<j$, and
a multi-index $a\colon [k]\setminus M\to [n]$.
Given such a tuple we set \[\hat M:=\hat M(i,j;a):=M\cup a^{-1}(\{i,j\})\]
and $a_|:=a_{|[k]\setminus \hat M}$. The data of $i,j\in[n]$ with  $i<j$ is the same as the subset $\{i,j\}\subset [n]$. Thus, we will also write $(M;\{i,j\};a)$ instead of $(M;i,j;a)$. 
\end{defin}
For $E_1,\dots,E_k$ locally free sheaves on $X$ and $\ell=1,\dots,k$ we define the coherent sheaf 
\[T_\ell(E_1,\dots,E_k):=\bigoplus_{(M;i,j;a)\in I_\ell} S^{\ell-1}N^\vee_{\Delta_{i,j}}\otimes\left( \bigotimes_{\alpha\in M} E_\alpha\right)_{i,j}\otimes \left(\bigotimes_{\beta\in [k]\setminus M} \pr_{a(\beta)}^*E_\beta\right)\]
on $X^n$. We will often leave out the $E_i$ in the argument of $T_\ell$ and denote the direct summands by $T_\ell(M;i,j;a)$. If we want to emphasise the values of $k$ and $n$ we will put them in the left under respectively upper index of the objects and  morphisms, e.g. we will write $_k^nT_\ell$. We can rewrite the summands as 
\[T_\ell(M;i,j;a)= \left( S^{\ell-1}\Omega_X\otimes ( \bigotimes_{\alpha\in \hat M} E_\alpha)\right)_{i,j}\otimes \left(\bigotimes_{\beta\in [k]\setminus\hat M} \pr_{a(\beta)}^*E_\beta\right)\]
or as (see \cite[Chapter II 8]{Har2})
\begin{align*}T_\ell(M;i,j;a)&= S^{\ell-1}(\mathcal I_{i,j}/\mathcal I_{i,j}^2)\otimes \left ( \bigotimes_{\alpha\in \hat M} E_\alpha\right)_{i,j}\otimes \left(\bigotimes_{\beta\in [k]\setminus\hat M} \pr_{a(\beta)}^*E_\beta\right)\\
&\cong (\mathcal I_{i,j}^{\ell-1}/\mathcal I_{i,j}^\ell) \otimes \left ( \bigotimes_{\alpha\in \hat M} E_\alpha\right)_{i,j}\otimes \left(\bigotimes_{\beta\in [k]\setminus\hat M} \pr_{a(\beta)}^*E_\beta\right)\,.
\end{align*}
As in subsection \ref{ccomplex} for the terms $F_I$, we get for $\sigma\in \sym_n$ by flat base change canonical isomorphisms
\[\sigma_*\colon T_\ell(M;\sigma^{-1}(\{i,j\});\sigma^{-1}\circ a)\to \sigma^*T_\ell(M;i,j;a)\,.\]
Thus, there is a $\sym_n$-linearization $\lambda$ of $T_\ell$ given on local sections $s\in T_\ell$ by
\[\lambda_\sigma(s)(M;i,j;a)=\eps_{\sigma,\sigma^{-1}(\{i,j\})}^\ell\sigma_*s(M;\sigma^{-1}(\{i,j\});\sigma^{-1}\circ a)\,.\]
\begin{remark}\label{minus}
For $\sigma=(i\,\,j)$ the map $\sigma_*\colon N_{\Delta_{ij}}\to N_{\Delta_{ij}}$ is given by multiplication with $-1$ (see \cite[section 4]{Kru}). Thus, $\sigma_*\colon T_\ell(M;i,j;a)\to T_\ell(M;i,j;\sigma^{-1}\circ a)$ is given by multiplication with $(-1)^{\ell-1}$. Together with the sign $\eps_{\sigma,\sigma^{-1}(\{i,j\})}^\ell$ this makes $\sigma$ act by $-1$ on $T_\ell(M;i,j;a)$ for every tuple $(M;i,j;a)$ such that $a^{-1}(\{i,j\})=\emptyset$, i.e. if $\hat M=M$.  
\end{remark}
If $E_1=\dots=E_k$ we define a $\sym_k$-action on $T_\ell$ by setting 
\[(\mu\cdot s)(M;i,j;a):=\mu\cdot s(\mu^{-1}(M);i,j;a\circ\mu)\]
for $\mu\in \sym_k$. The action of $\mu$ on the right-hand side is given by permuting the factors $E_t$ of the tensor product.  Since the two linearizations commute, they give a $\sym_n\times \sym_k$-linearization of $T_\ell$. We will now successively define $\sym_n$- respectively $\sym_n\times\sym_k$-equivariant morphisms $\phi_\ell\colon K_{\ell-1}\to T_\ell$, where
\[K_0(E_1,\dots,E_k):=K_0:=\bigotimes_{t=1}^kC_{E_t}^0= \bigoplus_{a\colon[k]\to [n]}K_0(a)\quad,\quad K_0(a)=\bigotimes_{t=1}^k\pr_{a(t)}^*E_t\]
and $K_\ell:=\ker(\phi_\ell)$ for $\ell=1,\dots,k$. 
We set $I_0:=\Map([k],[n])$.
We consider $K_0$ with the $\sym_n$-linearization $\lambda$ given by
$\lambda_\sigma(s)(a):=\sigma_*s(\sigma^{-1}\circ a)$ and, if all the $E_t$ are equal, with the $\sym_k$-action $(\mu\cdot s)(a):= \mu\cdot s(a\circ \mu)$
(see also remark \ref{permute}).
For $(M;i,j;a)\in I_\ell$ we set
\[I_0\supset I(M;i,j;a):=\bigl\{c\colon[k]\to [n]\mid c(M)\subset \{i,j\}\,,\,c_{|[k]\setminus M}=a\bigr\}=\bigl\{a\uplus b\mid b\colon M\to \{i,j\}   \bigr\}\]
and define $K_{\ell-1}(M;i,j;a)$ as the image of $K_{\ell-1}$ under the projection
\[K_{\ell-1}\hookrightarrow K_0=\bigoplus_{c\in I_0} K_0(c)\to \bigoplus_{c\in I(M;i,j;a)}K_0(c)\]
We define the component $\phi_{\ell}(M;i,j;a)\colon K_{\ell-1}\to T_\ell(M;i,j;a)$ of $\phi_\ell$ as the composition
of the projection $K_{\ell-1}\to K_{\ell-1}(M;i,j;a)$ with a morphism $K_{\ell-1}(M;i,j;a)\to T_\ell(M;i,j;a)$. We denote the latter morphism again by $\phi_\ell(M;i,j;a)$ and will define it in the following.
For $s\in K_{\ell-1}\subset K_0$ the above means that $s(c)$ for $c\notin I(M;i,j;a)$ does not contribute to $\phi_\ell(s)(M;i,j;a)$.
We assume first that for all $t\in \hat M=M\cup a^{-1}(\{i,j\})$ the bundles $E_t$ equal the trivial line bundle, i.e. $E_t=\reg_X$.
Then for all $b\colon M\to \{i,j\}$ we have 
\[K_0(a\uplus b)=H\quad,\quad T_\ell(M;i,j;a)=(\mathcal I_{i,j}^{\ell-1}/\mathcal I_{i,j}^{\ell})
 \otimes H\quad,\quad H:= \bigotimes_{t\in [k]\setminus \hat M}\pr_{a(t)}^*E_t\,.
\]
Thus, for a local section $s\in K_{\ell-1}$ the components $s(a\uplus b)\in K_{\ell-1}(M;i,j;a)$ are all sections of the same locally free sheaf $H$ and we can define 
 \[\phi_\ell(M;i,j;a)(s):=\sum_{b\colon M\to \{i,j\}}\eps_b s(a\uplus b)\quad \text{mod $\I_{ij}^\ell$}\]
where $\eps_b=(-1)^{\#\{t\mid b(t)=j\}}$.
Inductively, the map $\phi_\ell(M;i,j;a)$ is well defined, which means that $\phi_\ell(M;i,j;a)(s)\in \mathcal I_{i,j}^{\ell-1}\cdot H$, since if we take any $m\in M$ we have
\[\phi_\ell(M;i,j;a)(s)=\sum_{b\colon M\setminus\{m\}\to \{i,j\}}\eps_b s(a\uplus b,m\mapsto i)\,-\, \sum_{b\colon M\setminus\{m\}\to\{i,j\}}\eps_b s(a\uplus b,m\mapsto j)\,.\]
Both sums occuring are elements of $\mathcal I_{i,j}^{\ell-1}$ since because of $s\in \ker\phi_{\ell-1}$ we have
\[\phi_{\ell-1}(s)(M\setminus \{m\};i,j;a,m\mapsto i)=\phi_{\ell-1}(s)(M\setminus \{m\};i,j;a,m\mapsto j)=0 \quad\text{mod $\mathcal I_{i,j}^{\ell-1}$}\,.\]  
Let now $E_t$ for $t\in \hat M$ be the trivial vector bundle of rank $r_t$. Let $r:=\prod_{t\in\hat M} r_t$. Then for $b\colon M\to \{i,j\}$ we have 
$K_0(a\uplus b)=(\oplus_\alpha H)$ and $T_\ell=(\I_{ij}^{\ell-1}/\I_{ij}^\ell)\otimes (\oplus_\alpha H)$. 
Here the index $\alpha$ goes through all multi-indices $(\alpha_t|t\in \hat M)$ with $1\le \alpha_t\le r_t$.
Now we can define $\phi_\ell(M;i,j;a)$ the same way as before. 
The components $\phi_{\ell}(M;i,j;a)(\alpha,\alpha')$ are zero if $\alpha\neq \alpha'$ and coincide with the $\phi_\ell(M;i,j;a)$ from the trivial line bundle case if $\alpha=\alpha'$. 
\begin{remark}
Every collection of automorphisms of the $E_t=\reg_X^{\oplus r_t}$ for $t\in \hat M$ induces canonically an automorphism of $\oplus_\alpha H$. The morphism $\phi_\ell(M;i,j;a)$ commutes with the automorphisms induced by this automorphism on its domain and codomain.
\end{remark}
This observation allows us to define $\phi_\ell(M;i,j;a)$ in the case of general locally free sheaves as follows.    
We choose an open covering $\{U_m\}_m$ of $X$ such that on every open set $U_m$ all the $E_t$ are simultaneously trivial, say with trivialisations $\mu_{m,t}\colon E_{t|U_m}\xrightarrow{\cong} \reg_{U_m}^{r_t}$. 
Then the tivialisations $\mu_{m,t}$ for $t\in \hat M$ induce over $\pr_{ij}^{-1}(U_m\times U_m)$ isomorphisms $K_0(a\uplus b)\cong(\oplus_\alpha H)$ and $T_\ell\cong(\I_{ij}^{\ell-1}/\I_{ij}^\ell)\otimes (\oplus_\alpha H)$. We define the restriction of $\phi_\ell(M;i,j;a)$ to $\pr_{ij}^{-1}(U_m^2)$ under these isomorphisms as the morphism $\phi_\ell(M;i,j;a)$ from the case of trivial vector bundles. It is independent of the chosen trivialisations by the above remark. Thus, the $\phi_\ell(M;i,j;a)$ defined over the $\pr_{ij}^{-1}(U_m^2)$ for varying $m$ glue together. Since the $\pr_{ij}^{-1}(U_m^2)$ cover the partial diagonal $\Delta_{ij}$, which is the support of $T_\ell(M;i,j;a)$, this defines $\phi_\ell(M;i,j;a)$ globally. Using lemma \ref{signlemma} one can check that the morphisms $\phi_\ell$ are indeed equivariant.   
\subsection{The open subset $X^n_{**}$}\label{twostar}
As done in \cite{Dandual} and \cite{Sca1}, we consider the  following open subvarieties of $X^n$, $S^nX$, and $X^{[n]}$. Let
$W\subset S^nX $ be the closed subvariety of unordered tuples $\sum_{i=1}^n x_i$ with the property that $|\{x_1,\dots,x_n\}|\le n-2$, i.e 
\[W=\pi\left(\bigcup_{|J|=|K|=2,\\ J\ne K} \Delta_J\cap\Delta_K\right) \ \,.\]
We set $S^nX_{**}:=S^nX\setminus W$, $X^n_{**}:=\pi^{-1}(S^nX_{**})$, $X^{[n]}_{**}:=\mu^{-1}(S^nX_{**})$ and 
\[I^nX_{**}= q^{-1}(X^{[n]}_{**}) = p^{-1}(X^n_{**}) = (X^{[n]}_{**}\times_{S^nX_{**}} X^n_{**})_{\text{red}}\,.\]
In summary, there is the following open immersion of commutative diagrams   
\[
\begin{CD}
I^nX_{**}
@>{p_{**}}>>
X^n_{**} \\
@V{q_{**}}VV
@VV{\pi_{**}}V \\
X^{[n]}_{**}
@>>{\mu_{**}}>
S^nX_{**}
\end{CD}\quad\overset{\iota}{\hookrightarrow}\quad
 \begin{CD}
I^nX
@>{p}>>
X^n \\
@V{q}VV
@VV{\pi}V \\
X^{[n]}
@>>{\mu}>
S^nX
\end{CD}
\]
where we denote every open immersion $(\_)_{**}\hookrightarrow (\_)$ by $\iota$. For a sheaf or complex of sheaves $F$ on $X^n$, $S^nX$, $X^{[n]}$ or $I^nX$ we write $F_{**}$ for its restriction to the appropriate open subset.
The codimensions of the complements are at least two. More precisely, we have
\begin{align*}\codim(X^n\setminus X^n_{**},X^n)&=\codim(S^nX\setminus S^nX_{**},S^nX)=4\,,\\ \codim(X^{[n]}\setminus X^{[n]}_{**},X^{[n]})&=\codim(I^nX\setminus I^nX_{**},I^nX)=2\,.\end{align*}
\begin{lemma}\label{pushpull}
Let $E_1,\dots,E_k$ be locally free sheaves on $X$. Then on $X^n$ there is a natural isomorphism
\[\iota_*\iota^*p_*q^*(E_1^{[n]}\otimes \dots\otimes E_k^{[n]})\cong p_*q^*(E_1^{[n]}\otimes \dots\otimes E_k^{[n]})\,.\]
\end{lemma}
\begin{proof}
 We apply lemma \ref{codimpush} with $f=p$ and $U=X^n_{**}$.
\end{proof}
\subsection{Description of $p_*q^*(E_1^{[n]}\otimes \dots\otimes E_k^{[n]})_{**}$}
\begin{prop}\label{opendes}
Let $E_1,\dots, E_k$ be locally free sheaves on $X$. Then on the open subset $X^n_{**}\subset X^n$ there is the equality 
\[K_{k**}=p_*q^*(E_1^{[n]}\otimes \dots \otimes E_k^{[n]})_{**}\]
of subsheaves of $K_{0**}$.
\end{prop}
\begin{proof}
We will often drop the indices $(\_)_{**}$ in this proof. Using proposition \ref{im} it suffices to show that $K_k=\im(\gamma_{E_1}\otimes \dots\otimes \gamma_{E_k})$. For fixed $1\le i\le j\le n$ and $\ell\in[k]$ we denote by $\phi_\ell(i,j)$ the direct sum of all $\phi_\ell(M;i,j;a)$ with $(M;i,j;a)\in I_\ell$. On the open subset $X^n_{**}\subset X^n$ the pairwise diagonals $\Delta_{i,j}$ do not intersect. We denote the big diagonal by $\mathbbm D=\cup_{1\le i<j\le n} \Delta_{i,j}$. Then $X^n_{**}$ is covered by the open subsets $V_{i,j}:=(X^n_{**}\setminus\mathbbm D)\cup \Delta_{i,j}$. Thus we can test the equality on the $V_{i,j}$ where 
\[K_\ell=K_\ell(i,j)=\cap_{\ell=1}^k \ker\phi_\ell(i,j)\]
holds.
We will assume without loss of generality the case that $i=1$ and $j=2$. 
We consider as in the construction of the $\phi_\ell$ an open covering $\{U_m\}_m$ of $X$ on which all the $E_t$ are simultaneously trivial. Since both $K_k$ and $\im(\gamma_{E_1}\otimes \dots\otimes \gamma_{E_k})$ equal $K_0$ on $V_{12}\setminus \Delta_{12}= X^n_{**}\setminus \mathbbm D$, it is sufficient to show the equality on every member of the covering of $\Delta_{12}$ given by
\[U_m\times U_m\times U_{m_3}\times\dots\times U_{m_n}\,.\]
Since on these open sets the maps $\phi_\ell(1,2)$ are defined as the maps $\phi_\ell(1,2)(\reg_X^{\oplus r_1},\dots,\reg_X^{\oplus r_k})$ under the trivializations, we may assume that all the $E_t$ are trivial vector bundles of rank $r_t$, i.e. $E_t=\reg_X^{\oplus r_t}$. Since in this case the $\phi_\ell$ as well as $\gamma_{E_1}\otimes\dots\otimes \gamma_{E_k}$ are defined component-wise, we may assume that $E_1=\dots=E_t=\reg_X$. By theorem \ref{Sca} (i) a section $x\in p_*q^* \reg_X^{[n]}\subset C^0_{\reg_X}$ over $V_{12}$ is of the form $x=(x(1),x(2),\dots,x(n))$ with $x(\alpha)\in \pr_\alpha^*\reg_X\cong \reg_{X^n}$ for $\alpha\in [n]$ and $x(1)_{|\Delta_{12}}=x(2)_{|\Delta_{12}}$. For a section $s\in K_0$ and a multi-index $a\colon [k]\to [n]$ we denote the component of $s$ in 
\[K_0(a)=\pr_{a(1)}\reg_X\otimes\dots\otimes \pr_{a(k)}\reg_X\cong \reg_{X^n}\]
by $s(a)$. The image of a pure tensor $x_1\otimes \dots \otimes x_k\in (p_*q^*\reg_X)^{\otimes k}$ under the $k$-th power of $\gamma=\gamma_{\reg_X}$ is given by
\[\gamma^{\otimes k}(x_1\otimes \dots\otimes x_k)(a)=x_1(a(1))\cdots x_k(a(k))\in \reg_{X^n}\,.\]
For a tuple $(M,a)$ with $\emptyset\neq M\subset[k]$, $a\colon [k]\setminus M\to [n]$, and $s\in K_0$ we set
\[s(M,a)=\sum_{b\colon M\to[2]}\eps_b s(a\uplus b)\,\in\reg_{X^n} \,.\]
Then for a section $s\in K_0$ being a section of $K_k=K_k(1,2)$  is equivalent to the condition that $s(M,a)\in \mathcal I^{|M|}$ for each tuple $(M,a)$ as above, where $\mathcal I:=\mathcal I_{12}$. We first show the inclusion $\im (\gamma^{\otimes k})\subset K_k$. For this let $x=x_1\otimes\dots\otimes x_k\in (p_*q^*\reg_X)^{\otimes k}$ and $s=\gamma^{\otimes k}(x)$. We show by induction over $|M|$ that $s(M,a)\in \mathcal I^{|M|}$ for each pair $(M,a)$. For $M=\{t\}$ we have
\[s(\{t\},a)=s(t\mapsto 1,a)-s(t\mapsto 2,a)=(x_t(1)-x_t(2))\cdot\prod_{i\in[k]\setminus\{t\}}x_i(a(i))\]
which is indeed a section of $\mathcal I$ since $x_t(1)_{|\Delta_{12}}=x_t(2)_{|\Delta_{12}}$. For an arbitrary $M\subset [k]$ we choose an $m\in M$ and set
\[\tilde x=x_1\otimes \dots \otimes x_{m-1}\otimes \tilde x_m\otimes x_{m+1}\otimes \dots\otimes x_k\]
with $\tilde x_m(j)=1$ for every $j\in [n]$. We also set $\tilde s=\gamma^{\otimes k}(\tilde x)$. With this notation 
\[s(M,a)=(x_m(1)-x_m(2))\cdot \tilde s(M\setminus \{k\}, a,m\mapsto 1)\,.\]
By induction we have $\tilde s(M\setminus \{k\}, a,m\mapsto 1)\in \mathcal I^{|M|-1}$ and thus $s(M,a)\in \mathcal I^{|M|}$.
For the inclusion $K_k\subset \im \gamma^{\otimes k}$ we need the following lemma, where we are still working over $V_{12}$.
For $a\colon [k]\to [n]$ we set $\hat M(a):=a^{-1}(\{1,2\})$ and $a_|:=a_{|[k]\setminus \hat M(a)}$.
\begin{lemma}\label{elim}
  Let $s\in K_k$ be a local section and $a\colon [k]\to [n]$ such that $s(b)=0$ 
for all $b\colon [k]\to [n]$ with $(\hat M(b),b_|)=(\hat M(a),a_|)$ and $|b^{-1}(\{2\})|<|a^{-1}(\{2\})|$. Then there exists a local section $x\in (p_*q^*\reg_X^{[n]})^{\otimes k}$ such that $\gamma^{\otimes k}(x)(a)=s(a)$ and $\gamma^{\otimes k}(x)(c)=0$ for all multi-indices $c\colon [k]\to [n]$ with the property that $(\hat M(c),c_|)\neq (\hat M(a),a_|)$ or with the property that there exists an $i\in \hat M(a)=\hat M(c)$ with $c(i)=1$ and $a(i)=2$. 
\end{lemma}
\begin{proof}
We assume for simplicity that $\hat M(a)=[u]$ and $a^{-1}(\{2\})=[v]$ with $1\le v\le u\le k$, i.e. $a$ is of the form
\[a=\left(2,\dots,2,1,\dots,1,a(u+1),\dots,a(k)\right)\,.\]
By the assumptions
\[\mathcal I^v\ni s([v],a_{|[v+1,k]})= s([v],\underline 1 \uplus a_|)=\sum_{b\colon [v]\to [2]}\eps_b s(b\uplus \underline 1\uplus a_|)=(-1)^v s(a)\,.\]
Now we can write $s(a)=\sum_{\alpha\in A} y_{\alpha,1}\cdots y_{\alpha,v}$ as a finite sum with all $y_{\alpha,\beta}\in \mathcal I$. We denote by $e_j$ the section of $C^0_{\reg_X}$ with $e_j(h)=\delta_{jh}$. Then the section 
\[x=\sum_{\alpha\in A} y_{\alpha,1}e_2\otimes \dots\otimes y_{\alpha,v}e_2\otimes (e_1+e_2)\otimes \dots\otimes (e_1+e_2)\otimes e_{a(u+1)}\otimes \dots\otimes e_{a(k)}\,\]
is indeed in $(p_*q^*\reg_X^{[n]})^{\otimes k}$ and has the desired properties.   
\end{proof}  
Let $\prec$ be any total order on the set of tuples $(M,a)$ with $M\subset [k]$ and $a\colon [k]\setminus M\to [3,n]$ and let $\lhd_{M,v}$ be any total order on the set of subsets of $M$ of cardinality $v$. We define a total order $<$ on the set $I_0=\Map([k],[n])$ by setting $b<a$ if 
$(\hat M(b),b_|)\prec (\hat M(a),a_|)$ or if $(\hat M(b),b_|)= (\hat M(a),a_|)$ and $|b^{-1}(2)|<|a^{-1}(2)|$ or if $(\hat M(b),b_|)= (\hat M(a),a_|)=:M$ and $|b^{-1}(2)|=|a^{-1}(2)|=:v$ and $b^{-1}(2)\lhd_{M,v} a^{-1}(2)$. For $s\in K_k$ let $a$ be the minimal multi-index with $s(a)\neq 0$, i.e. $s(b)=0$ for all $b<a$. Then lemma \ref{elim} yields a $x\in (p_*q^*\reg_X^{[n]})^{\otimes k}$ such that $\hat s=s-\gamma^{\otimes k}(x)$ fulfills $\hat s(b)=0$ for all $b\le a$. Thus, by induction over the set $I_0$ with the order $<$, indeed, $s\in \im(\gamma^{\otimes k})$ which completes the proof of proposition \ref{opendes}.   
\end{proof}
\subsection{Description of $p_*q^*(E_1^{[n]}\otimes \dots\otimes E_k^{[n]})$}
The result of the last subsection carries over directly to the whole $X^n$.
\begin{theorem}\label{maindesc}
Let $E_1,\dots, E_k$ be locally free sheaves on $X$. Then on $X^n$ there is the equality 
\[K_{k}=p_*q^*(E_1^{[n]}\otimes \dots \otimes E_k^{[n]})\]
of subsheaves of $K_{0}$.
\end{theorem}
\begin{proof}
Since $K_0$ is locally free and $\codim(X^n\setminus X^n_{**},X^n)=4$ we have $\iota_*K_{0**}=K_0$. Furthermore the direct summands of $T_\ell$ are push forwards of locally free sheaves on the partial diagonals $\Delta_{i,j}$. Since \[\codim(\Delta_{ij}\setminus(\Delta_{ij}\cap X^n_{**}), \Delta_{i,j})=2\] we get by lemma \ref{codimpush} that $\iota_*T_{\ell**}=T_\ell$ for all $\ell\in[k]$. Using lemma \ref{codimseq} we get by induction that $\iota_*K_{\ell **}=K_\ell$ for $\ell\in [k]$. In particular
\[K_k=\iota_*K_{k**}\overset{\ref{opendes}}=\iota_*\left( p_*q^*(E_1^{[n]}\otimes \dots\otimes E_k^{[n]})_{**}\right)\overset{\ref{pushpull}}= p_*q^*(E_1^{[n]}\otimes \dots\otimes E_k^{[n]})\,.\]  
\end{proof}
\begin{cor}\label{mucor}
There are natural isomorphisms $\mu_*(E_1^{[n]}\otimes\dots\otimes E_k^{[n]})\cong K_k^{\sym_n}$ and
\[\Ho^*(X^{[n]},E_1^{[n]}\otimes \dots\otimes E_k^{[n]})\cong \Ho^*(X^n,K_k)^{\sym_n}\,.\] 
\end{cor}
\begin{proof}
 This follows by the previous theorem together with proposition \ref{mubkr} and theorem \ref{Sca}.
\end{proof}

%% file: orbits.tex
\section{Invariants of $K_0$ and the $T_\ell$}
In this section we will use Danila's lemma \ref{Dan} in order to compute the invariants of the sheaves $K_0$ and $T_\ell$.
\subsection{Orbits and their isotropy groups on the sets of indices}\label{orbi}
For $\ell=1,\dots,k$ we have the decomposition $T_\ell=\oplus_{I_\ell} T_\ell(M;i,j;a)$ with
\[I_\ell=\bigl\{(M;i,j;a)\mid M\subset [k]\,,\, \#M=\ell\,,\, 1\le i<j\le n\,,\, a\colon[k]\setminus M\to [n]\bigr\}\,.\]
The $\sym_n$- as well as the $\sym_k$-linearization of $T_\ell$ induce actions on $I_\ell$ given for $\sigma\in \sym_n$ and $\mu\in \sym_k$ by
\[\sigma\cdot (M;\{i,j\};a)=(M;\sigma(\{i,j\});\sigma\circ a)\quad,\quad \mu\cdot(M;\{i,j\};a)=(\mu(M);\{i,j\};a\circ \mu^{-1})\,.\] 
Furthermore there is the decomposition
\[K_0=\bigoplus_{a\in I_0}K_0(a)\quad,\quad I_0=\Map([k],[n])\quad,\quad K_0(a)=\bigotimes_{t=1}^k \pr_{a(t)}^* E_t\,.\]
The $\sym_n\times\sym_k$-action on $I_0$ is given by $\sigma\cdot a=\sigma\circ a$ and $\mu\cdot a=a\circ \mu^{-1}$. Let $\prec$ be any total order on the set of subsets of $[k]$ such that $\emptyset$ is the maximal element.
\begin{remark}\label{orbits}
\begin{enumerate}
Let $1\le \ell\le k$.
\item Every $\sym_n$-orbit of $I_0$ has a unique representative $a$ such that 
\[ a^{-1}(1)\prec a^{-1}(2)\prec\dots \prec a^{-1}(n)\,.\]
We denote the set of these representatives by $J_0$. For $a\in I_0$ the isotropy group is given by
$\Stab_{\sym_n}(a)=\sym_{[n]\setminus \im(a)}$.
For $a\in J_0(1)$ we have $[n]\setminus \im (a)=[\max a+1,n]$.
\item \label{Snorbits} Every $\sym_n$-orbit of $I_\ell$ has a unique representative of the form $(M;1,2;a)$ such that $a^{-1}(1)\prec a^{-1}(2)$ and
\[a^{-1}(3)\prec a^{-1}(4)\prec \dots\prec a^{-1}(n)\,.\]
We denote the set of these representatives by $J_\ell$. 
Furthermore, we set 
\[
 \hat I_\ell:=\bigl\{(M;i,j;a)\in I_\ell\mid a^{-1}(\{i,j\})\neq \emptyset\bigr\}\quad,\quad\hat J_\ell:=J_\ell\cap \hat I_\ell\,.
\]
We will often use the identification $(M;a)\cong (M;1,2;a)\in J_\ell$ in the notations.
The isotropy group of a tuple $(M;i,j;a)\in I_\ell$ with $Q:=\{i,j\}\cup \im(a)$ and $\bar Q=[n]\setminus Q$ is given by
\[\Stab_{\sym_n}(M;i,j;a)=\begin{cases}
             \sym_{\bar Q} &\text{ if $(M;i,j;a)\in \hat I_\ell$,}\\
             \sym_{\{i,j\}}\times  \sym_{\bar Q} &\text{ if $(M;i,j;a)\notin \hat I_\ell$.}
                          \end{cases}
  \]
If $(M;i,j;a)=(M;1,2;a)\in  J_\ell$ we have $Q=[\max (a,2)]$ and $\bar Q=[\max (a,2)+1,n]$.
\end{enumerate}
\end{remark} 
\subsection{The sheaves of invariants and their cohomology}
\begin{lemma}\label{kzero}
There is a natural isomorphism
\[K_0^{\sym_n}\cong \bigoplus_{a\in J_0}K_0(a)^{\sym_{[\max a+1,n]}}\,.\] 
\end{lemma}
\begin{proof}
 This follows from Danila's lemma and remark \ref{orbits} (i).
\end{proof}
\begin{lemma}
\[T_\ell^{\sym_n}\cong\bigoplus_{(M;a)\in \hat J_\ell(1)}T_\ell(M;a)^{\sym_{[\max(a,2)+1, n]}}\]
holds for every $\ell=1,\dots,k$.
\end{lemma}
\begin{proof}
By Danila's lemma and remark \ref{orbits} (ii) we have
\[T_\ell^{\sym_n}=\bigoplus_{(M;a)\in  J_\ell}T_\ell(M;a)^{\Stab(M;a)}\,.\]
Let $(M;a)\in J_\ell\setminus \hat J_\ell$. Then $\tau=(1\,\,2)\in \Stab(M;a)$ acts on $T_\ell(M;a)$ by $(-1)^{\ell+\ell-1}=-1$ (see remark \ref{minus}) which makes the invariants vanish.
\end{proof}
\begin{cor}\label{kvanish}
The sheaf $T_k^{\sym_n}$ is zero and thus 
\[\mu_*(E_1^{[n]}\otimes\dots\otimes E_k^{[n]})\cong p_*q^*(E_1^{[n]}\otimes\dots\otimes E_k^{[n]})^{\sym_n}\cong K_{k-1}^{\sym_n}\,.\]
\end{cor}
\begin{proof}
The set $\hat J_k$ is empty. The isomorphisms follow by corollary \ref{mucor}.
\end{proof}
\begin{remark}\label{partialquotient}
For a subset $Q\subset [n]$ with $|Q|=q$ and $\bar Q=[n]\setminus Q$ we denote by $X^Q\times S^{\bar Q} X$ the quotient of $X^n$ by the  $\sym_{\bar Q}$-action. It is isomorphic to $X^{q}\times S^{n-q}X$. We denote by $\pi_Q\colon X^Q\times S^{\bar Q}X\to S^nX$ the morphism induced by the quotient morphism $\pi\colon X^n\to S^nX$. Under the identification $X^Q\times S^{\bar Q} X\cong X^q\times S^{n-q}$ it is given by
\[(x_1,\dots,x_q,\Sigma)\mapsto x_1+\dots+x_q+\Sigma\,.\]
Let $a\in I_0$ respectively $(M;i,j;a)\in \hat I_\ell$ and $Q=\im(a)$ respectively $Q=\{i,j\}\cup \im(a)$, and $\bar Q=[n]\setminus Q$. The sheaves $K_0(a)^{\sym_{\bar Q}}$ respectively $T_\ell(M;i,j;a)^{\sym_{\bar Q}}$ in the two lemmas above are considered as sheaves on the $\sym_n$-quotient $S^nX$, i.e. they are abbreviations
\[K_0(a)^{\sym_{\bar Q}}:=(\pi_*K_0(a))^{\sym_{\bar Q}}\quad ,\quad T_\ell(M;i,j;a)^{\sym_{\bar Q}}:=(\pi_*T_\ell(M;i,j;a))^{\sym_{\bar Q}}\]
But we can also take the $\sym_{\bar Q}$-invariants already on the $\sym_{\bar Q}$-quotient $X^Q\times S^{\bar Q} X$ and consider $K_0(a)^{\sym_{\bar Q}}$ and $T_\ell(M;i,j;a)^{\sym_{\bar Q}}$ as sheaves on this variety. With this notation we have
\[(\pi_*K_0(a))^{\sym_{\bar Q}}=\pi_{Q*}(K_0(a)^{\sym_{\bar Q}})\quad,\quad (\pi_*T_\ell(M;i,j;a))^{\sym_{\bar Q}}=\pi_{Q*}(T_\ell(M;i,j;a)^{\sym_{\bar Q}})\,.\]
We denote for $m\in Q$ by $p_m\colon X^Q\times S^{\bar Q}X\to X$ the projection induced by the projection $\pr_m\colon X^n\to X$. For $I\subset Q$ we have the closed embedding $\Delta_I\times S^{\bar Q} X\subset X^Q\times S^{\bar Q}X$ which is the $\sym_{\bar Q}$-quotient of the closed embedding $\Delta_J\subset X^n$. Then the sheaves of invariants considered as sheaves on $X^Q\times S^{\bar Q}X$ are given by  
\begin{align*}K_0(a)^{\sym_{\bar Q}}=\bigotimes_{\substack{m\in Q\\ t\in a^{-1}(m)}}p_m^*E_t\quad,\quad T_\ell(M;i,j;a)^{\sym_{\bar Q}}=\Bigl(\bigotimes_{t\in M\cup a^{-1}(\{i,j\})}E_t\Bigr)_{ij}\otimes \bigotimes_{\substack{m\in Q\setminus\{i,j\}\\ t\in a^{-1}(m)}}p_m^*E_t\,.\end{align*}
In particular, $K_0(a)^{\sym_{\bar Q}}$ is still locally free. The sheaf $T(M;i,j;a)^{\sym_{\bar Q}}$ can also be considered as a sheaf on its support $\Delta_{ij}\times S^{\bar Q}X$ on which it is locally free, too.
\end{remark}
For the following, remember that we interpret an empty tensor product of sheaves on the surface $X$ as the structural sheaf $\reg_X$.
\begin{lemma}\label{kunneth}
\begin{enumerate}
 \item  For every $a\in I_0$ the cohomology $\Ho^*\left(X^n,K_0(a)\right)$ is naturally isomorphic to 
\[ \Ho^*\left(\bigotimes_{t\in a^{-1}(1)}E_t\right)\otimes \dots \otimes \Ho^*\left(\bigotimes_{t\in a^{-1}(n)}E_t\right)\,.\] 
\item  For every $a\in J_0$ the invariant cohomology $\Ho^*\left(X^n,K_0(a)\right)^{\sym_{[\max a+1,n]}}$ is naturally isomorphic to
\[\Ho^*\left(\bigotimes_{t\in a^{-1}(1)}E_t\right)\otimes \dots \otimes \Ho^*\left(\bigotimes_{t\in a^{-1}(\max a)}E_t\right)\otimes S^{n-\max a}\Ho^*(\reg_X)\,.\]
\item  For every $(M;i,j;a)\in I_\ell$ the cohomology $\Ho^*\left(X^n,T_\ell(M;i,j;a)\right)$ is  naturally isomorphic to
\[ \Ho^*\left(S^{\ell-1}\Omega_X\otimes\bigotimes_{t\in \hat M(a)}E_t\right)\otimes \bigotimes_{m\in [n]\setminus\{i,j\}} \Ho^*\left(\bigotimes_{t\in a^{-1}(m)}E_t\right)\,.\]
\item  For $(M;a)\in \hat J_\ell$ the invariant cohomology $\Ho^*\left(X^n,T_\ell(M;a)\right)^{\sym_{[\max(a,2)+1,n}]}$ is naturally isomorphic to
\[\Ho^*\left(S^{\ell-1}\Omega_X\otimes\bigotimes_{t\in \hat M(a)}E_t\right)\otimes\bigotimes_{m=3}^{\max a}\Ho^*\left(\bigotimes_{t\in a^{-1}(m)}E_t\right)\otimes S^{n-\max (a,2)}\Ho^*(\reg_X)\,.\]
\end{enumerate}
\end{lemma}
\begin{proof}
The natural isomorphisms in (i) and (iii) are the K\" unneth isomorphisms. The assertions (ii) and (iv) follow from the fact that the natural $\sym_n$-linearization of $\reg_{X^n}$ induces the action on 
\[\Ho^*(X^n,\reg_{X^n})\cong \Ho^*(\reg_X)^{\otimes n}\]
given by permuting the tensor factors together with the cohomoligcal sign $\eps_{\sigma,p_1,\dots,p_n}$ (see section \ref{graded}).  
\end{proof}
\begin{lemma}\label{Eulersummands}
Let $X$ be projective. 
\begin{enumerate}
  \item For every $a\in J_0$ the Euler characteristic of the invariants of $K_0(a)$ is given by
\[
 \chi\left(K_0(a)^{\sym_{[\max a+1, n]}}\right)=\prod_{m=1}^{\max a}\chi\left(\bigotimes_{t\in a^{-1}(m)}E_t\right)\cdot \binom{\chi(\reg_X)+n-\max a-1}{n-\max a}\,.
\]
  \item For every $(M;a)\in \hat J_\ell$ the Euler characteristic $\chi(T_\ell(M;a)^{\sym_{[\max(a,2)+1, n]}})$ is given by
\[
 \chi\left(S^{\ell-1}\Omega_X\otimes\bigotimes_{t\in \hat M (a)}E_t\right)\cdot\prod_{m=3}^{\max a}\chi\left(\bigotimes_{t\in a^{-1}(m)}E_t\right)\cdot \binom{\chi(\reg_X)+n-\max (a,2)-1}{n-\max (a,2)}\,.
\]
   \end{enumerate}
\end{lemma}
\begin{proof}
 This follows from the previous lemma and lemma \ref{Euler}.
\end{proof}

%% file: globalsec.tex
\subsection{The map $\phi_1$ on cohomology and the cup product}\label{phioncoh}
We consider the morphism $\phi_1\colon K_0\to T_1$ defined in subsection \ref{constr} and for $a\in I_0$ and $(\{t\};i,j;b)\in I_1$ its components
\[\phi_1(a\to(\{t\};i,j;b))\colon K_0(a)\to T_1(\{t\};i,j;b)\,.\]
The morphism $\phi_1(a\to(\{t\};i,j;b))$ is non-zero only if $a_{|[k]\setminus\{t\}}=b$ and $a(t)\in \{i,j\}$. In this case it is given by $\eps_{a(t),\{i,j\}}$ times the morphisms given by restricting sections to the pairwise diagonal $\Delta_{ij}$. 
For two sheaves $F,G\in \Coh(X)$ (or more generally two objects in $\D^b(X)$) the composition
\[\Ho^*(X,F)\otimes \Ho^*(X,G)\cong \Ho^*(X^2,F\boxtimes G)\to \Ho^*(X,F\otimes G)\]
of the K\" unneth isomorphism and the map induced by the restriction to the diagonal equals the cup product. 
Thus the map 
$H^*(\phi_1(a,(\{t\};i,j;a_{|[k]\setminus \{t\}})))$ is given in terms of the natural isomorphisms of lemma \ref{kunneth} by
sending a
\[v_1\otimes\dots \otimes v_n\in \Ho^*\left(\bigotimes_{t\in a^{-1}(1)}E_t\right)\otimes \dots \otimes \Ho^*\left(\bigotimes_{t\in a^{-1}(n)}E_t\right)\]
to the class
\[(v_i\cup v_j)\otimes v_1\otimes\dots\otimes \hat v_i\otimes \dots \otimes\hat v_j\otimes \dots\otimes v_n\in  \Ho^*\left(\bigotimes_{t\in a^{-1}(\{i,j\})}E_t\right)\otimes \bigotimes_{m\in [n]\setminus\{i,j\}} \Ho^*\left(\bigotimes_{t\in a^{-1}(m)}E_t\right)\,.\]
Remember (theorem \ref{Sca} (i)) that there are the augmentation morphisms $\gamma_i\colon p_*q^*E_i^{[n]}\to C^0_{E_i}$ and that $K_0=\otimes_{i=1}^k C^0_{E_i}$.
We consider the composition
\[\bigotimes_{i=1}^k\Ho^*(X^n,p_*q^*E_i^{[n]})\xrightarrow{\cup} \Ho^*(X^n,\bigotimes_{i=1}^kp_*q^*E_i^{[n]})\xrightarrow{\Ho^*(\otimes_i\gamma_i)}\Ho^*(X^n,K_0)\,.\]
Taking (factor-wise) the $\sym_n$-invariants we get the map (see formula (1) of the introduction)
\[\Psi\colon \bigotimes_{i=1}^k\Ho^*(X^{[n]},E_i^{[n]})\cong \bigotimes_{i=1}^k\left(\Ho^*(E_i)\otimes S^{n-1}\Ho^*(\reg_X)\right)\to \Ho^*(X^n,K_0)^{\sym_n}\,.\]
This map coincides with the $\sym_n$-invariant cup product $\otimes_i\Ho^*(C^0_{E_i})^{\sym_n}\to \Ho^*(\otimes_iC^0_{E_i})^{\sym_n}$. 
The inclusion $p_*q^*(\otimes_i E_i^{[n]})^{\sym_n}\subset K_0^{\sym_n}$ induces a map $\Ho^*(X^{[n]},\otimes_i E_i^{[n]})\to \Ho^*(X^n,K_0)^{\sym_n}$. Since $\im(\otimes_i\gamma_i)= \otimes_ip_*q^*E_i^{[n]}$ (proposition \ref{im}), the image of $\Psi$ is a subset of the image of this map. 
In degree zero the map $\Ho^0(X^{[n]},\otimes_i E_i^{[n]})\to \Ho^0(X^n,K_0)^{\sym_n}$ is a inclusion. Thus, we have $\im(\Psi)\subset \Ho^0(X^{[n]},\otimes_iE_i^{[n]}).$  
Let $X$ be projective. 
In this case $\Ho^0(\reg_X)=\langle 1\rangle\cong \C$ where $1$ is the function with constant value one.
Thus, we have for $a\in I_0$ the formula (see lemma \ref{kunneth})
\[
 \Ho^0(X^n,K_0(a))\cong\bigotimes_{m\in \im(a)}\Ho^0\left(\bigotimes_{r\in a^{-1}(m)} E_r \right)
\]
and the action of $\Stab_{\sym_n}(a)=\sym_{\overline{\im a}}$ on this vector spaces is the trivial one, which means $\Ho^0(X^n,K_0(a))^{\sym_{{\overline {\im a}}}}=\Ho^0(X^n,K_0(a))$.
Now, for \[x_i\in \Ho^0(E_i^{[n]})\cong \Ho^0(E_i)\otimes S^{n-1}\Ho^0(\reg_X)\cong \Ho^0(E_i)\] and $a\in J_0(1)$ we have
\begin{align}\label{psiformula}\Psi(x_1\otimes\dots\otimes x_k)(a)=\bigotimes_{m\in \im(a)}(\cup_{r\in a^{-1}(m)}x_r)\,.\end{align}
\section{Cohomology in the highest and lowest degree}
\subsection{Cohomology in degree $2n$}
Let for $\ell\in[k]$ be $B_\ell:=\im(\phi_\ell)\subset T_\ell$, i.e. we have exact sequences
\begin{align*}0\to K_\ell\to\K_{\ell-1}\to B_\ell\to 0 \,.\tag 1\end{align*}
Since $T_\ell$ is the push-forward of a sheaf on $\mathbbm D$, the subsheaf $B_\ell$ is, too.
Since $\dim \mathbbm D=2(n-1)$ we have $\Ho^i(X^n,B_\ell)=0$ for $i=2n-1,2n$. By the long exact sheaf cohomology sequence associated to (1)
\[\dots\to \Ho^{2n-1}(B_\ell)\to \Ho^{2n}(K_\ell)\to \Ho^{2n}(K_{\ell-1})\to \Ho^{2n}(B_\ell)\to 0\]
we see that $\Ho^{2n}(K_\ell)=\Ho^{2n}(K_{\ell-1})$. By induction we get $\Ho^{2n}(K_\ell)=\Ho^{2n}(K_{0})$. Using corollary \ref{mucor} and lemma \ref{kunneth} this yields the following formula for the cohomology of tensor products of tautological bundles in the maximal degree.
\begin{prop}\label{maxcoh}
\[\Ho^{2n}(X^{[n]}, E_1^{[n]}\otimes \dots\otimes E_k^{[n]})\cong \Ho^{2n}(X^n,K_0)^{\sym_n}\cong \bigoplus_{a\in J_0}\bigotimes_{r=1}^{\max a}\Ho^2(\bigotimes_{t\in a^{-1}(r)} E_t)\otimes S^{n-\max a}\Ho^2(\reg_X)\,. \]
\end{prop}
\begin{remark}
In the case that $X$ is projective, the above proposition is the same as \cite[Remark 6.22]{Kru} by Serre duality.
\end{remark}
\subsection{Global sections for $n\ge k$ and $X$ projective}
In this subsection we assume that $X$ is projective. We will generalise the formula given in \cite{Danglob} for the global sections of tensor powers of a tautological sheaf associated to a line bundles to a formula for tensor products of arbitrary tautological bundles.   
\begin{lemma}
 Let $a\colon [k]\to [n]$, $t\in [k]$, $i=a(t)$, and $j\in [n]\setminus\im(a)$. Then under the natural isomorphisms of lemma \ref{kunneth} the map
$\Ho^0(\phi_1)(a\to(\{t\};\{i,j\};a_{|[k]\setminus\{t\}}))$ corresponds to $\eps_{a(t),\{i,j\}}$ times the identity on
$\otimes_{m\in \im(a)}\Ho^0(\otimes_{r\in a^{-1}(m)} E_r )$.   
\end{lemma}
\begin{proof}
 This follows by the formula for $\Ho^*(\phi_1)(a\to(\{t\};\{i,j\};a_{|[k]\setminus\{t\}}))$ of subsection \ref{phioncoh} and the fact that $v\cup 1=v$ for every $v\in \Ho^0(\otimes_{r\in a^{-1}(i)} E_r)$.
\end{proof}
\begin{lemma}
 Let $m=\min(n,k)$. Then every $s\in \ker \Ho^0(\phi_1)$ is determined by its components $s(a)\in \Ho^0(K_0(a))$ for $a\in I_0$ with $|\im(a)|=m$.
\end{lemma}
\begin{proof}
We use induction over $w:=m-\im(b)$ with the hypothesis that $s(b)$ is determined by the values of the $s(a)$ with $|\im(a)|=m$. Clearly, the hypothesis is true for $w=0$. So now let $b\colon [n]\to [k]$ with $|\im(b)|<\min(n,k)$. Such a map $b$ is neither injective nor surjective. Thus, we can choose  $j\in [n]\setminus\im(a)$ and a pair $t,t'\in[k]$ with $t\ne t'$ and $i:=b(t)=b(t')$. We define $\tilde b\colon [k]\to [n]$ by $\tilde b_{|[k]\setminus \{t\}}= b_{|[k]\setminus \{t\}}$ and $\tilde b(t)=j$. Then $\im(\tilde b)=\im(b)\cup \{j\}$ which gives $|\im(\tilde b)|=|\im(b)|+1$. We have
\begin{align*}
 0&=\Ho^0(\phi_1)(s)(\{t\};\{i,j\};b_{|[k]\setminus\{t\}})\\
 &= \Ho^0(\phi_1)(b\to(\{t\};\{i,j\};b_{|[k]\setminus\{t\}}))(s(b)) +\Ho^0(\phi_1)(\tilde b\to(\{t\};\{i,j\};b_{|[k]\setminus\{t\}}))(s(\tilde b))   \,.
\end{align*}
Since $\Ho^0(\phi_1)(b\to(\{t\};\{i,j\};b_{|[k]\setminus\{t\}}))$ is an isomorphism by the previous lemma, $s(b)$ is determined by $s(\tilde b)$ which in turn is determined by the values of $s(a)$ with $|\im(a)|=m$ by the induction hypothesis.    
\end{proof}
Since the functor of taking global sections is left-exact, the inclusions 
\[p_*q^*(E_1^{[n]}\otimes \dots\otimes E_k^{[n]})=K_k\subset K_1\subset K_0\]
induce inclusions
\[\Ho^0(X^{[n]},E_1^{[n]}\otimes \dots\otimes E_k^{[n]})\subset \Ho^0(X^n,K_1)^{\sym_n}\subset \Ho^0(X^n,K_0)^{\sym_n}\,.\]
Furthermore, $\Ho^0(X^n,K_1)=\ker(\Ho^0(\phi_1))$ holds. 
\begin{lemma}\label{phiker}
Let $n\ge k$. Then the projection
\[\Ho^0(X^n,K_0)^{\sym_n}\cong \bigoplus_{a\in J_0}\Ho^0(X^n,K_0(a))\to \Ho^0(X^n, K_0(1,2,\dots,k))\]
induces an isomorphism 
\[\Ho^0(X^n,K_1)^{\sym_n}\xrightarrow\cong \Ho^0(X^n,K_0(1,2,\dots,k))\]
 as well as an isomorphism
\[\Ho^0(X^{[n]},E_1\otimes \dots\otimes E_k)\xrightarrow\cong \Ho^0(K_0(1,2,\dots,k))\,.\]  
\end{lemma}
\begin{proof}
 By the previous lemma, the map $\Ho^0(K_1)^{\sym_n}\to \Ho^0(K_0(1,2,\dots,k))$ is injective. Thus it is left to show that for each 
\[x=x_1\otimes \dots\otimes x_k\in \Ho^0(X^n,K_0(1,2,\dots,k))=\Ho^0(E_1)\otimes \dots\otimes \Ho^0(E_k)\]
there exists an $s\in \Ho^0(X^{[n]},E_1^{[n]}\otimes \dots \otimes E_k^{[n]})\subset \Ho^0(X^n,K_1)^{\sym_n}\subset \Ho^0(X^n,K_0)^{\sym_n}$ with $s(1,2,\dots,k)=x$. We can consider each $x_t$ as a section of $\Ho^0(X^{[n]},E_t^{[n]})\cong \Ho^0(E_t)$. Then by formula (\ref{psiformula}) of the previous section $s=\Psi(x_1\otimes \dots\otimes x_n)$ has the desired property. 
\end{proof}
\begin{theorem}\label{globsec}
For $n\ge k$ there is a natural isomorphism
\[\Ho^0(X^{[n]},E_1^{[n]}\otimes\dots\otimes E_k^{[n]})\cong \Ho^0(E_1)\otimes \dots\otimes \Ho^0(E_k)\,.\]
\end{theorem}
\begin{proof}
 This follows from the lemmas \ref{phiker} and \ref{kunneth}.
\end{proof}

%% file: longtwo.tex
\section{Tensor products of tautological bundles on $X^{[2]}$ and $X^{[n]}_{**}$}\label{long}
\subsection{Long exact sequences on $X^2$}\label{longzwei}
We want to enlarge the exact sequences \[0\to K_\ell\to K_{\ell-1}\xrightarrow{\phi_\ell} T_\ell\]
to long exact sequences with a $0$ on the right. We will do this first on $X^2$. Since the pairwise diagonals are disjoint on $X^n_{**}$, long exact sequences on $X^n_{**}$ can be obtained later from this case.
 We denote the diagonal in $X^2$ by $\Delta$ its inclusion by $\delta\colon X\to X^2$, and its vanishing ideal by $\mathcal I$. For a set $M=\{t_1<\dots<t_s\}\subset [k]$ of cardinality $s$ we will consider the standard representation $\rho_M\cong \rho_s$ of $\sym_M\cong \sym_s$ as the subspace of $\rho_k\subset \C^k$ with basis
\[\zeta_{M}^1:=e_{t_1}-e_{t_2}\,,\,\zeta_{M}^2:=e_{t_2}-e_{t_3}\,,\,\dots\,,\,\zeta_{M}^{s-1}:=e_{t_{s-1}}-e_{t_s}\,.\]
For $M\subset N$ we denote the inclusion by $\iota_{M\to N}\colon \rho_M\to \rho_N$ but will also often omit it in the following. For $\ell=1,\dots,k$ and $i=0,\dots , k-\ell$ we set
\[I_\ell^i:=\bigl\{(M;a)\mid M\subset [k]\,,\, \#M=\ell+i\,,\, a\colon [k]\setminus M\to[2]\bigr\}\quad,\quad R_\ell^i:=\bigoplus_{(M;a)\in I_\ell^i}\wedge^{\ell-1}\rho_M(a)\]
where $\rho_M(a)=\rho_M$ for every $a$. We define differentials $d_\ell^i\colon R_\ell^i\to R_\ell^{i+1}$ for $s\in R_\ell^i$ by
\[d_\ell^i(s)(M;a):=\sum_{i\in M}\eps_{i,M}\iota_{M\setminus\{i\}\to M}\left(s(M\setminus\{i\};a,i\mapsto 1)-s(M\setminus\{i\};a,i\mapsto 2)\right)\,.\]
We have indeed defined complexes $R_\ell^\bullet$ for $\ell=1\dots,k$ since
\begin{align*}
 (d\circ d)(s)(M;a)
=& \sum_{i\in M}\eps_{i,M}\iota_{M\setminus\{i\}\to M}\bigl(d(s)(M\setminus\{i\};a,i\mapsto 1)-d(s)(M\setminus\{i\};a,i\mapsto 2)\bigr)\\
=& \sum_{i\in M}\sum_{j\in M\setminus \{i\}}\eps_{i,M}\eps_{j,M\setminus\{i\}}\iota_{M\setminus\{i,j\}\to M}\left(\sum_{b\colon\{i,j\}\to[2]}\eps_b s(M\setminus\{i,j\};a\uplus b)\right)
\end{align*}
vanishes by the fact that $\eps_{i,M}\eps_{j,M\setminus\{i\}}=-\eps_{j,M}\eps_{i,M\setminus\{j\}}$ for all $i,j\in M$.
We define a $\sym_k$-action on every $R_\ell^i$ by setting
\[(\sigma\cdot s)(M;a):=\eps_{\sigma,\sigma^{-1}(M)}\sigma\cdot s(\sigma^{-1}(M);a\circ\sigma)\,.\]
The $\sym_k$-action on the right-hand side is the exterior power of the action on $\rho_k$. It maps indeed $\wedge^{\ell-1}\rho_{\sigma^{-1}(M)}$ to $\wedge^{\ell-1}\rho_M$.
This makes each $R_\ell^\bullet$ into a $\sym_k$-equivariant complex, since for $i\in M$ the term
\[s(\sigma^{-1}(M\setminus\{i\});a\circ\sigma,\sigma^{-1}(i)\mapsto 1) - s(\sigma^{-1}(M\setminus\{i\});a\circ\sigma,\sigma^{-1}(i)\mapsto 2)\]
occurs in $(\sigma\cdot d(s))(M;a)$ with the sign $\eps_{\sigma^{-1}(i),\sigma^{-1}(M)}\cdot\eps_{\sigma,\sigma^{-1}(M)}$ and in $d(\sigma\cdot s)(M;a)$ with the sign $\eps_{i,M}\cdot \eps_{\sigma,\sigma^{-1}(M\setminus\{i\})}$. Both signs are equal by \ref{signlemma}.
Note that for $(M;a)\in I_\ell^0$ we have $\wedge^{\ell-1}\rho_M(a)\cong \C$. We will denote the canonical base vector $\zeta_M^1\wedge\dots\wedge \zeta_M^{\ell-1}$ by $e_{(M;a)}$.
We also define
\[R_\ell^{-1}:=\bigoplus_{a\colon [k]\to [2]} \C(a)\quad,\quad \C(a)=\C\]
together with the $\sym_k$-equivariant map $\tilde\phi_\ell=d_{\ell}^{-1}\colon R^{-1}_\ell\to R^0_\ell$ given by
\[\tilde \phi_\ell(s)(M;a)=\left(\sum_{b\colon M\to [2]}\eps_bs(a\uplus b)\right)\cdot e_{(M;a)}\,.\]
The $\sym_k$-action on $R_\ell^{-1}$ is given by $(\sigma\cdot s)(a)=s(\sigma^{-1}\circ a)$.
We set
\[\tilde R_\ell^\bullet :=\bigl(0\to R_\ell^{-1}\to R^\bullet_\ell\bigr)=\bigl(0\to R_\ell^{-1}\to R_\ell^0\to\dots\to R_\ell^{k-\ell}\to 0\bigr)\,.\]
We make this complex also $\sym_2$-equivariant by defining the action of $\tau=(1\,\,2)$ in degree $-1$ by $(\tau\cdot s)(a):=a(\tau^{-1}\circ a)=a(\tau\circ a)$ and in degree $i\ge 0$ by \[(\tau\cdot s)(M;a):=\eps_{\tau,\{1,2\}}^{\ell+i} s(M;\tau^{-1}\circ a)=(-1)^{\ell+i} s(M;\tau\circ a)\,.\]
We will sometimes write a $k$ as a left lower index of the occurring objects and morphisms, e.g. ${_kR_\ell^i}$, if we want to emphasise a chosen value of $k$.
\begin{prop}\label{longtwoR}
For every $\ell=1,\dots,k$ the complex $\tilde R_\ell^\bullet$ is cohomologically concentrated in degree $-1$, i.e.
the sequence
\[R_\ell^{-1}\to R^0_\ell\to R_\ell^1\to\dots\to R_\ell^{n-\ell}\to 0\]
is exact.
\end{prop}
\begin{proof}
 We will divide the proof into several lemmas. We will often omit certain indices in the notation, when we think that it will not lead to confusion.
For $\ell=1$ the complex
$\tilde R_1^\bullet$
is isomorphic to $(\tilde C^\bullet)^{\otimes k}[1]$, where $\tilde C^\bullet$ is the complex concentrated  in degree 0 and 1 given by
\[0\to \C\oplus \C\to \C\to 0\quad,\quad \binom ab\mapsto a-b\,.\]
Since the complex $\tilde C^\bullet$ has only cohomology in degree zero and the tensor product is taken over the field $\C$, it follows that $\tilde R_1^\bullet$ is indeed cohomologically concentrated in degree $-1$. We go on by induction over $\ell$ assuming that the proposition is true for all values smaller than $\ell$.
\begin{lemma}\label{tinker}
 Let $t\in R_\ell^0$. Then $d_\ell^0(t)=0$ if and only if for every $(N;a)\in I_\ell^1$ and every pair $i,j\in N$ the following holds:
\[t(N\setminus\{i\};a,i\mapsto 1)-t(N\setminus\{i\};a,i\mapsto2)=t(N\setminus\{j\};a,j\mapsto 1)-t(N\setminus\{j\};a,j\mapsto2)\,.\]
Here for $(M;b)\in I_\ell^0$ we use the notation
$t(M;b)=t(M;b)\cdot e_{(M;b)}$,
i.e. we denote by $t(M;b)$ also its preimage under the canonical isomorphism $\C\cong \wedge^{\ell-1}\rho_{M}$.
\end{lemma}
\begin{proof}
Let $N=\{n_1<\dots<n_{\ell+1}\}\subset [k]$. The above formula holds for every pair $i,j\in N$ if and only if it holds for every pair of neighbors. Thus we may assume that $i=n_h$ and $j=n_{h+1}$ with $h\in [\ell]$. The wedge product $\wedge^{\ell-1}\rho_{N\setminus\{n_h\}}$ is spanned by the vector
\[e_{N\setminus\{n_h\}}=\zeta_{N\setminus\{n_h\}}^1\wedge\dots\wedge\zeta_{N\setminus\{n_h\}}^{\ell-1}=\begin{cases}
                                                                                \zeta_{N}^2\wedge\dots\wedge\zeta_{N}^\ell&\text{ for $h=1$,}\\
                                                                           \zeta_{N}^1\wedge\dots\wedge(\zeta_{N}^{h-1}+\zeta_{N}^h)\wedge\dots\wedge\zeta_{N}^\ell&\text{ else.}
                                                                               \end{cases}
\]
Thus, for $t\in R_\ell^0$ the coefficient of $\zeta_{N}^1\wedge\dots\wedge\widehat{\zeta_{N}^h}\wedge \dots\wedge \zeta_{N}^\ell$ of $d(t)(N,a)\in \wedge^{\ell-1}\rho_N$ equals
\begin{align*}&\eps_{n_h,N}\bigl(t(N\setminus\{n_h\};a,n_h\mapsto 1)-t(N\setminus\{n_h\};a,n_h\mapsto 2)\bigr)\\+&
 \eps_{n_{h+1},N}\bigl(t(N\setminus\{n_{h+1}\};a,n_{h+1}\mapsto 1)-t(N\setminus\{n_{h+1}\};a,n_{h+1}\mapsto 2)\bigr)
\end{align*}
which proves the lemma.
\end{proof}
The inclusion $\im(\tilde\phi)\subset \ker(d_\ell^0)$ follows since for $s\in R^{-1}_\ell$ and $t=\tilde\phi(s)$ both sides of the equation in the above lemma equal
\[\sum_{b\colon N\to [2]}\eps_b\cdot s(a\uplus b)\,.\]
We will actually show a bit more than $\im (\tilde\phi)\supset \ker(d_\ell^0)$ in the next lemma.
We decompose $R_\ell^{-1}$ into $U_\ell\oplus S_\ell$ with
\begin{align*}
U_\ell&=\{s\in R_\ell^{-1}\mid s(a)=0\,\forall a:\, |a^{-1}(\{2\})|\le \ell-1\}=\langle e_a\mid a^{-1}(\{2\})\ge\ell\rangle \\
S_\ell&=\{s\in R_\ell^{-1}\mid s(a)=0\,\forall a:\, |a^{-1}(\{2\})|\ge \ell\}=\langle e_a\mid a^{-1}(\{2\})\le\ell-1\rangle\,. 
\end{align*}
\begin{lemma}\label{Ul}
 The map $\tilde\phi_{\ell\mid U_\ell}\colon U_\ell \to \ker(d_\ell^0)$ is an isomorphism. 
\end{lemma}
\begin{proof}
We first show the injectivity. Let $s\in U$ with $\tilde\phi(s)=0$. We show that $s(a)$ is zero for every $a$ by induction over $\alpha=|a^{-1}(\{2\})|$. For $\alpha=\ell$ we set $M=a^{-1}(\{2\})$ and get
\[0=\tilde\phi(s)(M;\underline 1)=s(a)\,.\]
We may now assume that $s(b)=0$ for every $b\colon [k]\to[2]$ with $b^{-1}(\{2\})<\alpha$. Then for $M\subset a^{-1}(\{2\})$ with $|M|=\ell$ we obtain
\[0=\tilde\phi(s)(M;a_{|[k]\setminus M})=s(a)\,.\]
For the surjectivity we precede by induction on $k$. For $k=\ell$ the map $\tilde\phi$ sends the basis vector $e_{\underline 2}$ of the one-dimensional vector space $U$ with a factor of $(-1)^\ell$ to the basis vector $e_{[\ell]}$ of the one-dimensional vector space $T_\ell^0$. Let now be $k>\ell$. 
We set
\begin{align*}&V_0=\{a\colon [k]\to [2]\mid a(k)=2,\,\#a^{-1}(\{2\})=\ell\},\\ &V_1=\{a\colon [k]\to [2]\mid a(k)=1,\,\#a^{-1}(\{2\})\ge\ell\},\\ &V_2=\{a\colon [k]\to [2]\mid a(k)=2,\,\#a^{-1}(\{2\})\ge \ell+1\},\\ &W_0=\{(M,a)\in {_kI_\ell^0}\mid k\in M,\,a=\underline1\}\,,
\, W_1=\{(M,a)\in {_kI_\ell^0}\mid k\notin M,\,a(k)=1\},\\&W_2=\{(M,a)\in {_kI_\ell^0}\mid k\notin M,\,a(k)=2\}\,,\,  W_3=\{(M,a)\in {_kI_\ell^0}\mid k\in M,\,a\ne\underline1\}, 
\end{align*}    
We define subspaces of $U$ respectively $R_\ell^0$ by 
\[\langle V_i\rangle=\langle e_a\mid a\in V_i\rangle\quad,\quad \langle W_i\rangle=\langle e_{(M;a)}\mid (M;a)\in W_i\rangle\]
which yields $U=\langle V_0\rangle\oplus\langle V_1\rangle\oplus\langle V_2\rangle$ and $T_\ell^0=\langle W_0\rangle\oplus\langle W_1\rangle\oplus\langle W_2\rangle\oplus\langle W_3\rangle$.
We denote by $\tilde\phi(i,j)$ the component of $\tilde\phi_\ell$ given by the composition  
\[\langle V_i\rangle\to U\xrightarrow{\tilde\phi} R_\ell^0\to \langle W_j\rangle\,.\]
Let now $t\in \ker({_kd^0_\ell})$. We define $s_0\in \langle V_0\rangle$ by $s_0(a)=t(a^{-1}(\{2\}); \underline 1)$ for $a\in V_0$. This yields $\tilde\phi(0,0)(s_0)= t_{|\langle W_0\rangle}$. We set $\tilde t:= t-\tilde\phi(s_0)$. The components $\tilde\phi(1,0)$, $\tilde\phi(2,0)$, $\tilde\phi(1,2)$, and $\tilde\phi(2,1)$ are all zero. There are canonical bijections 
\begin{align*}
& V_1\xrightarrow{\cong}\{a\colon [k-1]\to [2]\mid \#a^{-1}(\{2\})\ge\ell\}\xleftarrow{\cong} V_2\,,\\
&W_1\xrightarrow{\cong}\{(M,a)\mid M\subset[k-1],\, \#M=\ell,\, a\colon [k-1]\setminus M\to[2]\}\xleftarrow{\cong} W_2
\end{align*}
 given by dropping $a(k)$. They induce linear isomorphisms $\langle V_1\rangle\cong {_{k-1}U_{\ell}}\cong \langle V_2\rangle$ as well as $\langle W_1\rangle\cong{_{k-1}T_\ell^0}\cong \langle W_2\rangle$ under which the linear maps $\tilde\phi(1,1)$ and $\tilde\phi(2,2)$ both correspond to $_{k-1}\tilde\phi$.
Thus, by induction there are $s_1\in \langle V_1\rangle$ and $s_2\in \langle V_2\rangle$ such that $\tilde\phi(1,1)(s_1)=\tilde t_{|\langle W_1\rangle}$
and $\tilde\phi(2,2)(s_2)=\tilde t_{|\langle W_2\rangle}$. Defining $s\in {_kU}$ by  $s_{|\langle V_i\rangle}=s_i$ for $i=0,1,2$ we get 
\[\tilde \phi(s)_{|\langle W_0\rangle\oplus \langle W_1\rangle\oplus \langle W_2\rangle}=t_{|\langle W_0\rangle\oplus \langle W_1\rangle\oplus \langle W_2\rangle}\,.\] It is left to show that the equation also holds on $\langle W_3\rangle$. For this we show that for every $x\in \ker d_\ell^0$ with $x_{|\langle W_0\rangle\oplus \langle W_1\rangle\oplus \langle W_2\rangle}=0$ also $x(M;a)=0$ for every $(M,a)\in W_3$ holds. We use induction over $\alpha = |a^{-1}(\{2\})|$. 
For $\alpha=0$ the tuple $(M;a)$ is an element of $W_0$. Hence, $x(M;a)=0$ holds. We now assume that $x(M;b)=0$ for every tuple $(M;b)\in W_3$ with $|b^{-1}(\{2\})|<\alpha=|a^{-1}(\{2\})|$ holds and choose an $i\in [k]\setminus M$ with $a(i)=2$. Applying lemma \ref{tinker} to $N=M\cup \{i\}$ we get
\[
x(M;a)=x(M;a_{|[k]\setminus N},i\mapsto 1)-x(N\setminus\{k\};a_{|[k]\setminus N},k\mapsto 1)+x(N\setminus\{k\};a_{|[k]\setminus N},k\mapsto 2)\,.
\]
The first term on the right hand side is zero by induction hypothesis and the second and third are zero since they are coefficients of $s$ in $W_1$ respectively $W_2$.
\end{proof}
\begin{remark}\label{countdim}
Counting the cardinality of the base we get 
\[\dim({_kU_\ell})=\sum_{j=\ell}^k\binom kj\,.\]
On the other hand we have 
\[\dim({_kR_\ell^i})=2^{k-\ell-i}\binom k{\ell+i}\binom{\ell+i-1}{\ell-1}\]
and thus also 
\[\chi({_kR_\ell^\bullet})=\sum_{i=0}^{k-\ell}(-1)^i2^{k-\ell-i}\binom k{\ell+i}\binom{\ell+i-1}{\ell-1}\overset{\ref{binom}}= \sum_{j=\ell}^k \binom kj\,.\]
It follows by the previous lemma that if the complex $ R_\ell^\bullet$ is exact in all but one positive degree it is exact in every positive degree.  
\end{remark}
\begin{lemma}
Let $k\ge \ell+2$ and $2\le j\le k-\ell$. Then $\mathcal H^j(R_\ell^\bullet)=0$, i.e. the sequence
\[R_\ell^{j-1}\to R_\ell^{j}\to R_\ell^{j+1}\]
is exact (with $R_\ell^{j+1}=0$ in the case $j=k-\ell$).  
\end{lemma}
\begin{proof}
The term $R_\ell^{k-\ell}=\wedge ^{\ell-1}\rho_{[k]}$ is an irreducible $\sym_k$-representation (see \cite[Proposition 3.12]{FH}). It is still irreducible after tensoring with the alternating representation. Since $d_\ell^{k-\ell-1}$ is non-zero and $\sym_k$-equivariant, it follows that it is surjective. This proves the case $j=k-\ell$. For general $j\in [2,k-\ell]$ we use induction over $k$. For $k=\ell+2$ only the case $j=k-\ell$ occurs which is already proven. So now let $k\ge \ell+3$, $2\le j\le k-\ell-1$, and $t\in \ker d_\ell^j$. For $i=0,\dots,k-\ell$ we decompose $R_\ell^i$ into the direct summands
 \begin{align*}
  R_\ell^i(0)=\bigoplus_{(M,a):\, k\in M} \wedge^{\ell-1}\rho_M(a)\,&,\, R_\ell^i(1)=\bigoplus_{(M,a):\, k\notin M,\, a(k)=1} \wedge^{\ell-1}\rho_M(a)\\
R_\ell^i(2)=&\bigoplus_{(M,a):\, k\notin M,\, a(k)=2} \wedge^{\ell-1}\rho_M(a)
 \end{align*}
and denote the components of the differential by $d(u,v)\colon R_\ell^i(u)\to R_\ell^{i+1}(v)$. Then $d(0,1)$, $d(0,2)$, $d(1,2)$ and $d(2,1)$ all vanish. It follows that $t_1\in \ker d(1,1)$ and $t_2\in \ker d(2,2)$ where $t_u$ is the component of $t$  in $R_\ell^j(u)$. By dropping $a(k)$ we get isomorphisms
\[{_kR_\ell^i(1)}\cong  {_{k-1}R_\ell^i}\cong {_kR_\ell^i(2)}\]
under which $_kd_\ell^i(1,1)$ as well as $_kd_\ell^i(2,2)$ coincide with $_{k-1}d_\ell^i$. Thus, by induction there exist $s_u\in {_kR_\ell^{j-1}(u)}$ for $u=1,2$ such that $d(u,u)(s_u)=t_u$. It remains to find a $s_0\in R_\ell^{j-1}(0)$ such that \[d(s_0)=d(0,0)(s_0)=\hat t:= t-d(s_1+s_2)\,.\]
For $M\subset [k-1]$ we have $\rho_{M\cup\{k\}}=\rho_M\oplus\langle \zeta_{M\cup \{k\},\max}\rangle$, where $\max=|M|$, i.e. 
\[\zeta_{M\cup\{k\},\max}=\zeta_{M\cup\{k\},|M|}=e_{\max(M)}-e_k\,.\]
Using this we can decompose the $R_\ell^i(0)$ further into the two direct summands
\[R_\ell^i(3)=\bigoplus_{(M,a):\, k\in M} \wedge^{\ell-1}\rho_{M\setminus\{k\}}(a)\quad,\quad R_\ell^i(4)=\bigoplus_{(M,a):\, k\in M} \wedge^{\ell-2}\rho_{M\setminus\{k\}}\otimes \langle \zeta_{M,\max}\rangle(a)\,.\]
The differential $d(3,4)$ vanishes. Thus, $d(4,4)(\hat t_4)=0$. Furthermore we have the following isomorphism of sequences
\[\begin{CD}
 {_{k-1}R_{\ell-1}^{j-1}} @>{{_{k-1}d_{\ell-1}^{j-1}}}>> {_{k-1}R_{\ell-1}^{j}} @>{{_{k-1}d_{\ell-1}^{j}}}>> {_{k-1}R_{\ell-1}^{j+1}}  
\\
 @V{\cong}VV @V{\cong}VV @V{\cong}VV 
\\
 {_kR_\ell^{j-1}(4)}@>{{_{k}d_{\ell}^{j-1}}(4,4)}>> {_kR_\ell^{j}(4)}@>{{_{k}d_{\ell}^{j}}(4,4)}>> {_kR_\ell^{j+1}(4)} 
\end{CD}\]
induced by the maps ${_{k-1}I_{\ell-1}^{i}}\to {_{k}I_\ell^{i}}$ given by $(L;a)\mapsto (L\cup\{k\};a)$ and the isomorphisms $\wedge^{\ell-2}\rho_{M\setminus\{k\}}\cong \wedge^{\ell-2}\rho_{M\setminus\{k\}}\otimes \langle \zeta_{M,\max}\rangle$. By the induction hypothesis for $\ell$ (saying that proposition \ref{longtwoR} is true for $\ell-1$) the upper sequence is exact. Thus, there is a $s_4\in R_\ell^{j-1}(4)$ such that $d(4,4)(s_4)=\hat t_4$. We set $\tilde t_3:=\hat t_3-d(4,3)(s_4)$. There is also a canonical isomorphism of the exact sequences
\[\begin{CD}
 {_{k-1}R_{\ell}^{j-2}} @>{{_{k-1}d_{\ell}^{j-2}}}>> {_{k-1}R_{\ell}^{j-1}} @>{{_{k-1}d_{\ell}^{j-1}}}>> {_{k-1}R_{\ell}^{j}}  
\\
 @V{\cong}VV @V{\cong}VV @V{\cong}VV 
\\
 {_kR_\ell^{j-1}(3)}@>{{_{k}d_{\ell}^{j-1}}(3,3)}>> {_kR_\ell^{j}(3)}@>{{_{k}d_{\ell}^{j}}(3,3)}>> {_kR_\ell^{j+1}(3)} 
\end{CD}\]
induced by the maps ${_{k-1}I_\ell^{i-1}}\to {_{k}I_\ell^{i}}$ which are again given by $(L;a)\mapsto (L\cup\{k\};a)$. 
By induction over $k$ the upper sequence is exact (in the case that $j=2$ we have to use remark \ref{countdim}). This yields a $s_3\in R_\ell^{j-1}(3)$ such that $d(3,3)(s_3)=\tilde t_3$ holds. In summary we have found a 
$s=(s_1,s_2,s_3,s_4)\in R_\ell^{j-1}$ with $d(s)=t$. 
\end{proof}
As mentioned in remark \ref{countdim} the previous lemma finishes the proof of proposition \ref{longtwoR}. 
\end{proof}
For $\ell=1,\dots,k$ we set
\[H_\ell:=\delta_*\left(S^{\ell-1}\Omega_X \otimes E_1\otimes \dots\otimes E_k\right)=\left(S^{\ell-1}\Omega_X \otimes E_1\otimes \dots\otimes E_k\right)_{12}\,.\]
Then $H_\ell$ equals $T_\ell(M;1,2;a)$ (see section \ref{constr}) for every tuple $(M;a)\in I_\ell^0\cong {_2I_\ell}$.
We define the complexes $T_\ell^\bullet:=H_\ell\otimes_\C R_\ell^\bullet$ and $\tilde T_\ell^\bullet=H_\ell\otimes_\C \tilde R_\ell^\bullet$. We denote again the differentials by $d_\ell^i$ and $d_\ell^{-1}=\tilde \phi_\ell$.
Then for every $\ell=1,\dots,k$ there are isomorphisms $T_\ell\cong T_\ell^0$. We will always consider the isomorphism $T_\ell\cong T_\ell^0$ induced by the canonical isomorphisms $\C\cong \wedge^{\ell-1}\rho_M(a)$ for $(M;a)\in I_\ell^0$ given by
$1\mapsto e_{(M;a)}$.
We will denote the composition of $\phi_\ell\colon K_{\ell-1}\to T_\ell$ with this isomorphism again by $\phi_\ell\colon K_{\ell-1}\to T_\ell^0$. 
If $E_1=\dots=E_k$, the complex $\tilde T_\ell^\bullet$ carries a canonical $\sym_k$-linearization given by applying the action of $\sym_k$ on $\tilde R_\ell^\bullet$ and permuting the tensor factors of $H_\ell$. 
We also define a $\sym_2$-action on $\tilde T_\ell^\bullet$ by using the $\sym_2$-action on $\tilde R_\ell^\bullet$ as well as the natural action on $H_\ell$ which means that $\tau$ acts on $H_\ell$ by $(-1)^{\ell-1}$ because of the factor $(S^{\ell-1}\Omega_X)_{12}=S^{\ell-1}N_\Delta^\vee$.
The isomorphism $T_\ell\cong T_\ell^0$ is $\sym_k$-equivariant since
\[\sigma\cdot e_{(\sigma^{-1}(M),\sigma^{-1}\circ a)}=\sigma\cdot(\zeta_{\sigma^{-1}(M)}^1\wedge \dots\wedge \zeta_{\sigma^{-1}(M)}^{\ell-1})=\eps_{\sigma,\sigma^{-1}(M)}\cdot\zeta_{M}^1\wedge \dots\wedge \zeta_{M}^{\ell-1}=\eps_{\sigma,\sigma^{-1}(M)}\cdot e_{(M;a)}\,.\]
It is also $\sym_2$-equivariant. Hence, $\phi_\ell\colon K_{\ell-1}\to T_\ell^0$ is again $\sym_2$-equivariant and, if all the $E_i$ are equal, also $\sym_k$-invariant.
\begin{prop}\label{longtwo}
For every $\ell=1,\dots,k$ the sequences 
\[0\to K_\ell\to K_{\ell-1}\xrightarrow{\phi_\ell} T_\ell^0\to T_\ell^1\to\dots\to T_\ell^{k-\ell}\to 0\]
are exact. 
\end{prop}
\begin{proof}
With the same arguments as in the proof of proposition \ref{opendes} we may assume that $E_1=\dots=E_k=\reg_X$.
Thus we have \[H_\ell=\delta_*\left(S^{\ell-1}\Omega_X \otimes E_1\otimes \dots\otimes E_k\right)=\mathcal I^{\ell-1}/\mathcal I^{\ell}\,.\] A section $s\in K_0$ with arbitrary values $s(a_1,\dots,a_k)\in \mathcal I^{\ell-1}\subset \reg_{X^2}$ is already in the kernel of all the maps $\phi_1,\dots,\phi_{\ell-1}$, i.e. is a section of $K_{\ell-1}$. Thus, the image of $\phi_\ell\colon K_{\ell-1}\to T_\ell^0$ equals the image of $\tilde\phi_\ell\colon T_\ell^{-1}\to T^0_\ell$.
Hence, it suffices to show the exactness of 
\[T_{\ell}^{-1}\xrightarrow{\tilde\phi_\ell} T_\ell^0\to T_\ell^1\to\dots\to T_\ell^{k-\ell}\to 0\,.\]
Since the tensor product over $\C$ is exact, this follows from proposition \ref{longtwoR}.
\end{proof}
\subsection{Symmetric powers}
Let $E_1=\dots=E_k=E$ be a line bundle. Then the $\sym_k$-action on $H_\ell$ for $\ell=1,\dots,k$ given by permuting the tensor factors is the trivial one. Furthermore, the $\sym_k$-invariants of $R_\ell^i$ vanish for $i>0$. Indeed, $\wedge^{\ell-1}\rho_M$ tensorised by the alternating representation is an irreducible $\sym_M$ representation. For $|M|>\ell$ it is not the trivial one and hence has no $\sym_M$-invariants. By lemma \ref{tensinv} also the invariants of $T_\ell^i=H_\ell\otimes R_\ell^i$ vanish for $i>0$. Thus, by proposition \ref{longtwo} there are on $X^2$ short exact sequences 
\[0\to K_\ell^{\sym_k}\to K_{\ell-1}^{\sym_k}\to T_\ell^{\sym_k}\to 0\,.\]
By remark \ref{permute} we have 
\[\Phi(S^kE^{[2]})=\Phi(((E^{[2]})^{\otimes k})^{\sym_k})=\Phi((E^{[2]})^{\otimes k})^{\sym_k}=K_k^{\sym_k}\,.\]
Thus, all the computations below can also be done after taking $\sym_k$-invariants in order to get (simpler) formulas for the Euler characteristics of symmetric products instead of tensor products.

%% file: longopen.tex
\subsection{Long exact sequences on $X^{[n]}_{**}$}\label{longopensec}
Let $k,n\in \N$ and $E_1,\dots,E_k$ be locally free sheaves on $X$. 
For $\ell=1,\dots,k$, $1\le i<j\le n$, $\hat M\subset[k]$, and $a\colon [k]\setminus \hat M\to [n]\setminus\{i,j\}$ we define the complexes 
\[ T_\ell^\bullet[\hat M ;i,j;a]:= \pr_{ij}^*T_\ell^\bullet(\{E_s\}_{s\in \hat M}) \otimes \bigotimes_{t\in [k]\setminus\hat M}\pr_{a(t)}^*E_t\,.\]
Here $\pr_{ij}\colon X^n\to X^2$ is the projection to the $i$-th and $j$-th factor and $T_\ell^\bullet(\{E_s\}_{s\in \hat M})$ is the complex defined in subsection \ref{longzwei} with $E_1,\dots,E_k$ replaced by $\{E_s\}_{s\in\hat M}$.
We furthermore set
\[T_\ell^\bullet=\bigoplus_{\substack{\hat M\subset [k]\,,\,1\le i<j\le n\\ a\colon [k]\setminus \hat M\to [n]\setminus\{i,j\}}}T_\ell^\bullet[\hat M;i,j;a]\,.\]
This complex can be equipped with a natural $\sym_n$-action and $T_\ell^0$ can be identified with $T_\ell$. By proposition \ref{longtwo} and the fact that the pairwise diagonals are disjoint on $X^n_{**}$ the sequences
\[0\to K_{\ell**}\to K_{\ell-1**}\to T_{\ell**}^0\to\dots\to T_{\ell**}^{k-\ell}\to 0\]
are exact for every $\ell=1,\dots,k$.

%% file: invatwo.tex
\subsection{The invariants on $S^2X$}
For $\ell=1,\dots,k$ we define the complex $\hat R_\ell^\bullet$ of $\sym_2\times \sym_k$-representations by setting $\hat R_\ell^\bullet=R_\ell^\bullet$ as $\sym_k$-representations and defining the $\sym_2$-action by letting $\tau=(1\,\,2)$ act on $\hat R_\ell^i$ by \[(\tau\cdot s)(M;a)=(-1)^{\ell+i+\ell-1}s(M;\tau^{-1}\circ a)=(-1)^{i-1}s(M;\tau\circ a)\,.\] In this way we have $T_\ell^\bullet=H_\ell\otimes_\C\hat R_\ell^\bullet$ as $\sym_2\times \sym_k$-equivariant complexes, when considering $H_\ell$ equipped with the trivial action.
The action of $\tau$ on the index set $I_\ell^i$ of the  direct sum $\hat R_\ell^i$ is given by $\tau\cdot(M;a)=(M;\tau\circ a)$.
We define for $\ell\le k$ the number $N(k,\ell)$ by
\[N(k,\ell)=             \frac12 \left(\sum_{j=\ell}^k \binom kj -\binom{k-1}{\ell-1}\right) \,.  
\]           
Note that $N(k,k)=0$.
\begin{lemma}\label{Nkl}
The following holds for two natural numbers $\ell\le k$:
 \[\dim\left(\ker({_k\hat d_\ell^0})^{\sym_2}\right)=\dim\left(\mathcal H^0({_k\hat R_\ell^\bullet})^{\sym_2}\right)= \chi\left(({_k\hat R_\ell^\bullet})^{\sym_2}\right) =N(k,\ell)\,.\]
\end{lemma}
\begin{proof}
The second equality follows by proposition \ref{longtwoR} and the fact that taking invariants is exact.
The only non-trivial element $\tau=(1\,\,2)$ of $\sym_2$ acts freely on $I_\ell^i$ for $0\le i<k-\ell$. Thus, by Danila's lemma $\dim((R_\ell^i)^{\sym_2})=1/2\dim(R_\ell^i)$ for $i<k-\ell$. Furthermore, $\tau$ acts by $(-1)^{k+\ell-1}=(-1)^{k-\ell-1}$ on $\hat R_\ell^{k-\ell}$. Hence we have $(\hat R_\ell^{k-\ell})^{\sym_2}=\hat R_\ell^{k-\ell}$ if $k-\ell$ is odd and $(\hat R_\ell^{k-\ell})^{\sym_2}=0$ if $k-\ell$ is even. The assertion follows using remark \ref{countdim}.  
\end{proof}
\begin{prop}\label{symtwoexact}
On $S^2X$ there are for $\ell=1,\dots,k$ exact sequences
\[0\to K_\ell^{\sym_2}\to K_{\ell-1}^{\sym_2}\to (\pi_*H_\ell)^{\oplus N(k,\ell)}\to 0\,.\]
In particular $K_k^{\sym_2}=K_{k-1}^{\sym_2}$. 
\end{prop}
\begin{proof}
By proposition \ref{longtwo} we have $\im(\phi_\ell)=\mathcal H^0(T_\ell^\bullet)$. Since the tensor product over $\C$ as well as taking $\sym_2$-invariants are exact functors and we consider $H_\ell$ equipped with the trivial $\sym_2$-action,
\[\im(\phi_\ell)^{\sym_2}=\mathcal H^0(T_\ell^\bullet)^{\sym_2}=H_\ell\otimes_\C\mathcal H^0(\hat R_\ell^\bullet)^{\sym_2}=(\pi_*H_\ell)^{\oplus N(k,\ell)}\]
follows by the previous lemma.
\end{proof}
We denote the diagonal embedding of $X$ into $S^2X$ again by $\delta$.
\begin{cor}\label{Ktwofree}
For $E_1,\dots,E_k$ locally free sheaves on $X$, there is in the Grothendieck group $\K(S^2X)$ the equality
\[\mu_!\bigl[(E_1^{[2]}\otimes\dots\otimes E_k^{[2]})\bigr]=\bigl[K_0^{\sym_2}\bigr]-\sum_{\ell=1}^k N(k,\ell)\delta_!\left(\bigl[S^{\ell-1}\Omega_X\bigr]\cdot\bigl[E_1\bigr]\cdots \bigl[E_k\bigr]\right)\,.\]  
\end{cor}
\begin{proof}
Since $R\mu_*(E_1^{[2]}\otimes\dots\otimes E_k^{[2]})$ is cohomologically concentrated in degree zero (see corollary proposition \ref{mubkr} and theorem \ref{Sca}), we have \[
\mu_! \bigl[(E_1^{[2]}\otimes\dots\otimes E_k^{[2]})\bigr]=\bigl[\mu_*(E_1^{[2]}\otimes\dots\otimes E_k^{[2]})\bigr]\,.                                                                                                                       
                                                                                                                       \]
Now the formula follows by $K_k^{\sym_2}=\mu_*(E_1^{[2]}\otimes\dots\otimes E_k^{[2]})$ (see corollary \ref{mucor}) and the previous proposition. 
\end{proof}
We can identify a set of representatives of the $\sym_2$-orbits of $_2I_0$ (see remark \ref{orbits} (i)) with the set of subsets $P\subset [k]$ with $1\in P$ by identifying $P$ with the map $a\colon [k]\to [2]$ with $a^{-1}(1)=P$ and $a^{-1}(2)=[k]\setminus P$. We get the following as a special case of lemma \ref{kzero}.
\begin{lemma}\label{zerotwo}
On $S^2X$ there is the following isomorphism
\[K_0^{\sym_2}\cong \bigoplus_{1\in P\subset [k]} \pi_*K_0(P)\quad,\quad K_0(P)=\left(\bigotimes_{t\in P} E_t\right)\boxtimes \left(\bigotimes_{t\in [k]\setminus P} E_t\right)\,.\]
\end{lemma}
\begin{prop}\label{Eulertwo}
Let $X$ be a smooth projective surface.
Then the Euler characteristic of tensor products of tautological bundles on $X^{[2]}$ is related to the Euler characteristic of the bundles on $X$ by the formula   
\begin{align*}&\chi_{X^{[2]}}\left(E_1^{[2]}\otimes \dots\otimes E_k^{[2]}\right)\\=&\sum_{1\in P\subset[k]}\chi\left(\bigotimes_{t\in P}E_t\right)\cdot \chi\left(\bigotimes_{t\in [k]\setminus P}E_t\right)-\sum_{\ell=1}^{k-1}N(k,\ell)\chi\left(S^{\ell-1}\Omega_X\otimes E_1\otimes\dots\otimes E_k\right)\,.\end{align*}
\end{prop}

%% file: bichar.tex
\subsection{Extension groups on $X^{[2]}$}\label{bichar}
Similarly, we can compute for locally free sheaves $E_1,\dots,E_k$ and $F_1,\dots,F_{\hat k}$ the Euler bicharacteristics
\[\chi\left( E_1^{[2]}\otimes\dots\otimes E_k^{[2]},F_{1}^{[2]}\otimes\dots\otimes F_{\hat k}^{[2]}\right):=\chi\left(\Ext^*\left( E_1^{[2]}\otimes\dots\otimes E_k^{[2]},F_{1}^{[2]}\otimes\dots\otimes F_{\hat k}^{[2]}\right)\right)\,.\]
We use the notation 
$K_\ell:=K_\ell(E_1,\dots,E_k)$, $\hat K_\ell:=K_\ell(F_1,\dots,F_{\hat k})$
and define $T_\ell^\bullet$ and $\hat T_\ell^\bullet$ analogously.
In $\K(X^2)$ the following holds 
\begin{align*}
&R\sHom(\Phi(E_1^{[2]}\otimes \dots\otimes E_k^{[2]}),\Phi(F_1^{[2]}\otimes\dots\otimes F_{\hat k}^{[2]}))\\
=&R\sHom(K_k,\hat K_{\hat k})\\
=&R\sHom(K_0-\sum_{\ell=1}^k T_\ell^\bullet,\hat K_0-\sum_{\hat \ell=1}^{\hat k} \hat T_{\hat \ell}^\bullet)\\
=&R\sHom(K_0,\hat K_0)-\sum_{\hat\ell=1}^{\hat k} R\sHom(K_0,\hat T_{\hat \ell}^\bullet)-\sum_{\ell=1}^kR\sHom(T_\ell^\bullet,\hat K_0)+\sum_{\ell=1}^k\sum_{\hat\ell=1}^{\hat k}
R\sHom(T_\ell^\bullet,\hat T_{\hat\ell}^\bullet)\,.
\end{align*}
The permutation $\tau=(1\,\,2)$ acts freely on the direct summands of every term of the complex 
$R\sHom(K_0,\hat T_{\hat \ell}^\bullet)=\sHom^\bullet(K_0,\hat T_{\hat \ell}^\bullet)$. Thus, in $\K(S^2X)$ there is the equality
\[R\sHom(K_0,\hat T_{\hat \ell}^\bullet)^{\sym_2}=a(k,\hat k,\hat \ell)\cdot \delta_! A_{\hat \ell}\]
where 
\begin{align*}
a(k,\hat k,\hat \ell)&=2^{k-1}\cdot\sum_{\hat j=\hat \ell}^{\hat k}\binom{\hat k}{\hat j}\quad,\quad
A_{\hat \ell}=\sHom(E_1\otimes \dots\otimes E_k, S^{\hat \ell-1} \Omega_X\otimes F_1\otimes \dots\otimes F_{\hat k})\,.
\end{align*}
Furthermore, we have by equivariant Grothendieck duality (see \cite{Has} for the general theory or \cite{Kru} for the special case which is needed here)
$R\sHom(T_\ell^i,\hat K_0)=\bigoplus_{{_kI}_\ell^i \times {_{\hat k}I}_0} \delta_* B_\ell[-2]$
with $\tau$ acting freely on the direct summands which are of the form
\[B_\ell=\omega_X^\vee\otimes \sHom(S^{\ell-1}\Omega_X\otimes E_1\otimes \dots\otimes E_k,F_1\otimes \dots\otimes F_{\hat k})\,.\]
Thus, in $\K(S^2X)$ there is the equality 
\[R\sHom(T_\ell^\bullet,\hat K_0)^{\sym_2}=b(k,\ell,\hat k)\cdot\delta_!B_\ell\quad,\quad b(k,\ell,\hat k)=2^{\hat k-1}\cdot\sum_{j=\ell}^k \binom kj\,.\]
Similarly, we have 
\[\sExt^*(T_\ell^i,\hat T_{\hat \ell}^{\hat i})=\bigoplus_{ {_kI}_\ell^i \times {_{\hat k}I}_{\hat \ell}^{\hat i}}\delta_*\left(C_{\ell,\hat \ell}[0]\oplus T_X\otimes C_{\ell,\hat \ell}[-1]\oplus \omega_X^\vee\otimes C_{\ell,\hat \ell}[-2]\right)\]
with $C_{\ell,\hat \ell}=\sHom(S^{\ell-1}\Omega_X\otimes E_1\otimes\dots\otimes E_k,S^{\hat \ell-1}\Omega_X\otimes F_1\otimes \dots\otimes F_{\hat k})$. On these summands, $\tau$ acts freely except for in the case that $i=k-\ell$ and $\hat i=\hat k-\hat \ell$.
In that case $\tau$ acts on $C$ by $(-1)^{k+\ell-1}(-1)^{\hat k+\hat\ell-1}=(-1)^{k-\ell}(-1)^{\hat k-\hat \ell}$. Thus, in $\K(S^2X)$ there is the equality
\begin{align*}
R\sHom(T_\ell^\bullet,\hat T_{\hat \ell}^\bullet)^{\sym_2}&=\delta_!\left(c_+(k,\hat k,\ell,\hat \ell)(C_{\ell,\hat \ell} +\omega_X^\vee\otimes C_{\ell,\hat \ell})- c_-(k,\hat k,\ell,\hat \ell) T_X\otimes C_{\ell,\hat \ell}\right)\,,\\
c_\pm(k,\hat k,\ell,\hat \ell)&=\frac 12\left(\sum_{j=\ell}^k\binom kj\cdot\sum_{\hat j=\hat \ell}^{\hat k}\binom{\hat k}{\hat j}\pm\binom{k-1}{\ell-1}\binom{\hat k-1}{\hat\ell-1}\right)\,.
\end{align*}
The different signs in the coefficients are due to the fact that $\tau$ acts by $-1$ on $T_X=N_\Delta$.
Finally, $\tau$ acts freely on the direct summands of $R\sHom(K_0,\hat K_0)$ which gives
\[R\sHom(K_0,\hat K_0)^{\sym_2}=\pi_!\left(\sum_{P,Q\subset[k],\,1\in P}\sHom(\bigotimes_{s\in P}E_s,\bigotimes_{t\in Q} E_t)\boxtimes \sHom(\bigotimes_{s\in [k]\setminus P}E_s,\bigotimes_{t\in[k]\setminus Q} E_t)\right)\,.\]
Summing up we get a formula for $R\sHom(\Phi(E_1^{[2]}\otimes \dots\otimes E_k^{[2]}),\Phi(F_1^{[2]}\otimes\dots\otimes F_{\hat k}^{[2]}))^{\sym_2}$ in $\K(S^2X)$. This gives using corollary \ref{bkrext} the following formula for the Euler characteristic:
\begin{align*}
&\chi(E_1^{[2]}\otimes\dots\otimes E_k^{[2]},F_1^{[2]}\otimes\dots \otimes F_{\hat k}^{[2]})\\
=&\sum_{P,Q\subset[k],\,1\in P}\chi(P,Q)\chi([k]\setminus P,[k]\setminus Q)-\sum_{\hat \ell=1}^{\hat k}a(k,\hat k,\hat \ell)\chi(A_{\hat \ell})
-\sum_{\ell=1}^kb(k,\ell,\hat k)\chi(B_\ell)\\ & + \sum_{\ell=1}^k\sum_{\hat \ell=1}^{\hat k}
\left( c_+(k,\hat k,\ell,\hat \ell)(\chi(C_{\ell,\hat\ell})+\chi(\omega_X^\vee\otimes C_{\ell,\hat \ell}))-c_-(k,\hat k,\ell,\hat \ell)\chi(T_X\otimes C_{\ell,\hat \ell})\right) 
\end{align*}
where $\chi(P,Q):=\chi(\otimes_{p\in P}E_p,\otimes_{q\in Q}F_q)$.

%% file: threetens.tex
\section{Triple tensor products of tautological bundles on $X^{[n]}$}\label{threetens}
We consider the case $k=3$ and $n\ge 3$. We have by corollary \ref{kvanish}
\[\mu_*(E_1^{[n]}\otimes E_2^{[n]}\otimes E_3^{[n]})\cong K_2^{\sym_n}\,.\] 
In this section we will enlarge the exact sequences 
\[0\to K_\ell^{\sym_n}\to K_{\ell-1}^{\sym_n}\xrightarrow{\phi_\ell^{\sym_n}} T_\ell^{\sym_n}\]
for $\ell=1,2$ to long exact sequences on $S^nX$ with a zero on the right. 
This will lead to results for the cohomology of triple tensor products of tautological bundles. Results for the double tensor product can be found in \cite{Sca1} and \cite{Sca2}.
We consider the set of representatives 
\begin{align*}
\tilde I:= \bigl\{(1,1,1),\,(1,1,3),\,(1,3,1),\,(3,1,1),\,(1,2,3)\bigl\}
\end{align*}
of the $\sym_n$-orbits of $I_0$ and the set of representatives $\tilde J=\tilde J^1\cup \tilde J^2\cup \tilde J^3\cup \tilde J^4$ of the $\sym_n$-orbits of $\hat I_1$ given by
\begin{align*}
\tilde J^1&=\{(\{1\};1,3;(1,1)),\,(\{2\};1,3;(1,1)),\,(\{3\};1,3;(1,1))\}\,,\\
\tilde J^2&=\{(\{1\};1,3;(1,3)),\,(\{2\};1,3;(3,1)),\,(\{3\};1,3;(1,3))\}\,,\\
\tilde J^3&=\{(\{1\};1,2;(1,3)),\,(\{2\};1,2;(3,1)),\,(\{3\};1,2;(1,3))\}\,,\\
\tilde J^4&=\{(\{1\};1,2;(3,1)),\,(\{2\};1,2;(1,3)),\,(\{3\};1,2;(3,1))\}\,.
\end{align*}
Here in the tuple $(\{t\};i,j;(\alpha,\beta))$ the canonical identification $\Map([3]\setminus\{t\},[n])\cong [n]^2$ is used, i.e. the tuple $(\alpha,\beta)$ stands for the map $a\colon [3]\setminus\{t\}\to [n]$ given by $a(\min([3]\setminus\{t\}))=\alpha$ and $a(\max([3]\setminus\{t\}))=\beta$. One can check that $\tilde I$ and $\tilde J$ are indeed systems of representatives of the $\sym_n$-orbits of $I_0$ respectively $\hat I_1$ by giving bijections which preserve the orbits $\tilde I\cong J_0$ and $\tilde J\cong \hat J_1$ to the sets of representatives given in remark \ref{orbits}. 
Thus, we have the isomorphism
$K_0^{\sym_n}\cong \oplus_{a\in \tilde I} K_0(a)^{\sym_{\overline {\im(a)}}}$ and $T_1^{\sym_n}$ is given by
\[\bigoplus_{\tilde J^1}T_\ell(\{t\};1,3;(\alpha,\beta))^{\sym_{\overline{\{1,3\}}}}\oplus \bigoplus_{\tilde J^2}T_\ell(\{t\};1,3;(\alpha,\beta))^{\sym_{\overline{\{1,3\}}}}\oplus \bigoplus_{\tilde J^3\cup \tilde J^4}T_\ell(\{t\};1,2;(\alpha,\beta))^{\sym_{[4,n]}}\,.\]
We use for $r=1,2,3,4$ the notation 
\[T_\ell(r)=\bigoplus_{(\{t\};i,j;(\alpha,\beta))\in \tilde J^r} T_1(\{t\};i,j;(\alpha,\beta))\]
as well as $\phi_1(r)=\oplus_{\tilde J^r} \phi_1(\{t\};i,j;(\alpha,\beta))$.
Then
\[K_1^{\sym_n}=\ker(\phi_1^{\sym_n})=\ker(\phi_1^{\sym_n}(1))\cap \ker(\phi_1^{\sym_n}(2))\cap \ker(\phi_1^{\sym_n}(3))\cap \ker(\phi_1^{\sym_n}(4))\,.\]
For every $\gamma\in \tilde J_1\cup \tilde J_2$ the sheaves $T_1(\gamma)$ are canonically isomorphic to
\[
 \mathcal F:=\bigl(E_1\otimes E_2\otimes E_3 \bigl)_{13}\,.
\]
Hence, $T_1(2)^{\sym_{\overline{\{1,3\}}}}\cong T_1(1)^{\sym_{\overline{\{1,3\}}}}\cong (\F^{\sym_{\overline{\{1,3\}}}})^{\oplus 3}$. 
\begin{lemma}\label{onlyone}
The following sequence on $S^nX$ is exact:
\[ 0\to \ker(\phi_1^{\sym_n}(1))\cap \ker(\phi_1^{\sym_n}(2))\to K_0^{\sym_n}\xrightarrow{\phi_1^{\sym_n}(1)}  T_1(1)^{\sym_{\overline{\{1,3\}}}}
\to 0\,.\]
In particular $\ker(\phi_1^{\sym_n}(1))\cap \ker(\phi_1^{\sym_n}(2))=\ker(\phi_1^{\sym_n}(1))$.
\end{lemma}
\begin{proof}
Let $\tau=(1\,\,3)\in \sym_n$. Note that whenever $s\in K_0^{\sym_n}$ is a $\sym_n$-invariant section $s(3,3,1)=\tau_*s(1,1,3)$ holds. Since $\tau$ acts trivially on $\Delta_{13}$, there is the equality \[s(3,3,1)_{|\Delta_{13}}=s(1,1,3)_{|\Delta_{13}}\] in $\F$. Analogously, $s(3,1,3)_{|\Delta_{13}}=s(1,3,1)_{|\Delta_{13}}$ and $s(1,1,3)_{|\Delta_{13}}=s(3,3,1)_{|\Delta_{13}}$.  
The exactness in the first two degrees comes from the inclusion $\ker(\phi_1(1))\subset \ker(\phi_1(2))$.
Indeed, a section $s\in K_0^{\sym_n}$ is in $\ker(\phi_1^{\sym_n}(1))$ if and only if
\[s(1,1,1)_{|\Delta_{13}}=s(1,1,3)_{|\Delta_{13}}=s(1,3,1)_{|\Delta_{13}}=s(3,1,1)_{|\Delta_{13}}\,,\]
whereas $s\in \ker(\phi_1(2))$ holds if and only if 
\[s(1,1,3)_{|\Delta_{13}}=s(1,3,1)_{|\Delta_{13}}=s(3,1,1)_{|\Delta_{13}}\,.\]
For the exactness on the right consider $t\in T_1(1)^{\sym_{\overline{\{1,3\}}}}$. Then locally there are sections $s(a)\in K_0(a)^{\sym_{\overline{\{1,3\}}}}$ for $a=(1,1,3),(1,3,1),(3,1,1)$ such that 
\begin{align*}&s(3,1,1)_{|\Delta_{13}}=t(\{1\};1,3;(1,1))\quad,\quad s(1,3,1)_{|\Delta_{13}}=t(\{2\};1,3;(1,1))\,,\\&s(1,1,3)_{|\Delta_{13}}=t(\{3\};1,3;(1,1))\,.\end{align*}
By setting $s(1,1,1)=0=s(1,2,3)$ we have indeed defined a local section $s\in K_0^{\sym_n}$ with $\phi_1^{\sym_n}(1)(s)=t$.
\end{proof}
For every tuple $(\{t\};i,j;(\alpha,\beta))\in \tilde J$ the restriction of the corresponding sheaf to $\Delta_{123}$ is given by
\[T_1(\{t\};i,j;(\alpha,\beta))_{|\Delta_{123}}=\left(E_1\otimes E_2 \otimes E_3 \right)_{123}=:\mathcal E\,.\]
Thus we can define the $\sym_{[4,n]}$-equivariant morphism
\begin{align*}
F\colon T_1(3)\to \E^{\oplus 2}\,,\quad s\mapsto\binom{s(\{1\};1,2;(1,3))_{|\Delta_{123}}-s(\{2\};1,2;(3,1))_{|\Delta_{123}}}{s(\{2\};1,2;(3,1))_{|\Delta_{123}}-s(\{3\};1,2;(1,3))_{|\Delta_{123}}}\,.
\end{align*}
\begin{lemma}\label{kone}
 On $S^nX$ there is the exact sequence
\[0\to K_1^{\sym_n}\to \ker(\phi_1^{\sym_n}(1))\xrightarrow{\phi_1^{\sym_n}(3)} T_1(3)^{\sym_{[4,n]}}\xrightarrow{F^{\sym_{[4,n]}}} (\mathcal E^{\oplus 2})^{\sym_{[4,n]}}\to 0\,.\]
\end{lemma}
\begin{proof}
Because of the invariance under the transposition $(1\,\,2)$, we have for every $s\in K_0^{\sym_n}$ the equality $s(1,2,3)_{|\Delta_{12}}=s(2,1,3)_{|\Delta_{12}}$. Thus,
\begin{align*}
\phi_1(s)(\{1\};1,2;(1,3))=s(1,1,3)_{|\Delta_{12}}-s(2,1,3)_{|\Delta_{12}}&=s(1,1,3)_{|\Delta_{12}}-s(1,2,3)_{|\Delta_{12}}\\&=
\phi_1(s)(\{2\};1,2;(1,3))\,.
\end{align*}
Similarly, we get \[\phi_1(s)(\{2\};1,2;(3,1))=\phi_1(s)(\{3\};1,2;(3,1))\,,\,\phi_1(s)(\{3\};1,2;(1,3))=\phi_1(s)(\{1\};1,2;(3,1))\,.\]
This shows that $\ker(\phi_1^{\sym_n}(3))=\ker(\phi_1^{\sym_n}(4))$.
Together with the previous lemma we thus have $K_1^{\sym_n}=\ker(\phi_1^{\sym_n}(1))\cap \ker(\phi_1^{\sym_n}(3))$ which shows the exactness in the first two degrees. Let $s\in K_0$ be a $\sym_n$-invariant section with $\phi_1(1)(s)=0$, i.e.
\[s(1,1,1)_{|\Delta_{13}}=s(1,1,3)_{|\Delta_{13}}=s(1,3,1)_{|\Delta_{13}}=s(3,1,1)_{|\Delta_{13}}\,.\]
By the $\sym_n$-invariance we have
\[s(2,1,3)=(1\,\,2)_*s(1,2,3)\,,\,s(3,2,1)=(1\,\,3)_*s(1,2,3)\,,\,s(1,3,2)=(2\,\,3)_*s(1,2,3)\,.\]
Since $\sym_{[3]}$ acts trivially on $\Delta_{123}$, it follows that
\[s(1,2,3)_{|\Delta_{123}}=s(2,1,3)_{|\Delta_{123}}=s(3,2,1)_{|\Delta_{123}}=s(1,3,2)_{|\Delta_{123}}\,.\]
 This yields for  the first component $F_1$ of the morphism $F$
\begin{align*}
F_1(\phi_1(3)(s))&= \phi_1(s)(\{1\};1,2;(1,3))_{|\Delta_{123}}-\phi_1(s)(\{2\};1,2;(3,1))_{|\Delta_{123}}\\
&=\bigl(s(1,1,3)-s(2,1,3)-s(3,1,1)+s(3,2,1)\bigr)_{|\Delta_{123}}=0\,.
\end{align*}
Analogously, $F_2(\phi_1(3)(s))=0$ which gives the inclusion $\im(\phi_1^{\sym_n}(3))\subset \ker(F^{\sym_{[4,n]}})$. To show the other inclusion let $t\in \ker(F^{\sym_{[4,n]}})$, i.e.
\[t(\{1\};1,2;(1,3))_{|\Delta_{123}}=t(\{2\};1,2;(3,1))_{|\Delta_{123}}=t(\{3\};1,2;(1,3))_{|\Delta_{123}}=:t_{123}\,.\]
As explained in remark \ref{partialquotient} we can consider the invariants of the occurring direct summands by their stabilisers as sheaves on the quotients of $X^n$ by the stabilisers, i.e. we have for $a=(1,1,3),(1,3,1),(3,1,1)$, $\gamma_1\in \tilde J^1$, and $\gamma_3\in \tilde J^3$
\begin{align*}
\mathcal E^{\sym_{[4,n]}}\in\Coh(\Delta_{123}\times S^{[4,n]}X)\,&,\,   T_1(\gamma_3)^{\sym_{[4,n]}}\in\Coh(\Delta_{12}\times S^{[4,n]}X)\\
\F^{\sym_{\overline{\{1,3\}}}}=T_1(\gamma_1)^{\sym_{\overline{\{1,3\}}}}\in\Coh(\Delta_{13}\times S^{\overline{\{1,3\}}}X)
\,&,\,
K_0(a)^{\sym_{\overline{\{1,3\}}}}\in\Coh(X^{\{1,3\}}\times S^{\overline{\{1,3\}}}X)
\end{align*}
There are the two cartesian diagrams of closed embeddings (see lemma \ref{stillclosed})
\[ \begin{CD}
\Delta_{123}
@>>>
\Delta_{12} \\
@VVV
@VVV \\
\Delta_{13}
@>>>
X^n
\end{CD}\quad,\quad
 \begin{CD}
\Delta_{123}\times S^{[4,n]}X
@>>>
\Delta_{12}\times S^{[4,n]}X \\
@VVV
@VVV \\
\Delta_{13}\times S^{\overline{\{1,3\}}}X
@>>>
X^{\{1,3\}}\times S^{\overline{\{1,3\}}}X
\end{CD} \]
where the second one is induced by the first one by the universal properties of the quotient morphisms. That the second diagram is indeed again cartesian can be seen by assuming that $X$ is affine with $\reg(X)=A$ (see lemma \ref{opencover}). Then
\begin{align*}
 &\reg(\Delta_{12}\times S^{[4,n]}X)\otimes_{\reg(X^{\{1,3\}}\times S^{\overline{\{1,3\}}}X)}\reg(\Delta_{13}\times S^{\overline{\{1,3\}}}X)\\
=&(A^{\otimes 2}\otimes S^{n-3} A)\otimes_{A^{\otimes 2}\otimes S^{n-2}A}(A\otimes S^{n-2}A)\cong (A^{\otimes 2}\otimes S^{n-3} A)\otimes_{A^{\otimes 2}}A\cong A\otimes S^{n-3}A\\
=&\reg(\Delta_{123}\times S^{[4,n]}X)\,.
\end{align*}
We now choose $s_{13}\in \F^{\sym_{\overline{\{1,3\}}}}$ as any local section on $\Delta_{13}\times S^{\overline{\{1,3\}}}X$ such that
\[s_{13|\Delta_{123}\times S^{[4,n]}X}=t_{123}\,.\]
Then by lemma \ref{closedres} there exists a local section $s(1,1,3)\in K_0(1,1,3)^{\sym_{\overline{\{1,3\}}}}$ on the quotient $X^{\{1,3\}}\times S^{\overline{\{1,3\}}}X$ such that
\[
s(1,1,3)_{|\Delta_{13}\times S^{\overline{\{1,3\}}}X}=s_{13}\quad,\quad s(1,1,3)_{|\Delta_{12}\times S^{[4,n]}X}=t(\{1\};1,2;(1,3))\,.
\]
Analogously, there are local sections $s(1,3,1)\in K_0(1,3,1)^{\sym_{\overline{\{1,3\}}}}$ and $s(3,1,1)\in K_0(3,1,1)^{\sym_{\overline{\{1,3\}}}}$ such that their restrictions to the appropriate closed subvarieties are $s_{13}$ and $t(\{3\};1,2;(1,3))$ respectively $t(\{2\};1,2;(3,1))$. We furthermore set $s(1,2,3)=0$ and choose any section $s(1,1,1)\in K_0(1,1,1)$ on $X^1\times S^{[2,n]}$ that restricts to $s_{13}$ on $\Delta_{13}\times S^{\overline{\{1,3\}}}X$. Then $s=(s(a))_{a\in \tilde I}$ is indeed a section of $\ker(\phi_1^{\sym_n}(1))$ with the property that
$\phi_1(3)(s)=t$. Finally, the surjectivity of the morphism $F$ follows directly from its definition. This implies the exactness on the right.  
\end{proof}
A system of representatives of $\hat I_2$ is given by
\[\tilde J_2:=\bigl\{(\{1,2\};1,3;1)\,,\,(\{1,3\};1,3;1)\,,\,(\{2,3\};1,3;1)\bigr\}
 \]
The sheaves $T_2(\gamma)$ for $\gamma\in \tilde J_2$ are all canonically isomorphic to
\[\mathcal H:=N_{\Delta_{13}}\otimes \bigl(E_1\otimes E_2\otimes E_3\bigr)_{13}\cong\bigl(\Omega_X\otimes E_1\otimes E_2\otimes E_3\bigr)_{13}\,.
\]
The restriction $\mathcal H\to \mathcal H_{|\Delta_{123}}$ induces the morphism $\res\colon \mathcal H^{\sym_{\overline{\{1,3\}}}}\to \mathcal H_{|\Delta_{123}}^{\sym_{[4,n]}}$.
\begin{lemma}\label{ktwo}
 The following sequence is exact:
\[0\to K_2^{\sym_n}\to K_1^{\sym_n}\xrightarrow{\phi_2^{\sym_n}(\{1,2\};1,3;1)} \mathcal H^{\sym_{\overline{\{1,3\}}}}\xrightarrow{\res} \mathcal H_{|\Delta_{123}}^{\sym_{[4,n]}}\to 0\,.\]
\end{lemma}
\begin{proof}
 Let $s\in K_1^{\sym_n}$ and $\tau=(1\,\,3)\in \sym_n$. Then 
\[s(3,3,1)=\tau_*s(1,1,3)\quad,\quad s(3,1,3)=\tau_*s(1,3,1)\quad,\quad s(1,3,3)=\tau_*s(3,1,1)\,.\]
Since $\tau$ acts by $(-1)^{2+1}=-1$ (see remark \ref{minus}) on $\mathcal H$ one can compute that $\phi_2(s)(\gamma)$ is equal for all $\gamma\in \tilde J_2$. Namely, we have modulo $\mathcal I_{13}^2$ the equalities
\begin{align*}
\phi_2(s)(\{1,2\};1,3;1)
=&s(1,1,1)-s(3,1,1)-s(1,3,1)+\tau_*s(1,1,3)\\ 
=&s(1,1,1)-s(3,1,1)-\tau_*\bigl(-s(1,1,3)+\tau_*s(1,3,1)\bigr)\\
=&s(1,1,1)-s(3,1,1)-s(1,1,3)+\tau_*s(1,3,1)\\
=&\phi_2(s)(\{1,3\};1,3;1)=\dots = \phi_2(s)(\{2,3\};1,3;1)\,.
\end{align*}
Thus, $K_2^{\sym_n}=\ker(\phi_2^{\sym_n})=\ker(\phi_2^{\sym_n}(\gamma))$ for any $\gamma\in \tilde J_2$. In particular, this holds for the tuple $\gamma=(\{1,2\};1,3;1)$ which shows the exactness of the sequence in the first two degrees. 
The surjectivity of the map $\res$ is due to the fact that it is given by the restriction of sections along the closed embedding 
$\Delta_{123}\times S^{[4,n]} X\to \Delta_{13}\times S^{\overline{\{1,3\}}}X$
which is induced by the embedding $\Delta_{123}\hookrightarrow \Delta_{13}$. Thus, it is only left to show that the sequence is exact at the term 
$\mathcal H^{\sym_{\overline{\{1,3\}}}}$ which equivalent to the equality 
\[\im(\phi_2^{\sym_n}(\{1,2\};1,3;1))=\mathcal (I_{123}\cdot\mathcal H)^{\sym_{\overline{\{1,3\}}}}\,.\]
Since it suffices to show the equality on a family of open subsets of $X^n$ covering $\Delta_{123}$, we can assume that $E_1=E_2=E_3=\reg_X$ (see proof of proposition \ref{opendes}), i.e. $\mathcal H=\mathcal I_{13}/\mathcal I_{13}^2=N_{\Delta_{13}}$. Furthermore, we may assume that $\mathcal I_{13}/\mathcal I_{13}^2$ is a free $\reg_{\Delta_{13}}$-module with generators $\bar \zeta_1,\bar\zeta_2$. The generators $\bar \zeta_i$ can be taken as the pullback along $\pr_{13}$ of generators of $N_\Delta$ where $\Delta$ is the diagonal in $X^2$. Thus, we may assume that their representatives $\zeta_i\in \mathcal I_{13}$ are $\sym_{\overline{\{1,3\}}}$-invariant. 
For $f\in \mathcal I_{123}^{\sym_{\overline{\{1,3\}}}}$ and $i=1,2$ we have to show that $f\cdot \bar \zeta_i\in \im(\phi_2^{\sym_n}(\{1,2\};1,3;1))$.
We set $F=-f\cdot \zeta_i\in \mathcal I_{123}\cdot \mathcal I_{13}$. Then also $(1\,\,3)_*F\in \mathcal I_{123}\cdot \mathcal I_{13}$. Since
\begin{align*}\I_{123} I_{13}=(\I_{12}+\I_{23}) \I_{13}=\I_{12}\I_{13}+\I_{23}\I_{13}&\subset \I(\Delta_{12}\cup \Delta_{13})+\I_{23}=\I((\Delta_{12}\cup \Delta_{13})\cap \Delta_{23})\,,\end{align*}
the restriction $((1\,\,3)_*F)_{|(\Delta_{12}\cup \Delta_{13})\cap \Delta_{23}}$ vanishes. Here $\cup$ and $\cap$ denote the scheme-theoretic union and intersection. By lemma \ref{closedres} there is a $G\in \reg_{X^n}$ such that $G_{|\Delta_{23}}=(1\,\,3)_*F$ and $G_{|\Delta_{12}}=0=G_{|\Delta_{13}}$.
Since $0$ as well as $(1\,\,3)_*F$ are $\sym_{[4,n]}$-invariant functions, it is possible to choose also $G\in \reg_{X^n}^{\sym_{[4,n]}}$.
We set
\[s(1,1,1)=s(1,3,1)=s(1,1,3)=0\quad,\quad s(3,1,1)=F\quad,\quad s(1,2,3)=G\,.\] 
Then $s\in K_1^{\sym_n}$. Indeed, since $F\in \mathcal I_{13}$ we have 
\[s(1,1,1)_{|\Delta_{13}}=s(1,3,1)_{|\Delta_{13}}=s(1,1,3)_{|\Delta_{13}}=s(3,1,1)_{|\Delta_{13}}\]
which is the condition for $s\in \ker(\phi_1(1))\cap\ker(\phi_1(2))$ (see proof of lemma \ref{onlyone}). 
Furthermore, $\phi_1(s)(\{2\};1,2;(3,1))=0$ since 
\begin{align*}s(3,2,1)_{|\Delta_{12}}= \bigl((1\,\,3)_* s(1,2,3)\bigr)_{|\Delta_{12}}=(1\,\,3)_* \bigl(s(1,2,3)_{|\Delta_{23}}\bigr)&=(1\,\,3)_*(1\,\,3)_* (F_{|\Delta_{12}})\\&=F_{|\Delta_{12}}=s(3,1,1)_{|\Delta_{12}}\,.\end{align*}
Similarly, we also get that $\phi_1(s)(\{1\};1,2;(1,3))=0=\phi_1(s)(\{3\};1,2;(1,3))$. Because of $\phi_2(s)(\{1,2\};1,3;1)=f\cdot\bar \zeta_i$ we get the inclusion  
\[\im(\phi_2^{\sym_n}(\{1,2\};1,3;1))\supset\mathcal I_{123}^{\sym_{\overline{\{1,3\}}}}\mathcal H\,.\]
For the other inclusion we first notice that the surjection $\I_{13}\to \I_{13}/\I_{13}^2$ is given under the identifications
\[\I_{13}/\I_{13}\cong N_{\Delta_{13}}\cong (\Omega_X)_{13}\cong (\pr_1^*\Omega_X)_{|\Delta_{13}}\cong (\pr_3^*\Omega_X)_{|\Delta_{13}}\]
by $s\mapsto d_1s-d_3s$. Here $d_i\colon \reg_{X^n}\to \pr_i^*\Omega_X$ denotes the composition of the  differential
\[d\colon \reg_{X^n}\to \Omega_{X^n}\cong \pr_1^*\Omega_X\oplus\dots\oplus\pr_n^*\Omega_X\]
with the projection to the $i$-th summand. For $s\in \reg_{X^n}$ and $\tau=(i\,\,j)\in \sym_n$ we have $d_i(\tau_*s)=d_js$.
Let $s\in K_1^{\sym_n}$ with $a:=s(1,1,1)$ and $b:=s(1,2,3)$. Then because of $\phi_1^{\sym_n}(1)(s)=0$
\[
s(1,1,3)_{|\Delta_{13}}=s(1,3,1)_{|\Delta_{13}}=s(3,1,1)_{|\Delta_{13}}=a_{|\Delta_{13}}\]
and because of $\phi_1^{\sym_n}(3)(s)=0$
\[s(1,1,3)_{|\Delta_{12}}=\bigl((1\,\,2)_*b\bigr)_{|\Delta_{12}}\,,\,s(1,3,1)_{|\Delta_{12}}=\bigl((2\,\,3)_*b\bigr)_{|\Delta_{12}}
\,,\, s(3,1,1)_{|\Delta_{12}}=\bigl((1\,\,3)_*b\bigr)_{|\Delta_{12}}\,.
\]
Thus, over $\Delta_{123}$ there are the equalities
\begin{align*}
&d_3s(1,1,3)=d_3b\,,\,d_3s(1,3,1)=d_2b\,,\,d_3s(3,1,1)=d_1b\,,\\ 
&d_1s(1,1,3)=(d_1+d_3)a-d_3b\,,\,d_1s(1,3,1)=(d_1+d_3)a-d_2b\,,\,d_1s(3,1,1)=(d_1+d_3)a-d_1b\,.
\end{align*}
Hence, still over $\Delta_{123}$ we get
\begin{align*}&\phi_2(s)(\{1,2\};1,3;1)\\
=&(d_1-d_3)\bigl(s(1,1,1)-s(3,1,1)-s(1,3,1)+(1\,\,3)_*s(1,1,3)\bigr)\\
=&d_1a-(d_1+d_3)a+d_1b-(d_1+d_3)a+d_2b+d_3b-d_3a+d_1b+d_2b-(d_1+d_3)a+d_3b\\   
=&-2d_1a-4d_3a+2d_1b+2d_2b+2d_3b\\
=&-2(d_1+d_2+d_3)a+2(d_1+d_2+d_3)b\\
=&0\,.
\end{align*}
The fourth equality is due to the fact that $a$ is $(2\,\,3)$-invariant and thus $d_2a=d_3a$. The last equality is due to the fact that $a_{|\Delta_{123}}=b_{|\Delta_{123}}$ and thus $(d_1+d_2+d_3)a=(d_1+d_2+d_3)b$. Now we have shown that $\phi_2(s)(\{1,2\};1,3;1)\in \I_{123}\cdot \mathcal H$ which finishes the proof.
\end{proof}
\begin{cor}\label{Kthree}
 In the Grothendieck group $\K(S^nX)$ there is the equality
\[\mu_!\bigl[E_1^{[n]}\otimes E_2^{[n]}\otimes E_3^{[n]}\bigr]=\bigl[K_0^{\sym_n}\bigr] -3\bigl[\mathcal F^{\sym_{\overline{\{1,3\}}}} \bigr] -\bigl[
T_1(3)^{\sym_{[4,n]}}  \bigr]+2\bigl[\mathcal E^{\sym_{[4,n]}}  \bigr]- \bigl[\mathcal H^{\sym_{\overline{\{1,3\}}}}  \bigr]+ 
\bigl[\mathcal H_{|\Delta_{123}}^{\sym_{[4,n]}}  \bigr]
\,.\] 
\end{cor}
\begin{proof}
This follows by the fact that $R^i\mu_*(E_1^{[n]}\otimes E_2^{[n]}\otimes E_3^{[n]})=0$ for $i>0$ (see proposition \ref{mubkr} and theorem \ref{Sca}) and the results of this subsection. 
\end{proof}
\begin{defin}\label{chi}
We use for $m\in \N$ and $F^\bullet\in\D^b(X)$ the abbreviation (see lemma \ref{Euler})
\[s^{m}\chi(F^\bullet):= \chi(((F^\bullet)^{\boxtimes m})^{\sym_m})=\chi(S^m\Ho^*(F^\bullet))=\binom{\chi(F^\bullet)+m-1}{m}\,.\]
Furthermore, for $F^\bullet=\reg_X$ we set
\[s^{m}\chi:=s^m\chi(\reg_X)= \chi(\reg_{S^mX})=\chi(S^m\Ho^*(\reg_X))=\binom{\chi(\reg_X)+m-1}{m}\,.\]
\end{defin}
\begin{cor}\label{Eulerthree}
If $X$ is projective, the Euler characteristic of the triple tensor product of tautological bundles is given by
\begin{align*}
&\chi(E_1^{[n]}\otimes E_2^{[n]}\otimes E_3^{[n]})\\
=&\chi(E_1)\chi(E_2)\chi(E_3)s^{n-3}\chi \\
 &+\bigl(\chi(E_1\otimes E_2)\chi(E_3)+\chi(E_1\otimes E_3)\chi(E_2)+\chi(E_1\otimes E_3)\chi(E_2)  \bigr)\bigl(s^{n-2}\chi -s^{n-3}\chi\bigr)\\
 &+\chi(E_1\otimes E_2\otimes E_3)(s^{n-1}\chi-3s^{n-2}\chi+2 s^{n-3}\chi)\\
 &+\chi(\Omega_X\otimes E_1\otimes E_2\otimes E_3)(s^{n-3}\chi-s^{n-2}\chi) \,.
\end{align*}
\end{cor}
\begin{proof}
The first summand comes from $K_0(1,2,3)$, the second from $K_0(1,1,3)$, $K_0(1,3,1)$, $K_0(3,1,1)$, and $T_1(3)$, the third from $K_0(1,1,1)$, $\F$,and $\E$, and the fourth from $\mathcal H$ and $\mathcal H_{|\Delta_{123}}$. 
\end{proof}

%% file: deterline.tex
\section{Generalisations}
\subsection{Determinant line bundles}
There is a homomorphism which associates to any line bundle on $X$ its associated \textit{determinant line bundle} on $X^{[n]}$ given by  
\[\mathcal D\colon \Pic X\to \Pic X^{[n]}\quad ,\quad L\mapsto \mathcal D_L:= \mu^*((L^{\boxtimes n})^{\sym_n})\,.\]
Here the $\sym_n$-linearization of $L^{\boxtimes n}$ is given by the canonical isomorphisms $p_{\sigma^{-1}(i)}^*L\cong \sigma^*p_i^*L$, i.e. given by permutation of the factors.
By \cite[Theorem 2.3]{DN} the sheaf of invariants of $L^{\boxtimes n}$ is also the decent of $L^{\boxtimes n}$, i.e. $L^{\boxtimes n}\cong \pi^*((L^{\boxtimes n})^{\sym_n})$.
\begin{remark}
The functor $\mathcal D$ maps the trivial respectively the canonical line bundle to the trivial respectively the canonical line bundle, i.e. 
$\mathcal D_{\reg_X}\cong \reg_{X^{[n]}}$ and $\mathcal D_{\omega_X}\cong \omega_{X^{[n]}}$. The assertion for the trivial line bundle is true since the pull-back of the trivial line bundle along any morphism is the trivial line bundle and since taking the invariants of the trivial line bundle yields the trivial line bundle on the quotient by the group action. For a proof of $\mathcal D_{\omega_X}\cong \omega_{X^{[n]}}$ see \cite[Proposition 1.6]{NW}.
\end{remark}
\begin{lemma}\label{det}
Let $L$ be a line bundle on $X$.
\begin{enumerate}
\item For every $\mathcal F^\bullet\in \D^b(X^{[n]})$ there is a natural isomorphism
$\Phi(\F^\bullet\otimes \mathcal D_L)\simeq \Phi(\F^\bullet)\otimes L^{\boxtimes n}$ in $\D_{\sym_n}(X^n)$. 
\item For every $\mathcal G^\bullet\in \D^b_{\sym_n}(X^n)$ and every subgroup $H\le \sym_n$ there is a natural isomorphism $[\pi_*(\mathcal G^\bullet\otimes L^{\boxtimes n})]^{H}\simeq
(\pi_*\mathcal G^\bullet)^{H}\otimes (\pi_*L^{\boxtimes n})^{\sym_n}$.
\item For every $\mathcal F^\bullet\in \D^b(X^{[n]})$ there is a natural isomorphism 
\[R\mu_*(\F^\bullet\otimes \mathcal D_L)\simeq R\mu_*\mathcal F\otimes(L^{\boxtimes n})^{\sym_n}\,.\]  
\end{enumerate}
\end{lemma}
\begin{proof}
 By the definition of the determinant line bundle and the fact that $\pi^*(L^{\boxtimes n})^{\sym_n}\cong L^{\boxtimes n}$ we have 
\[q^*\mathcal D_L\cong q^*\mu^* (L^{\boxtimes n})^{\sym_n} \cong p^*\pi^*(L^{\boxtimes n})^{\sym_n}\cong p^*L^{\boxtimes n}\,.\]
Using this, we get indeed natural isomorphisms
\begin{align*}
\Phi(\F^\bullet\otimes \mathcal D_L)\simeq Rp_*q^*(\F^\bullet\otimes \mathcal D_L)\simeq Rp_*\left(q^*\F^\bullet\otimes q^*\mathcal D_L\right)&\simeq Rp_*\left(q^*\F^\bullet\otimes p^*L^{\boxtimes n} \right)\\
&\overset{\text{PF}}\simeq Rp_*q^*\F^\bullet\otimes L^{\boxtimes n}\\
&\simeq \Phi(\F^\bullet)\otimes L^{\boxtimes n}\,. 
\end{align*}
This shows (i). For (ii) we remember that the functor $(\_)^{\sym_n}$ on $\D^b_{\sym_n}(X^n)$ is a abbreviation of the composition $(\_)^{\sym}\circ \pi_*$.
Then
\[\left[\pi_*(\mathcal G^\bullet\otimes L^{\boxtimes n})\right]^{H}\simeq \left[\pi_*(\mathcal G^\bullet\otimes \pi^*(L^{\boxtimes n})^{\sym_n})\right]^{H}
 \overset{\text{PF}}\simeq \left[\pi_*(\mathcal G^\bullet)\otimes (L^{\boxtimes n})^{\sym_n}\right]^{H}
 \overset{\ref{tensinv}}\simeq (\pi_*\mathcal G^\bullet)^{H}\otimes (L^{\boxtimes n})^{\sym_n}\,.
\]
Now (iii) follows by (i),(ii) with $H=\sym_n$ and proposition \ref{mubkr} or directly by the projection formula along $\mu$.
\end{proof}
\begin{cor}\label{zeroconc}
Let $E_1,\dots,E_k$ be locally free sheaves and $L$ a line bundle on $X$. Then $\Phi(E_1^{[n]}\otimes \dots\otimes E_k^{[n]}\otimes \mathcal D_L)$ as well as
$R\mu_*(E_1^{[n]}\otimes \dots\otimes E_k^{[n]}\otimes \mathcal D_L)$ are cohomologically concentrated in degree zero, i.e.
\begin{align*}\Phi(E_1^{[n]}\otimes \dots\otimes E_k^{[n]}\otimes \mathcal D_L)&\simeq p_*q^*(E_1^{[n]}\otimes \dots\otimes E_k^{[n]}\otimes \mathcal D_L)\,,\,\\
 R\mu_*(E_1^{[n]}\otimes \dots\otimes E_k^{[n]}\otimes \mathcal D_L)&\simeq \mu_*(E_1^{[n]}\otimes \dots\otimes E_k^{[n]}\otimes \mathcal D_L)\,.
\end{align*}
\end{cor}
\begin{proof}
 This follows from proposition \ref{mubkr}, theorem \ref{Sca}, and the previous lemma.
\end{proof}
Using the above lemma and the corollary we can generalise the results of this chapter to sheaves on $X^{[n]}$ tensorised with determinant line bundles. 
\subsection{Derived functors}
We can generalise the results on the the push forwards of the Grothendieck classes along the Hilbert-Chow morphism and with this also the results on the Euler characteristics from locally free sheaves on $X$ to objects in $\D^b(X)$. For this we have to note that we can define the occurring functors $K_0$ and $T_\ell$ (subsection \ref{constr}), $H_\ell$ (subsection \ref{longzwei}), and $\F$, $\E$, and $\mathcal H$ (section \ref{threetens}) also on the level of derived categories.  
\begin{defin}\label{derivedT}
Let $E_1^\bullet,\dots,E_k^\bullet\in \D^b(X)$. We define for $n\in N$ derived multi-functors as follows
\begin{align*}K_0(a)\colon \D^b(X)^k&\to \D_{\sym_n}^b(X^n)\,,\\
 (E_1^\bullet,\dots,E_k^\bullet)&\mapsto \pr_{a(1)}^*E_1^\bullet\otimes^L\dots\otimes^L\pr_{a(k)}^*E_k^\bullet\simeq \bigotimes_{t=1}^n\pr_t^*(\otimes_{\alpha\in a^{-1}(t)}^L E_\alpha^\bullet)\,,\\
 T_\ell(M;i,j;a)\colon \D^b(X)^k&\to \D_{\sym_n}^b(X^n)\,,\\
(E_1^\bullet,\dots,E_k^\bullet)&\mapsto\left(S^{\ell-1}\Omega_X\otimes (\otimes_{\alpha\in \hat M(a)}^L E_\alpha^\bullet) \right)_{ij}\otimes \bigotimes_{t=3}^n\pr_t^*(\otimes_{\alpha\in a^{-1}(t)}^L \E_\alpha^\bullet)\,.\\
K_0:=&\bigoplus_{a\in I_0}K_0(a)\quad,\quad T_\ell:=\bigoplus_{(M;i,j;a)\in I_\ell}T_\ell(M;i,j;a)\\
H_\ell\colon \D^b(X)^k&\to \D_{\sym_2}^b(X^2)\,,\,
(E_1^\bullet,\dots,E_k^\bullet)\mapsto \delta_*(E_1^\bullet\otimes^L\dots\otimes^LE_k^\bullet)\,,\\
\mathcal F\colon \D^b(X)^k&\to \D^b_{\sym_n}(X^n) \quad  (n\ge 3)\,,\,
(E_1^\bullet,E_2^\bullet,E_3^\bullet)\mapsto\left(E_1^\bullet\otimes^L E_2^\bullet\otimes^L E_3^\bullet  \right)_{13}\,,\\
\mathcal E\colon \D^b(X)^k&\to \D^b_{\sym_n}(X^n) \quad  (n\ge 3)\,,\,
(E_1^\bullet,E_2^\bullet,E_3^\bullet)\mapsto\left(E_1^\bullet\otimes^L E_2^\bullet\otimes^L E_3^\bullet  \right)_{123}\,,\\
\mathcal H\colon \D^b(X)^k&\to \D^b_{\sym_n}(X^n) \quad  (n\ge 3)\,,\,
(E_1^\bullet,E_2^\bullet,E_3^\bullet)\mapsto\left(\Omega_X\otimes E_1^\bullet\otimes^L E_2^\bullet\otimes^L E_3^\bullet  \right)_{13}\\
\mathcal H_{123}\colon \D^b(X)^k&\to \D^b_{\sym_n}(X^n) \quad  (n\ge 3)\,,\,
(E_1^\bullet,E_2^\bullet,E_3^\bullet)\mapsto\left(\Omega_X\otimes E_1^\bullet\otimes^L E_2^\bullet\otimes^L E_3^\bullet  \right)_{123}\\
\end{align*}
The empty derived tensor product has to be interpreted as the sheaf $\reg_X$.
\end{defin}
The functor $(\_)_I$ is the composition of the pull-back along the flat morphism $p_I$ and the closed embedding $\iota_I$. Thus, it is exact.
The only derived functor occurring in the above multi-functors is the tensor product on $X$. Thus the images under the functors can be computed by replacing $E_1^\bullet,\dots,E_k^\bullet$ by locally free resolutions.   
In particular, for $E_1,\dots,E_k$ locally free sheaves on $X$ the functors $K_0$, $T_\ell$, $H_\ell$, $\E$, and $\mathcal H$ coincide with the functors defined before and $\mathcal H_{123}$ coincides with $\mathcal H_{|\Delta_{123}}$.
Again we will often drop the arguments of the functors in the notations. We also define as in subsection \ref{threetens} the functors $T_1(1)$ and $T_1(3)$ as the direct sum of the $T_1(M;i,j;a)$ in the case $k=3$ over $\tilde J^1$ respectively $\tilde J^3$. Again $T_1(1)^{\sym_{\overline{\{1,3\}}}}\simeq (\F^{\sym_{\overline{\{1,3\}}}})^{\oplus 3}$.
\subsection{Generalised results}\label{gen}
\begin{theorem}\label{maindescnew}
Let $E_1,\dots, E_k$ be locally free sheaves and $L$ a line bundle on $X$. Then on $X^n$ there is the equality
\[K_{k}\otimes L^{\boxtimes n} =p_*q^*(E_1^{[n]}\otimes \dots \otimes E_k^{[n]}\otimes \mathcal D_L)\]
of subsheaves of $K_{0}\otimes L^{\boxtimes n}$.  Also, for every $\ell=1,\dots,k$ we have
$K_\ell\otimes L^{\boxtimes n}=\ker(\phi_\ell\otimes \id_{L^{\boxtimes n}})$.
\end{theorem}
\begin{proof}
Use theorem \ref{maindesc} and lemma \ref{det}. The second assertion is due to the fact that tensoring with the line bundle $L^{\boxtimes n}$ is exact.
\end{proof}
\begin{prop}\label{maxcohtwi}
Let $E_1,\dots, E_k$ be locally free sheaves and $L$ a line bundle on $X$. Then
\begin{align*}\Ho^{2n}(X^{[n]}, E_1^{[n]}\otimes \dots\otimes E_k^{[n]}\otimes \mathcal D_L)&\cong \Ho^{2n}(X^n,K_0\otimes L^{\boxtimes n})^{\sym_n}\\&\cong \bigoplus_{a\in J_0}\bigotimes_{r=1}^{\max a}\Ho^2(\bigotimes_{t\in a^{-1}(r)} E_t\otimes L)\otimes S^{n-\max a}\Ho^2(L)\,. \end{align*}
\end{prop}
\begin{proof}
This follows from theorem \ref{maindescnew} the same way proposition \ref{maxcoh} followed from theorem \ref{maindesc}.
\end{proof}
\begin{theorem}\label{longtwodet}
Let $E_1,\dots, E_k$ be locally free sheaves and $L$ a line bundle on $X$. Then on $X^2$ for every $\ell=1,\dots,k$ the sequences
\[0\to K_\ell\otimes L^{\boxtimes n}\to K_{\ell-1}\otimes L^{\boxtimes n}\to T_\ell^0\otimes L^{\boxtimes n} \to T_\ell^1\otimes L^{\boxtimes n}\to\dots\to T_\ell^{k-\ell}\otimes L^{\boxtimes n}\to 0\]
are exact.
\end{theorem}
\begin{proof}
 We tensorize the exact sequences of \ref{longtwo} with the line bundle $L^{\boxtimes n}$.
\end{proof}
\begin{prop}
On $S^2X$ there are for $\ell=1,\dots,k$ exact sequences
\[0\to (K_\ell\otimes L^{\boxtimes n})^{\sym_2}\to (K_{\ell-1}\otimes L^{\boxtimes n}) ^{\sym_2}\to \pi_*(H_\ell\otimes L^{\boxtimes n})^{\oplus N(k,\ell)}\to 0\,.\]
In particular $(K_k\otimes L^{\boxtimes n})^{\sym_2}=(K_{k-1}\otimes L^{\boxtimes n})^{\sym_2}$. 
\end{prop}
\begin{proof}
We tensorise the exact sequences of proposition \ref{symtwoexact} with the line  bundle $(L^{\boxtimes 2})^{\sym_2}$ and use lemma \ref{det} (ii). 
\end{proof}
\begin{lemma}
For $E_1,\dots,E_k$ locally free sheaves and $L$ a line bundle on $X$ there is in the Grothendieck group $\K(S^2X)$ the equality
\[\mu_!\bigl[(E_1^{[2]}\otimes\dots\otimes E_k^{[2]}\otimes \mathcal D_L)\bigr]=\bigl[(K_0\otimes L^{\boxtimes n})^{\sym_2}\bigr]-\sum_{\ell=1}^k N(k,\ell)\bigl[(H_\ell\otimes L^{\boxtimes n})^{\sym_2} \bigr]\,.\]   
\end{lemma}
\begin{proof}
 We multiply both sides of the equation in corollary \ref{Ktwofree} by $[(L^{\boxtimes n})^{\sym_n}]$. Then we apply \ref{det} (iii) to the left-hand side and \ref{det} (ii) to the right-hand side.
\end{proof}
\begin{prop}\label{Kgen}
 For $E_1^\bullet,\dots,E_k^\bullet\in\D^b(X)$ and $L$ a line bundle on $X$ there is in the Grothendieck group $\K(S^2X)$ the equality
\[\mu_!\bigl[((E_1^\bullet)^{[2]}\otimes\dots\otimes (E_k^\bullet)^{[2]}\otimes \mathcal D_L)\bigr]=\bigl[(K_0\otimes L^{\boxtimes n})^{\sym_2}\bigr]-\sum_{\ell=1}^k N(k,\ell)\bigl[(H_\ell\otimes L^{\boxtimes n})^{\sym_2} \bigr]\,.\]   
\end{prop}
\begin{proof}
We denote by $A$ the left-hand side of the equation and by $B$ the right hand side, both as functions in $E_1^\bullet,\dots,E_k^\bullet$.
We choose locally free resolutions $F_i^\bullet$ of $E_i^\bullet$ for $i\in[k]$. Since the functor $(\_)^{[n]}$ is exact and $R\mu_*$ coincides with $\mu_*$ on tensor products of tautological bundles and determinant line bundles by corollary \ref{zeroconc} we have for the left-hand side
\[A((E_i^\bullet)_i)=A((F_i^\bullet)_i)=\sum_{q\in \Z}(-1)^q\left(\sum_{p_1+\dots+p_k=q} A(F_1^{q_1},\dots,F_k^{q_k})\right)\,.\]
Since the derived functors occuring on the right-hand side coincide with their non-derived versions on locally free sheaves, we get also
\[B((E_i^\bullet)_i)=B((F_i^\bullet)_i)=\sum_{q\in \Z}(-1)^q\left(\sum_{p_1+\dots+p_k=q} B(F_1^{q_1},\dots,F_k^{q_k})\right)\]
and the assertion follows by the previous lemma.
\end{proof}
\begin{theorem}
Let $X$ be a smooth projective surface. Let $E_1^\bullet,\dots,E_k^\bullet\in \D^b(X)$ and $L$ a line bundle on $X$. Then there is the formula
\begin{align*}&\chi_{X^{[2]}}\left((E_1^\bullet)^{[2]}\otimes^L \dots\otimes^L (E_k^\bullet)^{[2]}\otimes^L \mathcal D_L \right)\\
=&\sum_{1\in P\in[k]}\chi\left(\otimes_{t\in P}^L E_t^\bullet\otimes L\right)\cdot \chi\left(\otimes_{t\in [k]\setminus P}^L E_t^\bullet\otimes L\right)-\sum_{\ell=1}^{k-1}N(k,\ell)\cdot\chi\left(\otimes_{t\in [k]}^LE_t^\bullet\otimes S^{\ell-1}\Omega_X\otimes L^{\otimes 2}\right)\,.\end{align*}
\end{theorem}
\begin{proof}
This follows from the previous proposition the same way proposition \ref{Eulertwo} followed from corollary \ref{Ktwofree}. 
\end{proof}
Similiarly, one can generalise the formula for the Euler bicharacteristic of subsection \ref{bichar}.
\begin{theorem}
 Let $E_1,\dots,E_k$ be locally free sheaves and $L$ a line bundle on $X$ and let $n\ge 3$. Then the sheaf 
\[\mu_*(E_1^{[n]}\otimes\dots\otimes E_k^{[n]}\otimes \mathcal D_L)\cong (K_2\otimes L^{\boxtimes n})^{\sym_n}\]
is given by the following exact sequences
\begin{align*}
 0\to \ker(\phi_1^{\sym_n}(1))\otimes (L^{\boxtimes n})^{\sym_n}\to (K_0\otimes L^{\boxtimes n})^{\sym_n}\to  (T_1(1) \otimes L^{\boxtimes n})^{\sym_{\overline{\{1,3\}}}}
\to 0\,,\\
0\to (K_1\otimes L^{\boxtimes n})^{\sym_n}\to \ker(\phi_1^{\sym_n}(1))\otimes (L^{\boxtimes n})^{\sym_n}\to (T_1(3)\otimes L^{\boxtimes n})^{\sym_{[4,n]}}\to (\mathcal E^2\otimes L^{\boxtimes n})^{\sym_{[4,n]}}\to 0\,,\\
0\to (K_2\otimes L^{\boxtimes n})^{\sym_n}\to (K_1\otimes L^{\boxtimes n})^{\sym_n}\to (\mathcal H\otimes L^{\boxtimes n})^{\sym_{\overline{\{1,3\}}}}\to (\mathcal H_{|\Delta_{123}}\otimes L^{\boxtimes n})^{\sym_{[4,n]}}\to 0\,.
\end{align*}
\end{theorem}
\begin{proof}
The isomorphism $\mu_*(E_1^{[n]}\otimes\dots\otimes E_k^{[n]}\otimes \mathcal D_L)\cong (K_2\otimes L^{\boxtimes n})^{\sym_n}$ follows by \ref{det} (iii) and \ref{kvanish}. We get the exactness of the sequences by tensoring the exact sequences of \ref{onlyone}, \ref{kone}, and \ref{ktwo} by $(L^{\boxtimes n})^{\sym_n}$ and using lemma \ref{det} (ii).
\end{proof}
\begin{prop}\label{Kgenthree}
Let $E_1^\bullet,E_2^\bullet,E_3^\bullet\in \D^b(X)$ and $L$ a line bundle on $X$.
Then for $n\ge 3$ there is in the Grothendieck group $\K(S^nX)$ the equality
\begin{align*}&\mu_!\bigl[(E_1^\bullet)^{[n]}\otimes^L (E_2^\bullet)^{[n]}\otimes^L (E_3^\bullet)^{[n]}\otimes \mathcal D_L\bigr]\\
=&\bigl[(K_0\otimes L^{\boxtimes n})^{\sym_n}\bigr] -3\bigl[(\mathcal F\otimes L^{\boxtimes n})^{\sym_{\overline{\{1,3\}}}} \bigr] -\bigl[
(T_1(3)\otimes L^{\boxtimes n})^{\sym_{[4,n]}}  \bigr]\\&+2\bigl[(\mathcal E\otimes L^{\boxtimes n})^{\sym_{[4,n]}}  \bigr]- \bigl[(\mathcal H\otimes L^{\boxtimes n})^{\sym_{\overline{\{1,3\}}}}  \bigr]+ 
\bigl[(\mathcal H_{123}\otimes L^{\boxtimes n})^{\sym_{[4,n]}}  \bigr]
\,.\end{align*}  
\end{prop}
\begin{proof}
 This follows from \ref{Kthree} the same way \ref{Kgen} followed from \ref{Ktwofree}. 
\end{proof}
We use the notation of \ref{chi}, namely
\[s^{m}\chi(L):= \chi((L^{\boxtimes m})^{\sym_m})=\chi(S^m\Ho^*(L))=\binom{\chi(L)+m-1}{m}\,.\]
\begin{theorem}
Let $E_1^\bullet,E_2^\bullet,E_3^\bullet\in \D^b(X)$ and $L$ a line bundle on a smooth projective surface $X$.
Then for $n\ge 3$ there is the formula 
\begin{align*}
&\chi((E_1^\bullet)^{[n]}\otimes^L (E_2^\bullet)^{[n]}\otimes^L (E_3^\bullet)^{[n]}\otimes \mathcal D_L)\\
=&\chi(E_1^\bullet\otimes L)\chi(E_2^\bullet\otimes L)\chi(E_3^\bullet\otimes L)s^{n-3}\chi(L) \\
 &+\bigl(\sum_I\chi(E_a^\bullet\otimes^L E_b^\bullet\otimes L)\chi(E_c^\bullet\otimes L) \bigr)\cdot s^{n-2}\chi(L) \\
&-\bigl(\sum_I\chi(E_a^\bullet\otimes^L E_b^\bullet\otimes L\otimes L)\chi(E_c^\bullet\otimes L) \bigr)\cdot s^{n-3}\chi(L) \\ 
&+\chi(E_1^\bullet\otimes^L E_2^\bullet\otimes^L E_3^\bullet\otimes L)s^{n-1}\chi(L)\\
&-3\chi(E_1^\bullet\otimes^L E_2^\bullet\otimes^L E_3^\bullet\otimes L\otimes L)s^{n-2}\chi(L)\\
&+2\chi(E_1^\bullet\otimes^L E_2^\bullet\otimes^L E_3^\bullet\otimes L\otimes L\otimes L)s^{n-3}\chi(L)\\ 
&-\chi(\Omega_X\otimes E_1^\bullet\otimes^L E_2^\bullet\otimes^L E_3^\bullet\otimes L\otimes L)s^{n-2}\chi(L) \\
&+\chi(\Omega_X\otimes E_1^\bullet\otimes^L E_2^\bullet\otimes^L E_3^\bullet\otimes L\otimes L\otimes L)s^{n-3}\chi(L)\,.
\end{align*}
Here $I$ denotes the index set
\[I=\{(a=1,b=2,c=3),(a=1,b=3,c=2),(a=2,b=3,c=1)\}\,.\]
\end{theorem}
\begin{proof}
This follows from the previous proposition the same way \ref{Eulerthree} followed from \ref{Kthree}.
\end{proof}

%% file: binomial.tex
\section{Combinatorical notations and preliminaries}\label{combi}
\subsection{Partial diagonals}\label{pardia} 
Let $X$ be a variety and $n\in \N$. We define for $I\subset [n]$ with $|I|\ge 2$ the \textit{$I$-th partial diagonal} as the reduced closed subvariety given by
\[\Delta_I=\{(x_1,\dots,x_n)\in X^n\mid x_i=x_j\,\forall\, i,j\in I\}\,.\]
We denote by $\pr_i\colon X^n\to X$ the projection on the $i$-th factor and by $p_I\colon \Delta_I\to X$ the projection induced by $\pr_i$ for any $i\in I$. We denote the inclusion of the partial diagonals into the product by $\iota_I\colon \Delta_I\to X^n$. For a coherent sheaf $F$ on $X$ we set
\[F_I:=\iota_{I*}p_I^*F\,.\]
The functor $(\_)_I\colon \Coh(X)\to \Coh(X^n)$ is an exact functor since it is a pull-back along a projection followed by a push-forward along a closed embedding. We will sometimes drop the brackets $\{\_\}$ in the notations, e.g we will write
\[\Delta_{12}=\Delta_{1,2}=\Delta_{\{1,2\}}\quad,\quad F_{12}=F_{1,2}=F_{\{1,2\}}\,.  \]
Moreover, we denote the vanishing ideal sheaf of $\Delta_I$ in $X^n$ by $\I_I$.
If $X$ is non-singular we denote the normal bundle of the partial diagonal by $N_I:=N_{\Delta_I}=(\iota_I^*\I_I)^\vee=(\I_I/\I_I^2)^\vee$.
For $|I|\in\{0,1\}$ we set $\Delta_I=X^n$. The \textit{big diagonal} in $X^n$ is defined as 
$\mathbbm D:=\cup_{1\le i<j\le n}\Delta_{ij}$. The \textit{small diagonal} is $\Delta=\Delta_{[n]}$. 
\subsection{Signs}
Let $M$ be a finite set. There is the group homomorphism $\sgn\colon\sym_M\to \{-1,+1\}$ which is given after choosing a total order $<$ on $M$ by  
\[\sgn\sigma= (-1)^{\#\{(i,j)\in M\times M\mid i<j\,,\, \sigma(i)>\sigma(j)\}}\]
for $\sigma\in \sym_M$.
For two finite totally ordered sets $M$, $L$ of the same cardinality we define $u_{M\to L}$ as the unique strictly increasing map.
Let now $N$ be totally ordered, $m\in M\subset N$ and $\sigma\in\sym_N$. We define the signs
\begin{align*}\eps_{\sigma,M}:=\sgn(u_{\sigma(M)\to M}\circ \sigma_{|M})=(-1)^{\#\{(i,j)\in M\times M\mid i<j\,,\, \sigma(i)>\sigma(j)\}} 
\end{align*}
and $\eps_{m,M}:=(-1)^{\#\{j\in M\mid j<m\}}$.
\begin{lemma}\label{signlemma}
 \begin{enumerate}
  \item Let $L,M$ be two finite totally ordered sets. For $\sigma\in \sym_M$ we consider $\tilde \sigma:=u_{M\to L}\circ\sigma\circ u_{L\to M}\in\sym_L$. Then $\sgn \sigma=\sgn \tilde\sigma$.
\item Let $N$ be a finite set with a total order and $M\subset N$. Then 
\[\eps_{\mu,\sigma(M)}\cdot \eps_{\sigma,M}=\eps_{\mu\circ\sigma,M}\]
for all $\sigma,\mu\in \sym_N$.
\item Let $M\subset N$ be as above and $m\in M$. Then 
\[\eps_{\sigma^{-1}(m),\sigma^{-1}(M)}\cdot \eps_{\sigma,\sigma^{-1}(M)}=\eps_{m,M}\cdot \eps_{\sigma,\sigma^{-1}(M\setminus\{m\})}
 \]
holds for every $\sigma\in \sym_N$. 
 \end{enumerate}
\end{lemma}
\begin{proof}
 The map $u_{M\to L}\times u_{M\to L}$ gives a bijection
\[
\{(i,j)\in M\times M\mid i<j\,,\, \sigma(i)>\sigma(j)\} \cong \{(i,j)\in L\times L\mid i<j\,,\, \tilde\sigma(i)>\tilde\sigma(j)\}\,,
\]
which proves (i). For (ii) we have
\begin{align*}
 \eps_{\mu,\sigma(M)}\cdot\eps_{\sigma,M}&= \sgn(u_{\mu(\sigma(M))\to \sigma(M)}\circ \mu_{|\sigma(M)})\sgn(u_{\sigma(M)\to M}\circ \sigma_{|M})\\
&\overset{(i)}= \sgn(u_{\sigma(M)\to M}u_{\mu(\sigma(M))\to \sigma(M)} \mu_{|\sigma(M)}u_{M\to \sigma(M)}) \sgn(u_{\sigma(M)\to M}\circ \sigma_{|M})\\
&= \sgn(u_{\sigma(M)\to M}u_{\mu(\sigma(M))\to \sigma(M)} \mu_{|\sigma(M)} u_{M\to \sigma(M)}u_{\sigma(M)\to M} \sigma_{|M})\\
&= \sgn\left( u_{\mu(\sigma(M))\to M} (\mu\circ\sigma)_{|M}\right)\\
&= \eps_{\mu\circ\sigma,M}\,.
\end{align*}
For $m\in M\subset N$ and $\sigma\in \sym_N$ as in (iii) we have
\begin{align*}
&|\{j\in \sigma^{-1}(M)\mid j<\sigma^{-1}(m)\}|-|\{j\in M\mid j<m\}|\\=&|\{j\in \sigma^{-1}(M)\mid j<\sigma^{-1}(m),\sigma(j)>m\}|-|\{j\in \sigma^{-1}(M)\mid j>\sigma^{-1}(m),\sigma(j)<m\}|\,.
\end{align*}
This yields 
\[\frac{\eps_{\sigma^{-1}(m),\sigma^{-1}(M)}}{\eps_{m,M}}=\frac{\eps_{\sigma,\sigma^{-1}(M)}}{\eps_{\sigma,\sigma^{-1}(M\setminus \{m\})}}
=\frac{\eps_{\sigma,\sigma^{-1}(M\setminus \{m\})}}{\eps_{\sigma,\sigma^{-1}(M)}}\,.\]
\end{proof}
\subsection{Binomial coefficients}
For $n,k\in \Z$ with $k\ge 0$ the binomial coefficient is given by
\[\binom nk=\frac 1{k!}\cdot n(n-1)\cdots(n-k+1)\,.\]
In particular for $0\le n\le k$ it is zero.
\begin{lemma}\label{binom}
For all $k,\ell\in \N$ with $k\ge \ell$ we have
\[\sum_{i=0}^{k-\ell}(-1)^i2^{k-\ell-i}\binom k{\ell+i}\binom{\ell+i-1}{\ell-1}=\sum_{j=\ell}^k\binom kj\,.\] 
\end{lemma}
\begin{proof}
For $k=\ell$ both sides of the equation equal $1$. Inserting $k+1$ for $k$ we get by induction
\begin{align*}
&\sum _{i=0}^{k-\ell+1}(-1)^i2^{k-\ell-i+1}\binom{k+1}{\ell+i}\binom{\ell+i-1}{\ell-1}\\
=&\sum _{i=0}^{k-\ell+1}(-1)^i2^{k-\ell-i+1}\left(\binom{k}{\ell+i}+\binom k{\ell+i-1}\right)\binom{\ell+i-1}{\ell-1}\\
=&2\sum _{i=0}^{k-\ell}(-1)^i2^{k-\ell-i}\binom{k}{\ell+i}\binom{\ell+i-1}{\ell-1}\\&+\sum _{i=0}^{k-\ell+1}(-1)^i2^{k-\ell-i+1}\binom{k}{\ell+i-1}\binom{\ell+i-1}{\ell-1}\\
=&2\sum_{j=\ell}^k\binom kj + 2^{k-\ell+1}\binom k{\ell-1}\sum_{i=0}^{k-\ell+1}(-2)^{-i}\binom{k-\ell+1}i\\
=&2\sum_{j=\ell}^k\binom kj+ 2^{k-\ell+1}\binom k{\ell-1}(1-\frac 12)^{k-\ell+1}\\
=&2\sum_{j=\ell}^k\binom kj +\binom k{\ell-1}\\
=&\sum_{j=\ell}^{k+1}\left(\binom kj+\binom k{j-1}\right)\\
=&\sum_{j=\ell}^{k+1}\binom{k+1}j\,.
\end{align*}
\end{proof}
\begin{lemma}\label{symwedge}
For $\chi,m\in \Z$ with $m\ge 0$ there is the equation 
\[(-1)^m\binom{-\chi}m=\binom{\chi+m-1}m\,.\]  
\end{lemma}
\begin{proof}
Indeed we have
\[m!\cdot \binom{-\chi}m =(-\chi)\cdots (-\chi-m+1)=(-1)^m\chi\cdots(\chi+m-1)=(-1)^mm!\binom{\chi+m-1}m\,.\] 
\end{proof}

%% file: graded.tex
\section{Graded vector spaces and their Euler characteristics}\label{graded}
In this section all vector spaces will be vector spaces over a field $k$ of characteristic zero.
A \textit{graded vector space $V^*$} is a vector space together with a decomposition $V^*=\oplus_{i\in \Z}V^i$ where only
finitely many $V^i$ are non-zero. A \textit{super vector space $V^\pm$} is a vector space together with a decomposition
$V^\pm =V^+\oplus V^-$. There is a functor $G\colon\GVS\to \SVS$ from the category of graded vector spaces to the category
of super vector spaces given by $V^*\mapsto V^\pm$ with
\[V^+=\bigoplus_{i \text{ even}}V^i\quad,\quad V^-=\bigoplus_{i \text{ odd}} V^i\,.\]
It has a right quasi-inverse $F$ given by $V^\pm\mapsto V^*$ with $V^0=V^+$, $V^1=V^-$, and $V^i=0$ for $i\notin \{0,1\}$.
There is also a functor $\GVS\to \D^b(\Ve)$ to the bounded derived category of vector spaces given by $V^*\mapsto V^\bullet$,
where $V^\bullet$ is the complex with trivial differentials and whose term in degree $i$ is $V^i$. Since $\Ve$ is a
semi-simple category, this is an equivalence with quasi-inverse given by $C^\bullet\mapsto C^*$ with
$C^i=\mathcal H^i(C^\bullet)$ (see \cite[II. 2.3]{GM}). Thus, we can define the tensor product of graded vector spaces
as the tensor product of the associated complexes. This means that for two graded vector spaces $V^*$ and $W^*$ the
tensor product is given by $(V\otimes W)^i=\oplus_{p+q=i} V^p\otimes W^q$. For two super vector spaces $V^\pm$ and $W^\pm$
we define their tensor product by $V^\pm\otimes W^\pm:= G(F(V^\pm)\otimes F(W^\pm))$. More concretely, we have
$(V\otimes W)^+=(V^+\otimes W^+)\oplus (V^-\otimes W^-)$ and $(V\otimes W)^-=(V^+\otimes W^-)\oplus (V^-\otimes W^+)$.
As one can check, the tensor power $(C^\bullet)^{\otimes m}$ of a complex $C^\bullet\in \D^b(\Ve)$ becomes a
$\sym_m$-equivariant complex with the action given by
\[\sigma\cdot(u_1\otimes \dots\otimes u_m):=\eps_{\sigma,p_1,\dots,p_m}(u_{\sigma^{-1}(1)}\otimes \dots\otimes
u_{\sigma^{-1}(m)})\,.
\]
Here  the $p_i$ are the degrees of the $u_i$, i.e. $u_i\in C^{p_i}$ and the sign $\eps_{\sigma,p_1,\dots,p_m}$ is defined by
setting $\eps_{\tau,p_1,\dots,p_m}=(-1)^{p_i\cdot p_{i+1}}$ for the transposition $\tau=(i\,,\, i+1)$ and requiring it to be
a homomorphism in $\sigma$. We call $\eps_{\sigma,p_1,\dots,p_m}$ the \textit{cohomological sign}. Consequently, we define
the symmetric product $S^mC^\bullet$ respectively $S_mC^\bullet$ degree-wise as the coinvariants respectively the invariants
of $(C^\bullet)^{\otimes m}$ under this action. Since $k$ is of characteristic zero, there is the natural isomorphism
$S^mC^\bullet\to S_mC^\bullet$ given by
\[u_1\cdots u_m \mapsto \frac 1{m!}\sum_{\sigma\in \sym_m}\eps_{\sigma,p_1,\dots,p_m}(u_{\sigma^{-1}(1)}\otimes \dots\otimes
 u_{\sigma^{-1}(m)})\,.
\]
Again, we define the symmetric product of graded vector spaces as the symmetric product of the associated complex and for a
super vector space $V^\pm$ we set $S^mV^\pm :=G(S^m(F(V^\pm)))$. For a bounded complex $C^\bullet\in \D^b(\Ve_f)$ of finite vector spaces its \textit{Euler characteristic}
is given by
\[
 \chi(C^\bullet):=\chi(V^*):=\sum_{i\in \Z}(-1)^ic^i=\sum_{i\in \Z}(-1)^i\dim\mathcal H^i(C^\bullet)=\sum_{i\in \Z}
 (-1)^iv^i=v^+-v^-\,.
\]
Here we use the convention of \ref{not} (vii) and denote the graded vector space associated to $C^\bullet$ by $V^*$ and the
associated super vector space by $V^\pm$. The third equality makes sure that the Euler characteristic is invariant under
isomorphisms in $\D^b(\Ve_f)$. The Euler characteristic of a finite graded vector space is defined as the Euler characteristic
of the associated complex. Let $X$ be a complete $k$-scheme. For sheaves $E,F\in \Coh(X)$ or more generally objects
$E,F\in \D^b(X)$ we set $\chi(F):=\chi(\Ho^*(X,F))$ and $\chi(E,F):=\chi(\Ext^*(E,F))$. 
\begin{lemma}\label{Euler}
 Let $V^*$ and $W^*$ be graded vector spaces and $m\in \N$. Then there are the formulas
 \[ \chi(V^*\otimes W^*):=\chi(V^*)\cdot \chi(W^*)\quad,\quad
  \chi(S^mV^*)=\binom{\chi(V^*)+m-1}{m}\,.
 \]
\end{lemma}
\begin{proof}
 Using the description of the tensor product $V^\pm\otimes W^\pm$ in the previous subsection, the first formula follows from
 the equality
 \[(v^+w^+-v^-w^-)-(v^+w^-+v^-w^+)=(v^+-v^-)(w^+-w^-)\,.
  \]
For the second formula we first consider the case that $\chi:=\chi(V^*)\ge 0$, i.e. $v^+\ge v^-$. We can construct a short
exact sequence
\[0\to k^\chi\to V^+\xrightarrow d V^-\to 0\,.\]
This gives an isomorphism $k^\chi[0]\simeq V^\bullet$, where $V^\bullet$ is the complex concentrated in degree 0 and 1 given
by $d\colon V^+\to V^-$. Since the symmetric product of complexes preserves isomorphisms in $\D^b(\Ve_f)$, we get
\[\chi(S^mV^*)=\chi(S^mV^\bullet)=\chi(S^m(k^\chi[0]))=\dim S^mk^\chi=\binom {\chi+m-1}m\,.\]
In the case that $\chi<0$ we can construct a short exact sequence
\[0\to V^+\xrightarrow d V^-\to k^{-\chi}\to 0\]
which is the same as an isomorphism $V^\bullet\simeq k^{-\chi}[-1]$. The complex $S^m(k^{-\chi}[-1])$ equals
$(\wedge^mk^{-\chi})[-m]$. Thus,
\begin{align*}
 \chi(S^mV^*)=\chi(S^mV^\bullet)=\chi(S^m(k^{-\chi}[-1]))=(-1)^m\dim(\wedge^m k^{-\chi})&=(-1)^m\binom{-\chi}m\\
 &\overset{\ref{symwedge}}=\binom{\chi+m-1}m\,.
\end{align*}
\end{proof}

%% file: Danila.tex
\section{Danila's lemma}
The following lemma was used by Danila in \cite{Dan} in order to simplify the computation of invariants.
Let $G$ be a finite group acting transitively on a set $I$. Let $G$ also act on a scheme $X$. Let
$M$ be a $G$-sheaf on $X$ admitting a decomposition $M=\oplus_{i\in I} M_i$ such that for any $i\in I$ and $g\in G$ the linearization $\lambda$ restricted to $M_i$ is an isomorphism $\lambda_g\colon M_i\overset\cong\to g^*M_{g(i)}$. 
This means that $M\cong_G \Inf_{\Stab(i)}^G M_i$ (see e.g. \cite[Lemma 3.1]{Kru}). 
Then the $G$-linearization of $M$ restricts to a  $\Stab_G(i)$-linearization of $M_i$, which makes the projection $M\to M_i$ a $\Stab_G(i)$-equivariant morphism. 
\begin{lemma}\label{Dan}
Let $\pi\colon X\to Y$ be a $G$-invariant morphism of schemes.
Then for all $i\in I$ the projection $M\to M_i$ induces an isomorphism $(\pi_*M)^G\overset\cong\to (\pi_*M_i)^{\Stab_G(i)}$.
\end{lemma}
\begin{proof}
The inverse is given on local sections $m_i\in (\pi_*M_i)^{\Stab_G(i)}$ by $m_i\mapsto \oplus_{[g]\in G/\Stab_G(i)} g\cdot m_i$ with $g\cdot m_i\in M_{g(i)}$.
\end{proof}
\begin{remark}\label{morphismdanila}
Let $G$, $I$ and $M$ be as above, $ N=\oplus_{j\in J} N_j$ a second equivariant sheaf such that $G$ acts transitively on $J$ and such that $g\colon N_j\xrightarrow\cong g^* N_{g(j)}$
for all $j\in J$. Let $\phi\colon M\to N$ be an equivariant morphism with components  $\phi(i,j)\colon M_i\to N_j$. Then for fixed $i\in I$ and $j\in J$ the map $\phi^G$ under the isomorphisms $M^G\cong M_i^{\Stab(i)}$ and $N^G\cong N_j^{\Stab(j)}$ of lemma \ref{Dan} is given by (see also \cite[Appendix B]{Sca1})
\[\phi^G\colon M_i^{\Stab(i)}\to N_j^{\Stab(j)}\quad,\quad m\mapsto\sum_{[g]\in \Stab(i)\setminus G} \phi(g(i),j)(g\cdot m)\,.\]
\end{remark}
\begin{remark}\label{notrans}
Danila's lemma and the corollaries can also be used to simplify the computation of invariants if $G$ does not act transitively on $I$. In that case let $I_1,\dots,I_k$ be the $G$-orbits in $I$. Then $G$ acts transitively on $I_\ell$ for every $1\le\ell\le k$ and the lemma can be applied to every $M_{I_\ell}=\oplus_{i\in I_\ell}M_i$ instead of $M$. Choosing representatives $i_\ell\in I_\ell$ yields
\[M^G\cong \bigoplus_{\ell=1}^k M_{i_\ell}^{\Stab_G(i_\ell)}\,.\]  
\end{remark}
\begin{lemma}\label{tensinv}
 Let $k$ be a field of characteristic zero, $R$ a $k$-algebra, $G$ a finite group and $M$ an $R[G]$-module. Let $N$ be a $R$-module, i.e. a $R[G]$-module where $G$ is acting trivially. Then
\[(M\otimes_R^LN)^G=M^G\otimes^L_R N\,.\]
\end{lemma}
\begin{proof}
 See \cite[Lemma 1.7.1]{Sca1}.
\end{proof}
Also this lemma can be globalised to get an analogous result for $G$-sheaves.

%% file: normal.tex
\section{Normal varieties}
In this subsection let $M$ be a normal variety. For any open subvariety $\iota\colon U\hookrightarrow M$ with 
$\codim(M\setminus U,M)\ge 2$ and $F\in \Coh(X)$ a locally free sheaf $u_F\colon F\xrightarrow \cong\iota_*\iota^* F$ is an isomorphism. Here $u_F$ is the unit of the adjunction $(\iota^*,\iota_*)$, i.e. the morphism given by restriction of the sections.
\begin{lemma}\label{codimpush} Let $N$ be another normal variety and $f\colon N\to M$ be a proper morphism such that also $\codim(N\setminus f^{-1}(U),N)\ge 2$ holds.
 Then there is a natural isomorphism 
\[\iota_*\iota^*f_*E\cong f_*E\]
for every locally free sheaf $E$ on $N$.
\end{lemma}
\begin{proof}
Due to the flat base change 
\[
 \begin{CD}
f^{-1}(U)
@>{\iota'}>>
N \\
@V{f'}VV
@VV{f}V \\
U
@>>{\iota}>
M
\end{CD}
\]
we get indeed
\[
 \iota_*\iota^*f_*E\cong \iota_*f'_*\iota'^* E\cong f_*\iota_*'\iota'^* E\cong f_* E\,.
\]
\end{proof}
\begin{lemma}\label{codimseq}
Let $\iota\colon U\to X$ be any open immersion and let \[0\to F'\to F\to F''\] be an exact sequence in $\Coh(X)$ such that $u_F$ and $u_{F''}$ are isomorphisms. 
Then $u_{F'}$ is also an isomorphism. 
\end{lemma}
\begin{proof}
 The functor $\iota_*\iota^*$ is left-exact. Therefore, there is the following diagram with exact horizontal sequences:
\[
 \begin{CD}
0@>>> F'
@>>>
F @>>> F'' \\
@.
@VVV @V{\cong}VV @V\cong VV\\
0
@>>>
\iota_*\iota^* F'@>>> \iota_*\iota^* F @>>> \iota_*\iota^* F''\,. 
\end{CD}
\]
Since the last two vertical maps are isomorphisms, the first one is also.
\end{proof}
\begin{lemma}\label{normalhom}
 Let $M$ be a normal variety and $U\subset M$ an open subvariety such that $\codim(M\setminus U,X)\ge 2$. Given two locally free sheaves $F$ and $G$ on $M$ and 
a subsheaf $E\subset F$ with $E_{|U}=F_{|U}$ the maps $a\colon\Hom(F,G)\to \Hom(E,G)$ and $\hat a\colon\sHom(F,G)\to \sHom(E,G)$, given by restricting the domain of the morphisms, are isomorphism.
\end{lemma}
\begin{proof}
It suffices to show that $a$ is an isomorphism. Since every open $V\subset M$ is again a normal variety with $V\setminus(U\cap V)$ of codimension at least 2, it will follow by considering an open affine covering that 
$\hat a$ is also an isomorphism.   
We construct the inverse $b\colon\Hom(E,G)\to \Hom(F,G)$ of $a$. For a morphism $\phi\colon E\to G$ the morphism $b(\phi)$ sends $s\in F(V)$ to the unique section in $G(V)$ which restricts to $\phi(s_{|V\cap U})\in G(V\cap U)$.   
\end{proof}

%% file: restriction.tex
\section{Restriction of local sections to closed subvarieties}\label{subres}
\begin{lemma}\label{opencover}
 Let $X$ be a quasi-projective variety. Then $X^n$ has an open covering consisting of subsets of the form $U^n$ for $U\subset X$ open and affine. 
\end{lemma}
\begin{proof}
 See \cite[Lemma 1.4.3]{Sca1}.
\end{proof}

\begin{lemma}\label{closedres}
Let $X$ be a scheme, $Z_1,Z_2\subset X$ two closed subschemes and $Z_1\cap Z_2$ their scheme-theoretic intersection. Then for every locally free sheaf $F$ on $X$ the following sequence is exact 
\begin{align*}F&\to F_{|Z_1}\oplus F_{|Z_2}\to F_{|Z_1\cap Z_2}\to 0\\
             s&\mapsto\binom{s_{|Z_1}}{s_{|Z_2}}\,,\,\binom ab\mapsto a_{|Z_1\cap Z_2}- b_{|Z_2\cap Z_2}\,.  
\end{align*}
\end{lemma}
\begin{proof}
Since the question is local, we can assume that $F$ is a trivial vector bundle. Since in this case the restriction of sections is defined component-wise, we can assume that $F=\reg_X$ is the trivial line bundle. Furthermore we can assume that $X=\Spec A$ is affine. Now the assertion follows from the fact that for two ideals $I,J\subset A$ the sequence 
\[A\to A/I\oplus A/J\to A/(I+J)\to 0\]
is exact. 
\end{proof}
Let $I\subset [n]$ and $i\in [n]\setminus I$. Then the closed embedding $\iota\colon \Delta_{I\cup \{i\}}\hookrightarrow \Delta_I$ induces by the universal property of the $\sym_{\overline{I\cup\{i\}}}$-quotient $\Delta_{I\cup \{i\}}\times S^{\overline{I\cup\{i\}}}X$ the commutative diagram 
\[ \begin{CD}
\Delta_{I\cup\{i\}}
@>{\iota}>>
\Delta_{I} \\
@V{\pi_{\overline{I\cup\{i\}}}}VV
@V{\pi_{\bar I}}VV \\
\Delta_{I\cup\{i\}}\times S^{\overline{I\cup\{i\}}}X
@>{\bar \iota}>>
\Delta_I\times S^{\overline{I}}X.
\end{CD}
\]
We also allow the case that $I=\{j\}$ consists of only one element. Then $\Delta_I=X^n$ and $\Delta_I\times S^{\overline I}X= X^I\times S^{\overline I}X\cong X\times S^{n-1} X$.   
\begin{lemma}\label{stillclosed}
 The morphism $\bar \iota$ is again a closed embedding.
\end{lemma}
\begin{proof}
Let $|I|=\ell$. We can make the identifications $\Delta_{I\cup\{i\}}\times S^{\overline{I\cup\{i\}}}X\cong X\times S^{n-\ell-1}$ and
$\Delta_I\times S^{\overline{I\cup\{i\}}}X\cong X\times S^{n-\ell}$. 
Under these identifications the map $\bar\iota$ on the level of points is given by $(x,\Sigma)\mapsto (x,x+\Sigma)$. Hence, it is injective. 
Since the quotient morphism $\pi_{\bar I}$ is finite and hence proper, $\im (\bar \iota)=\pi_{\bar I}(\Delta_{ I\cup\{i\}})$ is a closed subset.
It remains to show that the 
morphism of sheaves $\reg_{\Delta_I}^{\sym_{\overline I}}\to \reg_{\Delta\cup\{i\}}^{\sym_{\overline{I\cup\{i\}}}}$ is still surjective. This can be seen by  covering $X^n$ by open affines of the form $U^n$ (see lemma \ref{opencover}) and using the lemmas 5.1.1. and 5.1.2 of \cite{Sca1} in the case $F=\reg_X$ and $*=0$. 
\end{proof}
\begin{defin}
Let a finite group $G$ act on a scheme $X$. A locally free $G$-sheaf $F$ on $X$ of rank $r$ is called \textit{locally trivial} (as a $G$-sheaf) if there is a  cover over $X$ consisting of $G$-invariant open subsets $U\subset X$ such that $F_{|U}\cong_G \reg_U^{\oplus r}$.    
\end{defin}
\begin{remark}
 Let $F$ be a locally free sheaf on $X$ and $\emptyset\neq I\subset [n]$. Then $F_I$ (see section \ref{not} (xi)) is locally trivial as a $\overline{\sym_I}$-sheaf on $\Delta_I$.
\end{remark}
\begin{cor}\label{invsur}
 Let $\F$ be a locally trivial $\sym_{\bar I}$-sheaf on $\Delta_I$ and $r\colon \F\to \F_{|\Delta_{I\cup \{i\}}}$ the morphism given by restriction of sections.
Then the induced morphism $\underline r\colon \F^{\sym_{\bar I}}\to \F^{\sym_{\overline{I\cup\{i\}}}}_{|\Delta_{I\cup\{i\}}}$ is still surjective. 
\end{cor}
\begin{proof}
The morphism is given by restricting sections of the locally free sheaf $\mathcal F^{\sym_{\bar I}}$ along the closed embedding $\bar \iota$ (see also remark \ref{partialquotient}). 
\end{proof}